\documentclass[pdflatex,sn-mathphys]{sn-jnl}% Math and Physical Sciences Reference Style
\usepackage{float}
\usepackage{etoolbox}
\makeatletter
\patchcmd{\ps@headings}
{\hbox to \hsize{\hfill Springer Nature 2021 \LaTeX\ template\hfill}}
{\hbox to \hsize{}}
{}
{}
\patchcmd{\ps@headings}
{\hbox to \hsize{\hfill Springer Nature 2021 \LaTeX\ template\hfill}}
{\hbox to \hsize{}}
{}
{}
\patchcmd{\ps@titlepage}
{\hbox to \hsize{\hfill Springer Nature 2021 \LaTeX\ template\hfill}}
{\hbox to \hsize{}}
{}
{}
\makeatother
\jyear{2023}%

\usepackage{amsthm}
\usepackage{amsmath}
\usepackage{amssymb}
\usepackage{mathtools}
\usepackage{color}
\usepackage{overpic}
\usepackage{rotating}
\usepackage{microtype}
\usepackage{sidecap}
\usepackage{subcaption}
\usepackage{contour}
\usepackage{url}
\contourlength{1pt}

\usepackage{contour}
\contourlength{1pt}

\newcommand{\mlput}[1]{\hbox to 0pt{\hss{#1}}}
\newcommand{\mrput}[1]{\hbox to 0pt{{#1}\hss}}
\newcommand{\mcput}[1]{\hbox to 0pt{\hss{#1}\hss}}

\newlength{\brevehght}

\def\be{\begin{equation}}
\def\ee{\end{equation}}
\newcommand{\ignore}[1]{}

\newcommand{\bluevar}[1]{\begingroup #1\endgroup}
\usepackage{comment}
\theoremstyle{theorem}
\newtheorem{theorem}{Theorem}
\newtheorem{prop}[theorem]{Proposition}

\newtheorem{lemma}[theorem]{Lemma}
\theoremstyle{definition}
\newtheorem{example}[theorem]{Example}
\theoremstyle{remark}
\newtheorem{remark}[theorem]{Remark}%[section]

\newcommand{\sign}{\mathrm{sgn}\,}

\newcommand{\mscomm}[1]{}
\newcommand{\khcomm}[1]{}
\newcommand{\HP}[1]{}
\begin{document}

\title{Surfaces of constant principal-curvatures ratio in isotropic geometry}


\author[]{\fnm{Khusrav} \sur{Yorov}}

\author*[]{\fnm{Mikhail} \sur{Skopenkov$^*$,}}
%\email{mikhail.skopenkov@gmail.com}
%%% FOR ARXIV:
\email{mikhail.skopenkov\,@\,gmail.com}

\author[]{\fnm{Helmut} \sur{Pottmann}}
%\email{helmut.pottmann@kaust.edu.sa}

\affil[]{\orgname{King Abdullah University of Science and Technology}, \orgaddress{\city{Thuwal}, \postcode{23955}, \country{Saudi Arabia}}}


\abstract{
We study surfaces with a constant ratio of principal curvatures in Euclidean and \bluevar{simply} isotropic geometries and characterize rotational, channel, ruled, helical, and translational surfaces of this kind under some technical restrictions (the latter two cases only in isotropic geometry). We use the interlacing of various methods of differential geometry, including line geometry and Lie sphere geometry, ordinary differential equations, and elementary algebraic geometry.
}

\maketitle

%%%%%%%%%%%%%%%%%%%%%%%%%%%%%%%%%%%%%%%%%%%%%%%%%%%%%%%%%%%%%%%%%%%

%%%% FOR ARXIV ONLY: %%%%%
%%%%%\vspace{-0.9cm}
\section*{Contents}
\vspace{-0.7cm}
\renewcommand{\contentsname}{}
\tableofcontents
%%%%%%%%%%%%%%%%%%%%%%%%%%

%\begin{document}

%\title{Surfaces of constant principal-curvatures ratio in isotropic geometry}
%\author{H. Pottmann, M. Skopenkov, Kh. Yorov}
%%\maketitle

%------------------------------------------------------------------
\section{Introduction} \label{sec-intro}
%-------------------------------------------------------------------

We study surfaces with a constant ratio of principal curvatures in Euclidean and \bluevar{simply} isotropic geometries and characterize rotational, channel, ruled, helical, and translational surfaces of this kind under some technical restrictions (the latter two cases only in isotropic geometry). 

 \begin{figure}[t]%[htbp]
% \begin{overpic}[width=0.72\textwidth]{Figures/pavilion1.png}
%     %\cput(50,68){\contour{white}{INSIDE-OUT Pavilion (Eike Schling, Denis Hitrec, Jonas Schikore).}}
% \end{overpic}
 \begin{overpic}[width=0.97\textwidth]{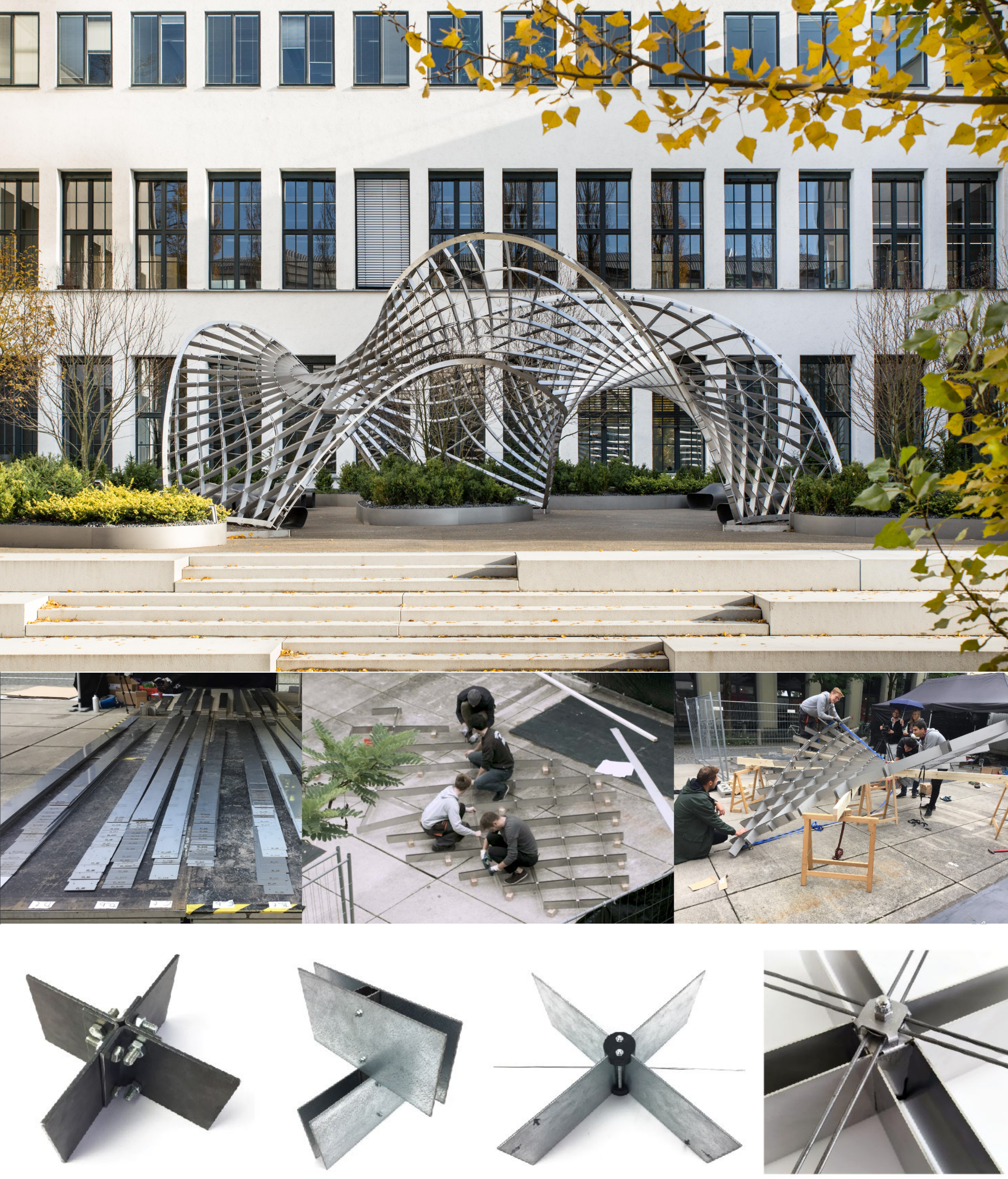}
    %\cput(50,68){\contour{white}{INSIDE-OUT Pavilion (Eike Schling, Denis Hitrec, Jonas Schikore).}}
\end{overpic}
% \begin{overpic}[width=1.0\textwidth]{Figures/stilljoints2.PNG}
%     %\cput(50,68){\contour{white}{Steel Joints.}}
% \end{overpic}
  \caption{\bluevar{An asymptotic gridshell (top) \cite{schling:2018}. %by E. Schling et al. % In this image, 
  Straight lamellas (middle) follow asymptotic curves of the reference surface %Thus, lamellas 
  and intersect under a right angle. This simplifies the manufacturing process as all steel joints (bottom) %will be 
  are identical but forces the reference surface %of the structure is an 
  to be a Euclidean minimal surface. If the intersection angle is constant and not right, then the joints are still identical, and we get more general surfaces: ones with %It is a special case of CRPC surfaces, occurring when the 
  a constant ratio of principal curvatures.
  %is $-1$. This uniformity in joint angles is also an advantage of CRPC surfaces, although the angle may differ.
    }}
    \label{fig:Munich}
\end{figure}

Surfaces with a constant ratio of principal curvatures, or briefly \emph{CRPC surfaces}, generalize minimal surfaces while keeping invariance under similarities. However, they are significantly harder to construct than minimal surfaces. CRPC surfaces are characterized geometrically as surfaces having a constant angle between characteristic curves (asymptotic curves in the case of negative Gaussian curvature; conjugate and principal symmetric curves
in case of positive Gaussian curvature). 

Recent interest in CRPC surfaces has its origin in architecture, in
particular in the aim of
building geometrically complex shapes from simple elements. A remarkable
class of such shapes is given by the asymptotic gridshells of E.~Schling \cite{schling:2018,schling-aag-2018}. \bluevar{See Fig.~\ref{fig:Munich}.} They are formed by bending originally
flat straight lamellas of bendable material (metal, timber) and
arranging them in a quadrilateral structure so that all strips are orthogonal
to some reference surface $S$. This requires the strips to follow
asymptotic curves on $S$. If, in addition, one aims at congruent nodes
to further simplify fabrication, one arrives at surfaces $S$ on
which asymptotic directions form a constant angle, i.e., at
negatively curved CRPC surfaces. Even positively curved CRPC
surfaces and other Weingarten surfaces are of interest in architecture,
since they only have a one-parameter family of curvature elements, 
which simplifies surface paneling of double-curved
architectural skins through mold re-use \cite{weingarten-2021}.

A classical general approach to complicated problems in Euclidean geometry is to start with their
simpler analogs in so-called \bluevar{simply} isotropic geometry. The isotropic analogs give a lot of geometric insight and also provide an initial guess for numerical optimization. This approach has been (implicitly) used since as early as the work \cite{Muntz} by M\"untz from 1911,  who solved the Plateau problem for Euclidean minimal surfaces in a quite general setup by deformation of graphs of harmonic functions. Such graphs are minimal surfaces in \bluevar{simply} isotropic geometry; thus in this case the optimization %has 
led to the whole existence proof.

%However, CRPC surfaces become a lot simpler for the more degenerated affine Cayley-Klein geometries. This is illustrated here for  
\emph{\bluevar{Simply} isotropic geometry}, \bluevar{also called just \emph{isotropic geometry} in the literature and the rest of this paper}, has been studied extensively
by K.~Strubecker (see, e.g., \cite{strubecker:1941,strubecker:1942a,strubecker:1942}) and is treated in the monograph by H.~Sachs \cite{sachs}. It is based on the group of affine transformations which preserve the \emph{isotropic semi-norm} $\|(x,y,z)\|_i:=\sqrt{x^2+y^2}$ in space with the coordinates $x,y,z$.
It can also be seen as relative differential geometry with respect to the \emph{unit isotropic sphere} (paraboloid of revolution) $2z=x^2+y^2$ \bluevar{(see~\cite{muller1921relative, simon1991introduction, dSL23}).}
%The relation between isotropic geometry and relative geometry % (see, e.g., \cite{muller1921relative, simon1991introduction}) 
%is discussed in \cite{muller1921relative, simon1991introduction, dSL23}.} 
Isotropic geometry is simpler but has much in common with Euclidean and other Cayley--Klein geometries.

The isotropic geometry of surfaces appears also in structural design and statics (see, e.g., \cite{millar:2023}), due to the close relation between the stresses in a planar body
and the isotropic curvatures of the associated Airy stress surface 
\cite{Strubecker_Airy}. CRPC surfaces in isotropic space
represent planar stress states with a constant ratio of principal stresses.

In our arguments, we use the interlacing of various methods of differential geometry%(envelopes, line geometry)
, ordinary differential equations, and elementary algebraic geometry (the latter --- in the classification of the translational CRPC surfaces, which we consider as our main contribution; see Section~\ref{sec-translational}).

%  \begin{figure}[htbp]
% \begin{overpic}[width=1.0\textwidth]{Figures/pavilion1.png}
%     %\cput(50,68){\contour{white}{INSIDE-OUT Pavilion (Eike Schling, Denis Hitrec, Jonas Schikore).}}
% \end{overpic}
% \begin{overpic}[width=1.0\textwidth]{Figures/stilljoints1.PNG}
%     %\cput(50,68){\contour{white}{Steel Joints.}}
% \end{overpic}
%   \caption{\bluevar{The Asymptotic Gridshell by E. Schling et al.(2017). \emph{Top:} Straight lamellas manually connected into flat segments, then transformed elastically by fixing each node to 90 degrees. \emph{Bottom:} 
% The standard grid joint includes two lamellas aligned parallelly in both directions. Star-shaped washers ensure a secure 90-degree angle, providing a central axis for the carriage bolt. Steel cables are paired and fastened with a cross-shaped clamp. }
%     }
%     \label{fig:Munich}
% \end{figure}

\subsection{Previous work}

\textbf{Euclidean Geometry.} Only a few explicit examples of CRPC surfaces, not being minimal surfaces, have been known before.  The explicit parameterizations were only available for rotational and helical CRPC surfaces \cite{hopf-1951,kuehnel-2013,mladenov+2003,mladenov+2007,lopez-pampano-2020,HelmutHui2022, helicalcrpc}. 

%\mscomm{Khusrav: there is no grammatic construction 'We call X Y' in English!} 
CRPC surfaces are a special case of so-called Weingarten surfaces. 
A \emph{Weingarten surface} is a surface with a fixed functional relation $f(H, K) = 0$ between the mean curvature $H$ and Gaussian curvature $K$ at each point (we assume that the zero set of the function $f$ is an analytic curve). 
A surface   is called \emph{linear Weingarten} if there is a fixed linear relation between the two principal curvatures $\kappa_1$ and $\kappa_2$ %(Euclidean or isotropic) 
at each point. Recently \bluevar{L{\'o}pez and P{\'a}mpano}~\cite{lopez-pampano-2020} have classified all rotational linear Weingarten surfaces, which are CRPC surfaces when the intercept of the relation is zero. 
Moreover, it has been shown that linear Weingarten surfaces are rotational if they are foliated by a family of circles \cite{lopez:2008}. Using quite involved computations, Havlíček \cite{Havlíček53} proved that channel Weingarten surfaces must be rotational or pipe surfaces; cf.~\cite{udo-channel}. %Hertrich-Jeromin et al.
In Section~\ref{sec-rotational} we give a geometric proof of this result.

In  works \cite{riveros+2012,riveros+2013,staeckel-1896}
rotational CRPC surfaces with $K<0$ have been characterized via isogonal asymptotic parameterizations. Yang et al.~\cite{helicalcrpc} have recently presented a characterization of all helical CRPC surfaces. 

The most well-studied Weingarten surfaces are %of course 
the ones with $K=\mathrm{const}$ or $H=\mathrm{const}$. Classification results for rotational, helical, and translational surfaces of this kind can be found in \cite[Chapter~26, p.~147--158]{strubecker-1969} and \cite{wunderlich_helical-minimal, KenmotsuKKatsuei2003Swcm, Liu1999, hasanis2018, hasanis2019, lopez2010}. Recently Udo et al.~\cite{Hertrich-Jeromin2023} obtained explicit parametrization for channel surfaces with $K=\mathrm{const}$ in the space forms. 
%\mscomm{MS: We should probably write at least something about the two classes $H=\mathrm{const}$ and $K=\mathrm{const}$ in Euclidean geometry and give references to the classification of rotational, helical, translational surfaces of this kind.}

CRPC surfaces, via a Christoffel-type transformation of certain spherical nets, were derived in \cite{jimenez+2019} with a focus on discrete models. In work \cite{HelmutHui2022}, we can find an effective method for the computation of discrete CRPC surfaces that provides insight into the shape variety of CRPC surfaces. Since numerical optimization was involved there, one
cannot derive precise mathematical conclusions, but it can be helpful for further studies.

\textbf{Isotropic Geometry.} %The CRPC surfaces are generalizations of the minimal surfaces. 
As far as we know, the examples of isotropic CRPC surfaces known before were either minimal (having isotropic mean curvature $H=0$) or paraboloids (having both $H=\mathrm{const}$ and the isotropic Gaussian curvature $K=\mathrm{const}$). However, there is a variety of related works regarding the conditions
$H=\mathrm{const}$ or $K=\mathrm{const}$ separately. Surfaces with $K=\mathrm{const}$ have received early attention as solutions of the
\bluevar{Monge-Amp{\'e}re} equation, but only within isotropic geometry their 
geometric constructions, e.g., as Clifford translational surfaces, are elegant and simple \cite{strubecker:1942a}. 
\bluevar{Invariant surfaces with $H=\mathrm{const}$ or $K=\mathrm{const}$, including the parabolic rotational surfaces, were studied in detail by da Silva \cite{dS21a}.}
An exact representation of several types of ruled surfaces with $H=\mathrm{const}$ or $K=\mathrm{const}$ 
%constant isotropic mean or Gaussian curvatures 
can be found in \cite{Alper2018, KaracanYoonYuksel+2017+87+98}.
All helical surfaces with $H=\mathrm{const}$ or $K=\mathrm{const}$ 
%constant isotropic mean or Gaussian curvatures 
were classified in \cite{Aydn2016ClassificationRO, Yoon2016}. Translational surfaces with $H=\mathrm{const}$ or $K=\mathrm{const}$
%constant isotropic mean or Gaussian curvatures 
were classified in \cite{Milin}, \bluevar{(generalizing the classical result for $H = 0$ from \cite{Strubecker+1977+minimal})} in the case when the generating curves are planar and in \cite{Aydin2020} in the case when one of the generating curves is spatial. However, the classification is still unknown when both generating curves are spatial.

%  \begin{figure}[htbp]
% \begin{overpic}[width=1.0\textwidth]{Figures/pavilion1.png}
%     %\cput(50,68){\contour{white}{INSIDE-OUT Pavilion (Eike Schling, Denis Hitrec, Jonas Schikore).}}
% \end{overpic}
% \begin{overpic}[width=1.0\textwidth]{Figures/stilljoints1.PNG}
%     %\cput(50,68){\contour{white}{Steel Joints.}}
% \end{overpic}
%   \caption{\bluevar{The Asymptotic Gridshell by E. Schling et al.(2017). \emph{Top:} Straight lamellas manually connected into flat segments, then transformed elastically by fixing each node to 90 degrees. \emph{Bottom:} 
% The standard grid joint includes two lamellas aligned parallelly in both directions. Star-shaped washers ensure a secure 90-degree angle, providing a central axis for the carriage bolt. Steel cables are paired and fastened with a cross-shaped clamp. }
%     }
%     \label{fig:Munich}
% \end{figure}

% \begin{figure}[htbp]
% \begin{center}
% \begin{overpic}[width=0.5\textwidth]{Figures/CRPC3.jpg}
% \end{overpic}
%    \caption{\bluevar{The measurement of the angle between characteristic curves in Euclidean (top) and isotropic (bottom) geometries.}
%     }
%     \label{fig:CRPC3}
% \end{center}
% \end{figure}
\section{Preliminaries} \label{sec-preliminaries}

%\mscomm{Basics of isotropic geometry and CRPC surfaces}

\subsection{Admissible surfaces and isotropic curvatures}

Recall that the \emph{isotropic semi-norm} in space with the coordinates $x,y,z$ is 
$\|(x,y,z)\|_i:=\sqrt{x^2+y^2}$.
%Three-dimensional isotropic geometry is based on the group of affine transformations whose projective extensions map a pair $(j,\bar{j})$ of conjugate complex lines in the ideal plane onto itself. We may embed this geometry in Euclidean geometry and use a Cartesian system $(x_1,x_2,x_3)$ where the pair of planes $x_1^2+x_2^2=0$ intersects the ideal plane in the absolute figure $(j,\bar{j})$. Affine maps whose projective extensions fix $(j,\bar{j})$ as a whole, are represented by
%
An %NO! \bluevar{orientation-preserving} 
affine transformation of $\mathbb{R}^3$ that scales the isotropic semi-norm by a %\bluevar{non-zero} 
constant factor %\bluevar{(and preserves the orientation in the projection to the $xy$-plane)} 
has the form
\begin{equation*}
  \mathbf{x}'= A\cdot \mathbf{x}+\mathbf{b}, \quad 
A=  \begin{pmatrix} \bluevar{\pm}h_1 & \bluevar{\mp}h_2 & 0  \\
 h_2 & h_1 & 0  \\
 c_1 & c_2 & c_3 \end{pmatrix} 
\end{equation*}
for some values of the parameters $\mathbf{b}\in\mathbb{R}^3$ and $h_1,h_2,c_1,c_2,c_3\in\mathbb{R}$. Such transformations form the 8-parametric group $\bluevar{\mathcal{G}^8}$ of \emph{general isotropic similarities}. The group of \emph{isotropic congruences}
is the 6-parametric subgroup %$G^6$ 
with 
$$ h_1=\cos \phi, \quad h_2=\sin \phi,\quad c_3=1. $$
These transformations appear as Euclidean congruences in the projection onto the
plane $z=0$, which we call \emph{top view}.
%$(x,y,z) \mapsto (x,y,0)$
Therefore, \emph{isotropic distances} between
points and \emph{isotropic angles} between lines appear in the top view as Euclidean distances and angles, respectively.

%Following Strubecker \cite{strubecker:1942}, we speak of isotropic space $I^3$ and perform differential geometry within $G^6$. 

Lines and planes \bluevar{that} are parallel to the $z$-axis are called \emph{isotropic} or \emph{vertical}. They play a special role and are usually excluded as tangent spaces in differential geometry. A point of a surface is \emph{admissible} if the tangent plane at the point is non-isotropic, and a surface is \emph{admissible} if it has only admissible points. Hereafter by a \emph{surface} we mean the image of a proper injective $C^3$ map of a closed planar domain into $\mathbb{R}^3$ with nondegenerate differential at each point, or, more generally, an embedded \bluevar{connected} $2$-dimensional $C^3$ submanifold of $\mathbb{R}^3$, possibly with boundary and possibly non-compact. 
%%% All the results a true for any of the two notions of a surface besides the ones where it is clear from the context that the more general definition is understood.
%By a \emph{smooth part} of the surface we mean a subset that is itself a surface.

An admissible surface can be locally represented as the graph of a function,
$$ z=f(x,y). $$
It is natural to measure the curvature of the surface in a given direction by a second-order quantity invariant under isotropic congruences and vanishing for a plane. 
Thus the \emph{isotropic normal curvature} in a tangent direction $t=(t_1,t_2,t_3)$ with $\|t\|_i=t_1^2+t_2^2=1$ is defined to be
the second directional
derivative of $f$, 
$$ \kappa_n(t)= (t_1,t_2) \cdot 
\begin{pmatrix}
    f_{xx} & f_{xy} \\
    f_{yx} & f_{yy}
\end{pmatrix}
\cdot 
\begin{pmatrix}
t_1 \\ t_2
\end{pmatrix},
%\nabla^2(f) \cdot (t_1,t_2)^{\mathrm{T}}, 
$$
and the \emph{isotropic shape operator} is defined to be the Hessian $\nabla^2(f)$ of $f$. Its
eigenvalues $\kappa_1$ and $\kappa_2$ are the \emph{isotropic principal curvatures} \bluevar{and the eigenvectors are the \emph{isotropic principal directions}. For $\kappa_1\ne\kappa_2$, the latter}  %Of course, they occur in directions (called \emph{isotropic principal directions}) that 
are orthogonal in the top view, and thus also orthogonal in the isotropic sense. %They are those tangents at a surface point, which are conjugate and orthogonal. Recall that two tangents are \emph{conjugate} if one is tangent to a curve on the surface, while the other one is a ruling of the envelope of the planes tangent to the surface at points on the curve. The \emph{characteristic directions} are the two conjugate tangent directions with the top views symmetric with respect to the eigenvectors of $\nabla^2(f)$. \emph{Isotropic principal curvature lines, asymptotic} and \emph{characteristic curves} etc are then defined analogously to the Euclidean case.  

The \emph{isotropic mean} and \emph{Gaussian curvatures} are defined respectively by
\begin{equation*}
  H:=\frac{\kappa_1+\kappa_2}{2}=\frac{f_{xx}+f_{yy}}{2}\qquad\text{and}\qquad K:=\kappa_1\kappa_2=f_{xx}f_{yy}-f_{xy}^2.
\end{equation*}

An admissible surface \emph{has a constant ratio \bluevar{$a$} %$a\ne0$ 
of isotropic principal curvatures}, or is a \emph{CRPC surface}, if \bluevar{$K\ne 0$, and} $\kappa_1/\kappa_2=a$ or $\kappa_2/\kappa_1=a$ at each point of the surface. The latter condition is equivalent to ${H^2}/{K} = (a+1)^2/(4a)$. 

In particular, for $a=-1$ we get \emph{isotropic minimal surfaces}, characterized by the condition $H=0$, i.e., the graphs of harmonic functions. As another example, for $a=1$ %\bluevar{(and $K\ne 0$)} 
we get 
a unique up to \bluevar{isotropic} similarity CRPC surface $2z = x^2+y^2$, also known as the \emph{isotropic unit sphere} \cite[Section~62, p.~402]{strubecker:1942}. 
%%%
%\bluevar{As another example, for $a=1$, we get totally umbilical surfaces ($H^2 - K = 0$). These surfaces are either a plane or, up to similarity, a unique CRPC surface $2z = x^2 + y^2$, known as the isotropic unit sphere \cite[Section~62, p.~402]{strubecker:1942}.}
%%%

\emph{Isotropic principal curvature lines, asymptotic \bluevar{ curves},} and \emph{isotropic characteristic curves}  are defined analogously to the Euclidean case as curves tangent to corresponding directions \bluevar{(we exclude the points where $\kappa_1=\kappa_2$). Recall that two directions $(t_1,t_2,t_3)$ and $(s_1,s_2,s_3)$ at a surface point are \emph{conjugate} if $f_{xx}t_1s_1+f_{xy}t_1s_2+f_{yx}t_2s_1+f_{yy}t_2s_2=0$. One can see that two directions} %Recall that two tangents at a surface point 
are conjugate if one is tangent to a curve on the surface, while the other one is a ruling of the envelope of the tangent planes at points on the curve \bluevar{(see, e.g., \cite[Section~60
]{K91}, \cite[Sections~2-10]{S88}). In particular,} the isotropic principal directions are the ones that are conjugate and orthogonal in the top view. For $K>0$, the \emph{isotropic characteristic directions} are the ones that are conjugate and symmetric with respect to the isotropic principal directions in the top view. For $K<0$, they coincide with the asymptotic directions, which are the same in Euclidean and isotropic geometry. 

For isotropic CRPC surfaces, the isotropic characteristic curves intersect 
%the isotropic principal curvature lines 
under the constant isotropic angle $\gamma$ 
%and $\pi/2-\gamma$, where 
with $\cot^2(\gamma/2) = \lvert a\rvert$. To see this, make the tangent plane at the intersection point horizontal by an appropriate isotropic congruence of the form $z\mapsto z+px+qy$ and apply a similar assertion in Euclidean geometry \cite[Section~2.1]{helicalcrpc}. All sufficiently smooth Euclidean or isotropic CRPC surfaces are analytic by the Petrowsky theorem \cite[p.3--4]{Petrowsky-39}. We sometimes restrict our results to the case of analytic surfaces, if this simplifies the proofs.

\begin{example} \label{ex-paraboloid} For any paraboloid with a vertical axis, or, equivalently, the graph of any quadratic function $f(x,y)$ with $\det\nabla^2(f)\ne 0$, both isotropic principal curvatures are constant. Hence their ratio $a$ is also constant. This surface can be brought to the paraboloid $z = x^2+ay^2$ by an appropriate general isotropic similarity. %, where $a$ is this ratio. 
(Technically, $1/a$ can also be considered as such a ratio of the same surface, but $z = x^2+ay^2$ is isotropic similar to $z = x^2+y^2/a$.)

The isotropic principal curvature lines of a non-rotational paraboloid $z = x^2+ay^2$, where $a\ne 0,1$, are parabolae   (parabolic isotropic circles, to be discussed below) in the isotropic planes $x=\mathrm{const}$ and $y=\mathrm{const}$. The rotational paraboloid $z = x^2+ay^2+(a-1)(x-x_0)^2$ \bluevar{is tangent to} the surface $z = x^2+ay^2$ along the isotropic principal curvature line $x=x_0$. Thus the surface is an envelope of a one-parameter family (actually, two families) of congruent rotational paraboloids with vertical axes (%also 
known as parabolic isotropic spheres). 

The characteristic curves of the paraboloid $z = x^2+ay^2$, where $a\neq 0,1$, appear in the top view as lines parallel to $x=\pm \sqrt{\lvert a\rvert}y$.
For a hyperbolic paraboloid ($a < 0$), these %characteristic
curves are the rulings. For an elliptic paraboloid ($a > 0$), they are parabolae forming a translational net on the surface. 

Thus a paraboloid with a vertical axis is both a translational (see Section~\ref{sec-translational}), a parabolic rotational (see Section~\ref{sec-rotational}), isotropic channel (see Section~\ref{sec-channel}), 
and, for $a<0$, a ruled surface (see Section~\ref{sec-ruled}). 
\end{example}

\subsection{Isotropic spheres and circles}

In isotropic geometry, there are two types of %isotropic 
spheres.  

The set of all points at the same isotropic distance $r$ from a fixed point $O$ is called a \emph{\bluevar{cylindrical} isotropic sphere}. In Euclidean terms, it can be visualized as a right circular cylinder with vertical rulings. Its top view appears as a Euclidean circle with the center at the top view of $O$ and the radius $r$. Any point %parallel to $O$ lies 
on the axis of this cylinder %and 
can serve as the center of the same isotropic sphere. 

An (inclusion-maximal) surface with both isotropic principal curvatures equal to a constant $A\ne 0$ 
is a \emph{parabolic isotropic sphere}. It has the equation 
$$
2z=A\left(x^2+y^2\right)+B x+C y+D, \qquad A \neq 0,
$$
for some $B,C,D\in\mathbb{R}$. Here ${1}/{A}$ is called 
the \emph{radius} of the isotropic sphere.
%In Euclidean terms, 
Such isotropic spheres are paraboloids of revolution with vertical axes. 

The intersection of an isotropic sphere $S$ with a non-tangential plane $P$ is an \emph{isotropic circle}. The isotropic circle is \emph{elliptic} if $P$ is non-isotropic, \emph{parabolic} if $S$ is parabolic and $P$ is isotropic, \emph{\bluevar{cylindrical}} if $S$ is \bluevar{cylindrical} and $P$ is isotropic. The resulting isotropic circle is an ellipse whose top view is a Euclidean circle, a parabola with a vertical axis, or a pair of vertical lines, respectively.

Recall that two %parameterized 
curves $\mathbf{x}(t)$ and $\mathbf{y}(t)$ have a \emph{second-order contact} for $t=0$, if $\mathbf{x}(0)=\mathbf{y}(0)$, $\mathbf{x}'(0)=\mathbf{y}'(0)$, and $\mathbf{x}''(0)=\mathbf{y}''(0)$. Two non-parameterized curves have a \emph{second-order contact} if some of their regular parametrizations do. The \emph{osculating isotropic circle} of a spatial curve at a non-inflection point is an isotropic circle having a second-order contact with the curve at the point. %It degenerates to a line at an inflection point of the curve.

There is an analog of Meusnier's theorem in isotropic geometry.

\begin{theorem} \label{l-Meusnier} \textup{(See \cite[Theorem 9.3]{sachs}, \cite[Section~47%, p.~394
]{strubecker:1942})}  
Let an admissible surface~$\Phi$ have isotropic normal curvature $\kappa_n\ne 0$ at a point $p \in \Phi$ along a surface tangent line $T$. 
%Consider an admissible surface~$\Phi$ and a point $p \in \Phi$ with a surface tangent line $T$. Let $\kappa_n\ne 0$ be the isotropic normal curvature of $\Phi$ at $p$ along $T$. 
Then the osculating isotropic circles of all curves on $\Phi$ \bluevar{that are tangent to} $T$ at $p$ lie on the parabolic isotropic sphere of radius $1/\kappa_n$ \bluevar{tangent to} $\Phi$ at $p$.
%Moreover, this sphere \bluevar{is tangent to} $\Phi^i$   at $P^i$.
\end{theorem}

\section{Rotational surfaces}\label{sec-rotational}

\subsection{Isotropic rotational CRPC surfaces}

Euclidean rotations about the $z$-axis are also %a normal form of isotropic rotations. 
isotropic congruences. A surface invariant under these rotations is called \emph{isotropic rotational}, as well as the image of the surface under any isotropic congruence. 

Looking for isotropic rotational CRPC surfaces, we %express the profiles as 
consider the graph of a smooth function $z=h(r)$ of the \emph{radial distance} $r:=\sqrt{x^2+y^2}$. Profiles $x/y=\mathrm{const}$ and parallel circles $r=\mathrm{const}$ give a principal parameterization in isotropic geometry,
due to the symmetries. The isotropic profile curvature $\kappa_2$
equals the 2nd derivative $h''(r)$, and $\kappa_1=h'(r)/r$ by
Meusnier's theorem (Theorem~\ref{l-Meusnier}). Hence
$\kappa_2/\kappa_1=a$ amounts to solutions of $ah'=rh''$. Up to isotropic similarities,
this yields the profile curves
\begin{equation}\label{eq-profile}
h(r)= \begin{cases}  r^{1+a}, & \text{if } a \ne -1; \\
                     \log r,  & \text{if } a = -1, 
      \end{cases} 
\end{equation}                       
and the ones with $a$ replaced by $1/a$. %\bluevar{These curves also have another remarkable property: they are the simply isotropic analogues of catenaries when $a = -1$, and $\alpha$-catenaries when $a \neq -1$, with $\alpha = -a$; see, e.g. \cite[Corollary~60]{dSL23}.} 
We have arrived at the following \bluevar{result}.
%proposition.

\begin{prop}\label{thm-rotational} (See Figure~\ref{fig:rotational})
An admissible isotropic rotational surface has a constant ratio $a\ne 0$ of isotropic principal curvatures if and only if it is isotropic similar to a subset of one of the surfaces 
\begin{align}\label{eq-rotational}
    z &= (x^2+y^2)^{(1+a)/2}\qquad\text{or}\qquad z  = (x^2+y^2)^{(1+a)/(2a)},
    & &\text{if }a\ne -1;    \\
    \label{eq-logarithmoid}
    z &= \log(x^2+y^2), 
    & &\text{if }a = -1.
\end{align}
\end{prop}

\begin{figure}[htbp]
\begin{overpic}[width=0.18\textwidth]{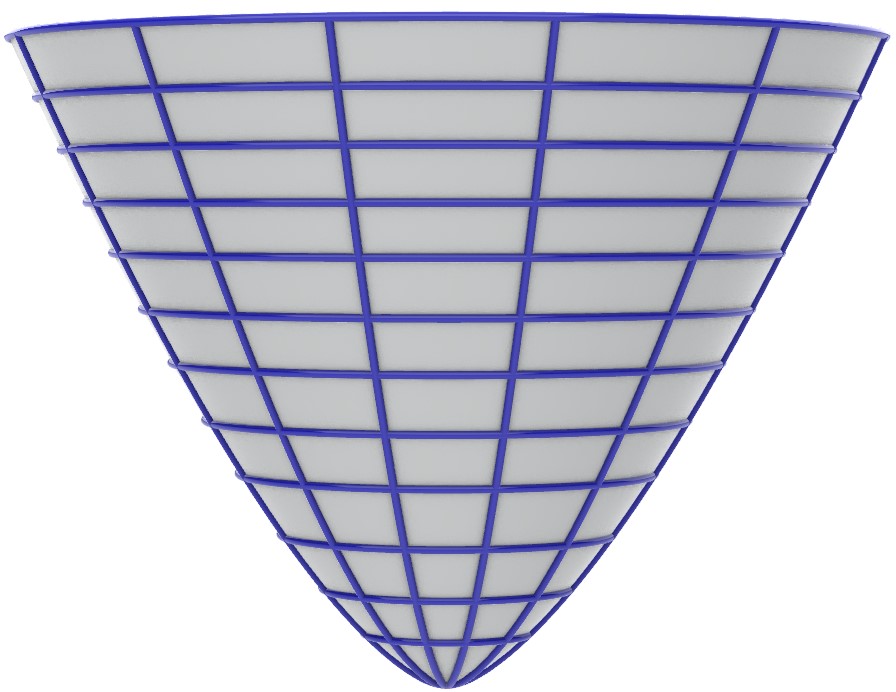}
    %\put(44,80){\contour{white}{$\widetilde{z}$}}
\end{overpic} 
\qquad
\begin{overpic}[width=0.30\textwidth]{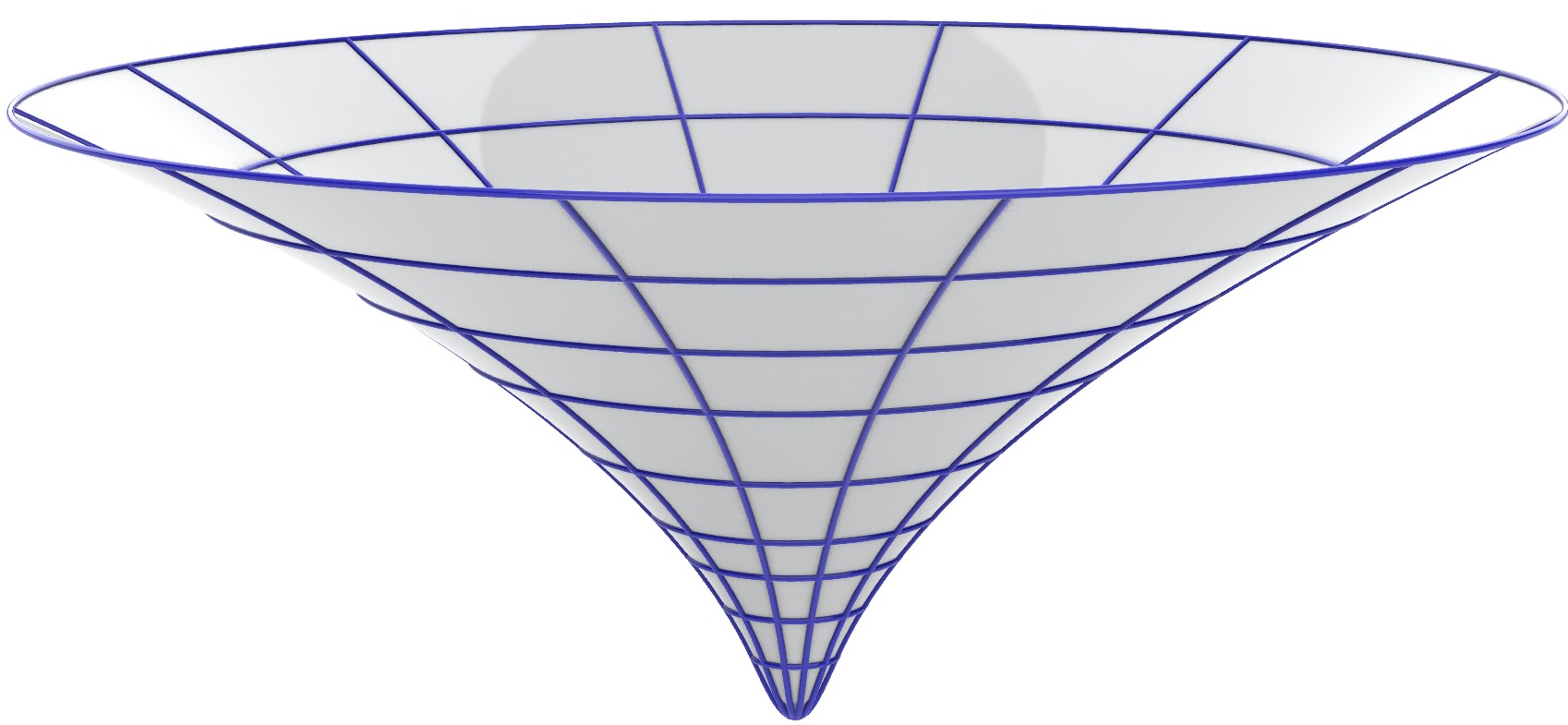}
    %\put(44,80){\contour{white}{$\widetilde{z}$}}
\end{overpic}
%\hfill
\begin{overpic}[width=0.22\textwidth]{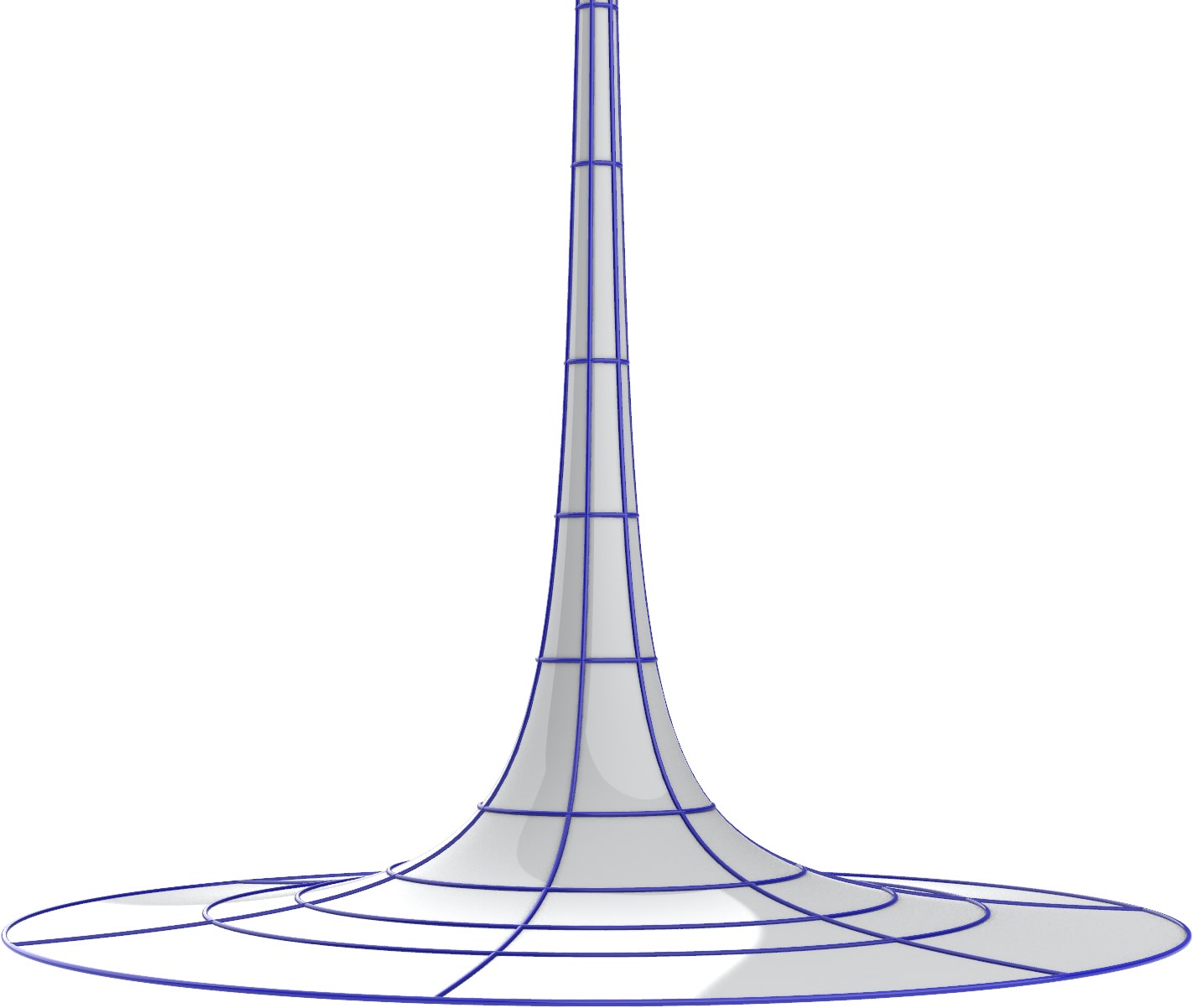}
    %\put(44,80){\contour{white}{$\widetilde{z}$}}
\end{overpic}
\begin{overpic}[width=0.22\textwidth]{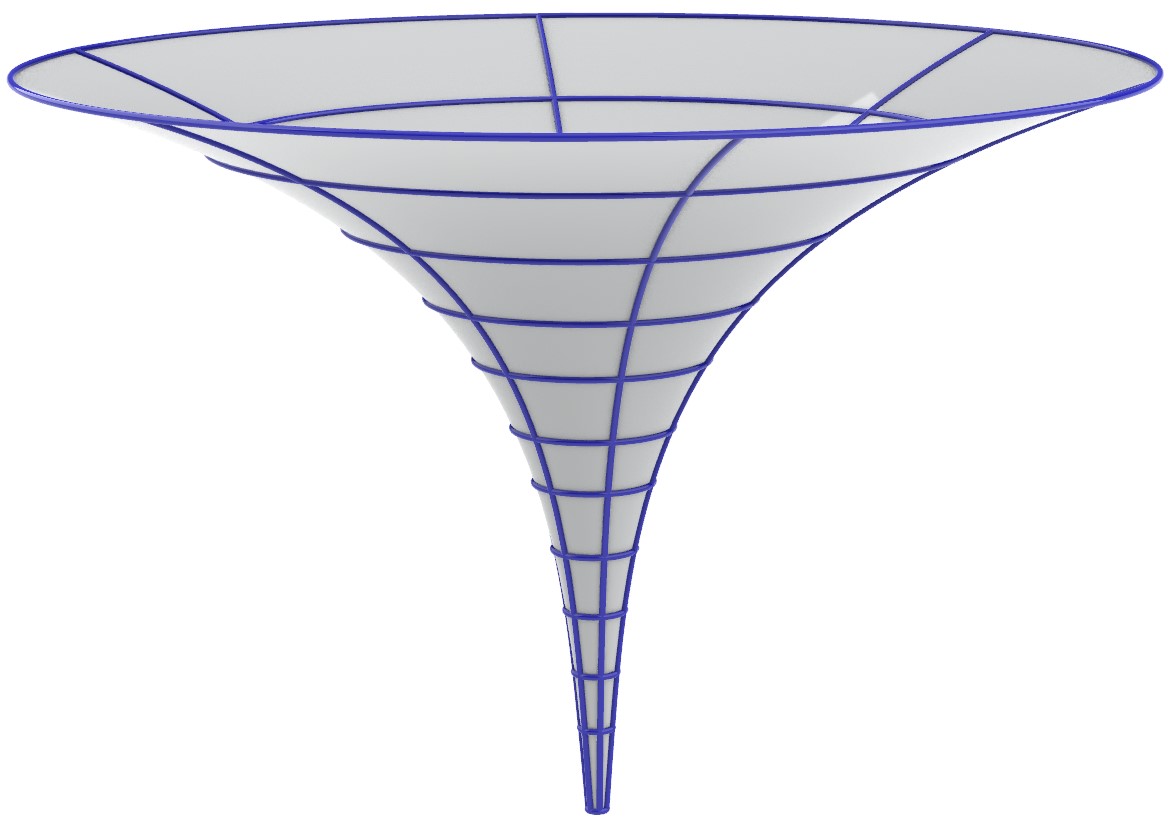}
    %\put(44,80){\contour{white}{$\widetilde{z}$}}
\end{overpic}
 % \hfill{}
    \caption{Rotational isotropic CRPC surfaces (from the left to the right):
the first of two surfaces~\eqref{eq-rotational} for $a>0$, $0>a>-1$, and $a<-1$ respectively;
surface~\eqref{eq-logarithmoid}. 
    }
    %\mscomm{The figure for $a<-1$ looks incorrect: there cannot be any convex part on top!}
    % \khcomm{You are right. I decreased to much height, and the rhino made it convex. I fixed it.}
    %\mscomm{Specify which are the values of $a$ for each of the figures?}
    %\mscomm{Too much white space, remove! Maybe show circles $r=\mathrm{const}$ and profiles $\phi=\mathrm{const}$ instead of the curves $x=\mathrm{const}$ and $y=\mathrm{const}$?}}
    \label{fig:rotational}
\end{figure}
%\bluevar{Surface~\eqref{eq-logarithmoid} is the only isotropic minimal surface of revolution and can be viewed as the isotropic analogue of the catenoid~\cite[Section~4.3(a)]{ds21b}. 
%serves as the simply isotropic analogue of the catenoid in Euclidean space. It also makes sense to call it a \emph{catenoid of $\mathbb{I}^3$} (see, e.g., \cite[Example~4.3]{ds21b}).
%}

%\subsection{Geometry of the surfaces and their characteristic curves}
Let us discuss the geometry of the resulting surfaces, namely, the first of two surfaces~\eqref{eq-rotational} for $a\ne 0,\pm1$. Profile curves~\eqref{eq-profile} are also known as \emph{W-curves} \bluevar{or \emph{$(-a)$-catenaries} \cite[Corollary~1]{dSL23}}: they are
paths of one-parameter continuous subgroups of the group of affine maps, actually, of~$\bluevar{\mathcal{G}^8}$.
% They can be extended to $I^3$ \mscomm{(what does it mean?)} so that they become subgroups of $G^8$. 

The isotropic characteristic curves 
%Due to the analogue of Euler's formula in isotropic geometry \cite[Section~51]{strubecker:1942}, the characteristic curves have the same meaning in isotropic and Euclidean geometry. \mscomm{clarification needed!} They 
intersect the first isotropic principal curvature lines (parallel circles) under the constant isotropic angle $ \gamma/2$ with $\cot^2(\gamma/2) = \lvert a\rvert$. 
Since isotropic angles appear as Euclidean angles in the top view,
the top views of the former must be logarithmic spirals, which intersect the radial
lines at angles $(\pi-\gamma)/2$. Hence, in the \bluevar{cylindrical} coordinate system $(r,\phi,z)$, the isotropic characteristic curves are isotropic congruent to
%(in isotropic and Euclidean geometry) to 
%
\begin{equation*}
r(\phi)= e^{\phi/\sqrt{\lvert a\rvert}},\quad z(\phi)=e^{\phi(1+a)/\sqrt{\lvert a\rvert}}.
\end{equation*}
These curves are again W-curves of a one-parametric subgroup of the isotropic similarity group $\bluevar{\mathcal{G}^8}$, and the rotational surfaces themselves are generated by the \bluevar{subgroup}. Its elements
are compositions of a rotation about the $z$-axis through some angle $\phi$, a homothety with the center at the origin and the
coefficient $e^{\phi/\sqrt{\lvert a\rvert}}$, and the scaling by a factor of $e^{\phi a/\sqrt{\lvert a\rvert}}$ in the vertical direction.

%Pairs of surfaces to $\kappa_2/\kappa_1=\pm a$ are again reciprocal force diagrams and relative minimal surfaces. \mscomm{what does it mean?}

The simplest cases are $a=\pm 1$. We obtain the isotropic sphere %(paraboloid) %with profile 
$z=r^2$ and the logarithmoid $z=\log r$. The latter \bluevar{is the only isotropic minimal surface  of revolution (besides planes and up to general isotropic similarities) and can be viewed as the isotropic analog of the catenoid~\cite[Section~4.3(a)]{ds21b}.} 

Also of interest is the case $a=-1/2$, which leads to surfaces obtained by rotating the parabola $z^2=r$ about
its tangent at the vertex. Recall that in Euclidean geometry, the ratio $\kappa_2/\kappa_1=-1/2$ also leads to parabolae as
profiles, but rotated about the directrix \cite{HelmutHui2022}. Both surfaces are algebraic of order $4$.

Clearly, we get rational algebraic surfaces for rational values of $a\ne -1$. 

\bluevar{Yet another observation is that analytic rotational isotropic CRPC surfaces cannot intersect the rotation axis (at a nonsingular point) unless $a$ or $1/a$ is a positive odd integer.}

\subsection{Euclidean rotational CRPC surfaces}

It is worth %interesting to compare 
comparing profile curves~\eqref{eq-profile} %Proposition~\ref{thm-rotational}
of rotational CRPC surfaces with their analogs in Euclidean geometry (\cite{hopf-1951}, \cite[Ex.~3.27]{kuehnel-2013}, \cite[Eq.~(3.2)]{lopez-pampano-2020}, \cite[Eq.~(7)]{HelmutHui2022}):
\begin{multline}
    \hspace{-0.4cm}h(r)=\int \frac{r^a\,dr}{\sqrt{1-r^{2a}}}
    =\frac{r^{1+a}}{1+a} 
    \, {}_2F_1\left(\frac{1}{2},\frac{1}{2}+\frac{1}{2 a}; \frac{3}{2}+\frac{1}{2 a}; r^{2a}\right),
    %%% VERSION 2:
    %\\=\frac{2 r^{1-a}}{a+1} \left(1-\sqrt{1-r^{2a}} \right) \, {}_2F_1\left(1,\frac{1}{2}-\frac{1}{2 a};\frac{3}{2}+\frac{1}{2 a}; 
    %-\left(\frac{1-\sqrt{1-r^{2a}}}{r^a}\right)^2\right)
    %%% VERSION 3:
    %= r\left\lvert {}_2F_1\left(\frac{1}{2},-\frac{1}{2a};1-\frac{1}{2a};\frac{1}{r^{2a}}\right)\right\rvert, 
    \quad a\ne -1,-\frac{1}{3},%-\frac{1}{5},
    \dots
\end{multline}
%up to similarity.
Here ${}_2F_{1}(\alpha,\beta;\gamma;z)$ is the 
%\emph{principal branch of the 
Gauss hypergeometric function %in $\mathbb{C}-[1,+\infty)$
(see e.g. \cite[Ch.~V, Section~7]{Nehari-75} for a definition), $r^a\le 1$, and $1/a$ is not a negative odd integer. The latter equality is checked in~\cite[Section~3]{check}.

%which we restate here in a more explicit form.
%
%\begin{prop}\label{thm-euclidean-rotational} (Cf.~\cite[Eq.~(3.2) or reference 21 there]{Lopez-Pampano-20}) A rotational surface in Euclidean space has a constant ratio $a< 0$ of principal curvatures if and only if it is similar to a subset of the surface
%\begin{equation}\label{eq-euclidean-rotational}
%z=\pm\sqrt{x^2+y^2}
%\left\lvert {}_2F_1\left(\frac{1}{2},-\frac{1}{2a};1-\frac{1}{2a};\left(x^2+y^2\right)^{-a}\right)\right\rvert,
%\qquad x^2+y^2\ge 1.
%\end{equation}
%\end{prop}
%Here
%$$
%{}_2F_{1}(p,q;r;s):=\frac {\Gamma(r)}{\Gamma(q)\Gamma(r-q)}
%\int_{0}^{1}t^{q-1}(1-t)^{r-q-1}(1-st)^{-p}\,dt
%$$
%denotes the \emph{principal branch of the hypergeometric function} in $\mathbb{C}-[1,+\infty)$; \cite[Ch.~V, Section~7]{Nehari-75} (for $x,y,a$ in question, it has purely imaginary values).

%\mscomm{TODOs for MS: Check the latter equality one more time.} %Write something about the case $a>0$?}

\subsection{Parabolic rotational CRPC surfaces}
In isotropic geometry, there is a second type of rotations, so-called 
\emph{parabolic rotations}, given by
\begin{align*}%\label{eq-parabolic-rotation}
    x'&=x+t,\\
    y'&=y,\\
    z'&=t^2/2+(x+by)t+z,
\end{align*}
for some parameters $b,t$ \cite[Eq.~(2.14)]{sachs}. A surface invariant under these transformations 
for $b$ fixed and $t$ running through $\mathbb{R}$ is called \emph{parabolic rotational}, as well as the image of the surface under any isotropic congruence. 

It is not hard to find all parabolic rotational CRPC surfaces. Indeed, let the graph of a smooth function $z=z(x,y)$ be invariant under the parabolic rotations. The section $x=0$ is the graph of the function $h(y):=z(0,y)$. Applying the parabolic rotation %~\eqref{eq-parabolic-rotation} 
with $t=x$ to the latter curve, we get the identity $z(x,y)=x^2/2+bxy+h(y)$. Then the isotropic Gaussian and mean curvatures are $K=h''-b^2$ and $H=(h''+1)/2$. Then the equation ${H^2}/{K} = (a+1)^2/(4a)$ is equivalent to $a(h''+1)^2=(a+1)^2(h''-b^2)$. Hence $h''=\mathrm{const}$ and $z(x,y)$ is a quadratic function. By Example~\ref{ex-paraboloid},
%By an appropriate isotropic similarity, we can bring the graph to form $z = x^2+ay^2$  for some $a\in\mathbb{R}$. An immediate computation shows that this $a$ actually equals to the ratio of the isotropic principal curvatures of the resulting surface. We have 
we arrive at the following proposition.

%show that the CRPC surfaces generated under limit rotations are paraboloids with isotropic axes,
%congruent in $I^3$ to the surfaces
%
%$$ 2x_3=\kappa_1x_1^2 + \kappa_2 x_2^2. 
%$$
%

\begin{prop}\label{thm-parabolic-rotational}
An admissible parabolic rotational surface has a constant ratio $a\ne 0$ of isotropic principal curvatures if and only if it is isotropic similar to a subset of the paraboloid
%\begin{align*}%\label{eq-rotational}
$
    z = x^2+ay^2.
$ \label{parabolid1}   
%\end{align*}
\end{prop}

\bluevar{An analogous result for constant mean curvature surfaces was obtained in~\cite[Proposition~3]{dSL23}).} 

In this case, both isotropic principal curvatures are constant. In the next section, we show that this is the only surface with this property (Theorem~\ref{l-both-curvatures-constant}).

%\bluevar{Additionally, parabolic rotational CRPC surfaces have a special property: only graphs of quadratic functions are parabolic rotational surfaces with constant mean curvature (see, e.g., \cite[Proposition~3]{dSL23}). Adding a symmetric condition to Theorem~\ref{l-both-curvatures-constant} provides an alternative proof.}

%\HP{Very briefly discuss the discrete setting; it would be nice to show the W-surface property also there.}

%\HP{TODO: Can one obtain anything useful by considering isotropic space as a model of Euclidean Laguerre geometry? Related to the question whether CRPC surfaces have any relevance in  Moebius geometry. M-transforms of A-nets: S-nets with constant angle and Meusnier spheres in a bundle. What if the spheres lie in a complex: Non-euclidean counterparts to CRPC surfaces.}

\section{Channel surfaces} \label{sec-channel}

Now we turn to channel surfaces and show that all channel CRPC surfaces are rotational (or parabolic rotational). Thus we will not encounter new surfaces.
%get no new examples in this class.

%------------------------------------------------------------------
\subsection{Euclidean channel CRPC surfaces}
%-------------------------------------------------------------------

%It appears to be difficult to find other explicit CRPC surfaces different from minimal surfaces
%and rotational CRPC surfaces. 
%R.~L\'opez \cite{lopez:2008} proved that CRPC surfaces which are foliated by circles, are rotational
%surfaces. More recently, 

 A \emph{channel surface} $C$ is defined as the \emph{envelope} of a smooth one-parameter family of spheres $S(t)$, i.e., the %inclusion-maximal 
 surface $C$ \bluevar{that is tangent to} each sphere $S(t)$ along a single closed curve $c(t)$ so that the curves $c(t)$ cover $C$. The curves $c(t)$ are called \emph{characteristics} \bluevar{(%those should 
 not to be confused with the \emph{characteristic curves} discussed in Section~\ref{sec-preliminaries})}. Each characteristic $c(t)$ is a circle which is a principal curvature line on $C$ %, and the principal curvature along $c(t)$ is the inverse of the radius of $S(t)$. 
 (see Lemma~\ref{l-characteristic}). 
 %In general position, the sphere $S(t)$ \bluevar{is tangent to} $C$ along a circular arc $c(t)$ which is a principal curvature line on $C$.
 The locus of sphere centers $s(t)$ is referred to as a \emph{spine curve} and their radii $r(t)$ constitute the \emph{radius function}.  Special cases of channel surfaces are \emph{pipe surfaces}, being envelopes of congruent spheres. %Any smooth part of a channel or a pipe surface is still called a \emph{channel} or \emph{pipe surface} respectively.

For simplicity, we restrict ourselves to analytic surfaces. \bluevar{(Recall that all sufficiently smooth Euclidean or isotropic CRPC surfaces are analytic.)
%by the Petrowsky theorem \cite[p.3--4]{Petrowsky-39}.)
%We conjecture that the same results hold for smooth ones, but the proofs are expected to be more technical.
} If $C$ is analytic (and has no umbilic points), then the family $S(t)$ is also analytic up to a change of the parameter $t$ (because the \bluevar{principal directions}, hence the principal curvature lines $c(t)$, hence the spheres $S(t)$ analytically depend on a point of $C$). Thus by an \emph{analytic channel surface} we mean an analytic surface which is the envelope of a one-parameter analytic family of spheres $S(t)$, i.e, a family such that both $s(t)$ and $r(t)$ are real-analytic functions. %Further, by an \emph{analytic channel Weingarten surface} we mean an analytic channel surface such that the two principal curvatures $\kappa_1$ and $\kappa_2$ satisfy a fixed relation $F(\kappa_1,\kappa_2)=0$ at each point, where $F$ is 

We would like to give
a short proof of the following result by Havlíček \cite{Havlíček53}. %We say that the equation $F(x,y)=0$ is a \emph{regular relation} between $x$ and $y$ if it defines a smooth curve in the $xy$ plane.

\begin{theorem}\label{th-Euclidean-Weingarten}
An analytic channel Weingarten surface %to any regular relation $F(\kappa_1,\kappa_2)=0$ 
is a rotational or pipe surface. %In the latter case,  
%$\kappa_1=\mathrm{const}$ or $\kappa_2=\mathrm{const}$. 
%one of the principal curvatures is constant.
In particular, if an analytic channel surface has a constant ratio of principal curvatures, then it is rotational. 
\end{theorem}

%\mscomm{Is it true that if one of the principal curvatures is constant then the surface is a pipe surface? If yes (and we can find a reference), then we can restate the theorem accordingly. Probably this is not true without additional assumptions on the umbilic points.}

%Our short proof is based on the following well-known properties of Dupin cyclides (see e.g. \cite{blaschke-1929}).  A \emph{Dupin cyclide generated by three given spheres or planes} is the envelope of all spheres and planes tangent to the three given ones. \mscomm{Do we need orientation here?}
%
%\begin{lemma} \label{l-rotation-Dupin}
%A Dupin cyclide generated by three congruent spheres or planes is either a rotational or translational surface.
%\end{lemma}
%
%\begin{proof}
%    \mscomm{MS: Analogously to Khusrav's proof but with centers of the spheres instead of tangency points $p_1,p_2,p_3$.}
%\end{proof}
%
%\begin{lemma} \label{l-2nd-order}
%Consider a channel surface $C$ and a characteristic $c$. Then there is a Dupin cyclide $D$ which has second-order contact with $C$ along $c$ and whose spine curve has second-order contact with the spine curve of $C$.
%\end{lemma}
%
%This lemma is intuitively clear: the desired Dupin cyclide $D$ is generated by three spheres infinitesimally close to the sphere \bluevar{is tangent to} $C$ along $c$. The accurate proof is quite involved and we could not find a good reference for it; hence we postpone it to the appendix. \mscomm{MS: todo in the appendix.}

We start by recalling a well-known proof of the basic properties of characteristics; this will help us in establishing their isotropic analogs.

\begin{lemma}\label{l-characteristic}
    Under the notation at the beginning of Section~\ref{sec-channel}, $c(t)$ is a principal curvature line on $C$, and the principal curvature along $c(t)$ is $1/r(t)$. If the family $S(t)$ is analytic then $c(t)$ is a circle with the axis parallel to $s'(t)$. If $r'(t) \ne 0$ then $s'(t)\ne 0$.   
    %%% Version 2: %%%%
    %Under the notation at the beginning of Section~\ref{sec-channel}, $c(t)$ is a circle with the axis parallel to $s'(t)$. If $r'(t) \ne 0$ then $s'(t)\ne 0$. The circle $c(t)$ is a principal curvature line on $C$ and the principal curvature along $c(t)$ is $1/r(t)$. 
    %%% Version 1: %%%%
    %A characteristic $c(t)$ of a channel surface $C$ is a circle with the axis parallel to the velocity vector $s'(t)$ of the spine curve $s(t)$. The circle $c(t)$ is a principal curvature line on $C$ and the principal curvature along $c(t)$ is $1/r(t)$, where $r(t)$ be the radius of the sphere $S(t)$ that \bluevar{is tangent to} $C$ along $c(t)$. If $r'(t) \ne 0$ then $s'(t)\ne 0$.
\end{lemma}

\begin{proof}
  Since $S(t)$ and $C$ \bluevar{are tangent} along $c(t)$, and $c(t)$ is a  principal curvature line of $S(t)$, it is a principal curvature line of $C$ as well by Joachimsthal theorem. 

  Let us compute the principal curvature $\kappa_1$ along $c(t)$. Let $p$ be a point on $c(t)$ and $T$ be the tangent line to $C$ through $p$. Then by Meusnier's theorem, the osculating circle of $c(t)$ lies on the sphere of radius $1/\kappa_1$ that \bluevar{is tangent to} $C$ at $p$. On the other hand, $c(t)$, hence its osculating circle, lies on the sphere $S(t)$ of radius $r(t)$ that \bluevar{is also tangent to} $C$ at $p$. This implies that $\kappa_1=1/r(t)$.

  Distinct characteristics $c(t_1)$ and $c(t_2)$ cannot have two common points or \bluevar{be tangent to each other}; otherwise $S(t_1)$ and $S(t_2)$ coincide and  \bluevar{would not be tangent to} $C$ along a single closed curve. \bluevar{For $t_2$ close to $t_1$, the curves $c(t_1)$ and $c(t_2)$ cannot have a unique non-tangential intersection point because a neighborhood of $c(t_1)$ is tangent to $S(t_1)$ along $c(t_1)$ and hence is orientable. By continuity, for each $t_2\ne t_1$ the curves $c(t_1)$ and $c(t_2)$ are either disjoint or coincide.}

  Let us compute $c(t)$. The sphere $S(t)$ has the equation $(\mathbf{x}-s(t))^2=r(t)^2$, hence $c(t)$ is contained in the intersection of $S(t)$ with the set $(\mathbf{x}-s(t))s'(t)=r'(t)r(t)$; cf.~\cite[Section~5.13]{bruce-giblin:1984}. If $s'(t)\ne 0$, the latter is a plane orthogonal to $s'(t)$, hence $c(t)$ is a circle with the axis parallel to $s'(t)$ because $c(t)$ is a closed curve by the definition of a channel surface. If $s'(t)=0$, then $c(t)$ is still a circle as a closed curve 
  containing the limit of circles $c(t_n)$, where $t_n\to t$ and $s'(t_n)\ne 0$ (the limit exists by analyticity and does not degenerate to a point because 
  a family of pairwise disjoint circles disjoint from a curve on a surface cannot shrink to a point on the curve).
  %containing any partial limit of circles $c(t_n)$, where $t_n\to t$ and $s'(t_n)\ne 0$ (here we drop all intervals where $s(t)=\mathrm{const}$ and hence $r(t)=\mathrm{const}$). 
  %%%%The axis of $c(t)$ is automatically parallel to $s'(t)=0$.

  %%%% Not convincing so far: %%%%%
  %Here no partial limit can be a point of $c(t)$. Indeed, distinct characteristics cannot cross at two points, otherwise, the corresponding spheres coincide, the same is true for all the close enough characteristics, and we get an interval where $s(t)=\mathrm{const}$, which we have just dropped.  This makes it topologically impossible for a continuous family of characteristics (being closed curves on a surface) to shrink to a point of another characteristic.  
  %%%%%
    
  Finally, if $r'(t)\ne 0$ then $s'(t)\ne 0$ because $c(t)\ne\emptyset$. 
\end{proof}

We now prove the theorem modulo a technical lemma and then the lemma.

\begin{proof}[Proof of Theorem~\ref{th-Euclidean-Weingarten}]
%We start with a certain ``general position'' case and then reduce the whole theorem to this case.
%
Consider an analytic channel Weingarten surface $C$ and a sphere $S=S(t)$ of the enveloping family. They \bluevar{are tangent} along the characteristic $c=c(t)$. 
%It is well known that there is a Dupin cyclide $D(t)$  which has second-order contact with the channel surface $C$ along $c(t)$. \mscomm{MS: This seems to be rather technical to prove.} There is also second-order contact between the spine curves $s(t)$ and $d(t)$ of $D(t)$ and $S(t)$, respectively (see e.g. \cite{blaschke-1929}). 
%Thus the angle between the plane of the circle $c$ and the normals along $c$ is constant.  Hence the normal 
By Lemma~\ref{l-characteristic} the principal curvature along the circle $c$, say $\kappa_1=\kappa_1(t)$, is constant (it is equal to the inverse of the radius of $S(t)$). 
%It is a principal curvature, say $\kappa_1=\kappa_1(t)\ne 0$. 
%If the Weingarten relation is $\kappa_1=\mathrm{const}$, then there is nothing to prove.
%$C$ is a pipe surface because the radius of $S(t)$ equals $1/\kappa_1$, and we are done. (Conversely, a pipe surface has $\kappa_1=\mathrm{const}$; it is not a CRPC surface because only planes, spheres, and rotational cylinders have both $\kappa_1=\mathrm{const}$ and $\kappa_2=\mathrm{const}$  \mscomm{ref!}.)
%%(Conversely, for a Weingarten pipe surface the Weingarten relation must be $\kappa_1=\mathrm{const}$, because there are no surfaces with both $\kappa_1=\mathrm{const}$ and $\kappa_2=\mathrm{const}$ besides spheres an cylinders \mscomm{ref!}.)
If the Weingarten relation is not $\kappa_1=\mathrm{const}$ then the other principal curvature $\kappa_2=\kappa_2(t)$ has to be constant along $c$. 

First, let us consider the \emph{general-position case}: $\kappa_1(t),\kappa_2(t),\kappa_1'(t),\kappa_2'(t)\ne 0$ and $\kappa_1(t)\ne\kappa_2(t)$ for all $t$.  
At each point of $c$, draw an osculating circle of the section of $C$ by the plane orthogonal to $c$. 
%(if $\kappa_2\ne 0$). 
%(it degenerates into a line for $\kappa_2=0$).
Since all such circles have curvature $\kappa_2$ and \bluevar{are tangent to} the sphere $S$, they are obtained by rotations of one circle about the axis of $c$ and form a rotational surface $D$. 
%Consider the family of spheres of radius $1/\kappa_2$ \bluevar{is tangent to} $C$ along $c$ (they degenerate into planes for $\kappa_2=0$). They all \bluevar{are tangent to} the sphere $S$, hence are obtained by rotations of one sphere (or plane) about the axis of $c$. Let $D$ be the envelope of the family; it is obtained by rotations of one circle (or line) about the same axis 
%(This is a particular case of a \emph{Dupin cyclide} \cite{blaschke-1929}.) 
By construction, $C$ and $D$ have a second-order contact along $c$. 
Then Lemma~\ref{l-osculating-Dupin-cyclide} below asserts that at least in the general-position case, the spine curve $s(t)$ of $C$ has a second-order contact with the spine curve of $D$, i.e., the axis of $c(t)$.  Since $s(t)$ has a second-order contact with a straight line for all $t$, it is a straight line and $C$ is a rotational surface. 

Next, let us reduce the theorem to the above general-position case. 

If  $\kappa_1(t)\equiv\kappa_2(t)$ as functions in $t$, then $C$ is a subset of a sphere because $\kappa_1(t)\ne 0$. 

If $\kappa_1(t)=\mathrm{const}$, then $C$ is a pipe surface because $1/\kappa_1(t)$ is the radius of $S(t)$.

If $\kappa_2(t)=\mathrm{const}$, then $C$ is a rotational surface. Indeed, consider a  principal curvature line orthogonal to $c(t)$ and let $p(t)$ be its intersection point with $c(t)$. Let $n(t)$ be the surface unit normal at $p(t)$. Then $n'(t)=\kappa_2(t)p'(t)$. Hence the curvature center $p(t)-n(t)/\kappa_2(t)=\mathrm{const}$ (or $n(t)=\mathrm{const}$ for $\kappa_2(t)=0$). Thus the principal curvature line lies on the sphere of radius $1/\kappa_2(t)$ \bluevar{that is tangent to} the surface (or on a plane for $\kappa_2(t)=0$). For the other principal curvature lines orthogonal to $c(t)$, such spheres (or planes) are obtained by rotations about the axis of the circle $c(t)$. Then $C$ is \bluevar{a subset of} the envelope of those spheres (or planes) and hence it is rotational. 

Otherwise, restrict the range of $t$ to an interval where the above general-position assumptions are satisfied. The envelope of the resulting sub-family $S(t)$ is rotational by the above general-position case. Then the whole $C$ is rotational by the analyticity. 

Finally, for a CRPC surface, the condition $\kappa_1(t)=\mathrm{const}$ implies $\kappa_2(t)=\mathrm{const}$ and yields again to a rotational surface.
\end{proof}

The deep idea beyond this proof is that a channel surface $C$ has a second-order contact with the \emph{osculating Dupin cyclide} $D$ \cite{blaschke-1929} along a characteristic, and the spine curves of $C$ and $D$ also have a second-order contact. Here $D$ is the limit of the envelope of all spheres tangent to three spheres $S(t_1)$, $S(t_2)$, $S(t_3)$ for $t_1,t_2,t_3$ tending to $t$. The spine curves of $C$ and the envelope have $3$ colliding common points $s(t_1),s(t_2),s(t_3)$, leading to a second-order contact. For a Weingarten surface $C$, the osculating Dupin cyclides $D$ turn out to be rotational, with straight spine curves, hence the same must be true for~$C$ itself.

Although this construction gives geometric insight, passing to the limit $t_1,t_2,t_3\to t$ rigorously is a bit technical. Thus we complete the proof by a different argument relying on additional curvature assumptions.

\begin{lemma}\label{l-osculating-Dupin-cyclide}
    Under the notation of the proof of Theorem~\ref{th-Euclidean-Weingarten}
    and the assumptions $\kappa_1(t),\kappa_2(t),\kappa_1'(t),\kappa_2'(t)\ne 0$ and $\kappa_1(t)\ne\kappa_2(t)$ for all $t$, the spine curve $s(t)$ has a second-order contact with the axis of the circle $c(t)$.
\end{lemma}

\begin{proof} %First let us prove that $s'(t)$ is nonzero and parallel to the axis of $c(t)$. Indeed, the sphere $S(t)$ has the equation $(\mathbf{x}-s(t))^2=1/\kappa_1(t)^2$, and the characteristic $c(t)$ is contained in the intersection of $S(t)$ with the set $(\mathbf{x}-s(t))s'(t)=\kappa_1'(t)/\kappa_1(t)^3$; cf.~\cite[Section~5.13]{bruce-giblin:1984}. Since $\kappa_1'(t)\ne 0$ and $c(t)\ne\emptyset$, it follows that $s'(t)\ne 0$ and $c(t)$ lies in a plane orthogonal to $s'(t)$. 
In what follows we fix a value of $t$, say, $t=0$, and assume that $t$ is sufficiently close to this value. It is also convenient to assume that $\kappa_2'(0)(\kappa_1(0)-\kappa_2(0))>0$ and $t>0$. Since $\kappa_1'(0)\ne 0$, by Lemma~\ref{l-characteristic} it follows that $s'(0)$ is nonzero and parallel to the axis of $c(0)$. % (this means first-order contact).
%, it follows that the distance between the projections of $s(t)$ and $s(0)$ to the axis is at least $\mathrm{const}\cdot t$. 
It remains to prove that the distance from $s(t)$ to the axis is $O(t^3)$. 

We do it by reducing to a planar problem; see Fig.~\ref{fig:2} to the left. %Fix some $t_0> 0$ and the plane $P$ passing through $s(t_0)$ and the axis of $c(0)$. 
\bluevar{Take a plane $P$ passing through $s(0)$ and parallel to $s'(0)$ and $s''(0)$ (an ``osculating plane'' of $s(t)$). It contains the axis of $c(0)$.}
The section of $c(t)$ by the plane $P$ consists of two points $c_1(t)$ and $c_2(t)$. Let $c_1$ and $c_2$ be the two curves formed by these points. The section of $C$ by $P$ coincides with $c_1$ and $c_2$ because the characteristics cover the surface. 
The section of $S(t)$ is a circle $S_1(t)$ \bluevar{tangent to} $c_1$ and $c_2$ at $c_1(t)$ and $c_2(t)$. %; the radius of $S(0)$ is $1/\kappa_1(0)$. 
Let $D_1(t)$ and $D_2(t)$ be the osculating circles of $c_1$ and $c_2$ at $c_1(t)$ and $c_2(t)$. 
%The section of $D$ consists of two circles $D_1$ and $D_2$ of radius $1/\kappa_2(0)$ having a second-order contact with the curves at the points $c_1(0)$ and $c_2(0)$. %We have $D_1,D_2\ne S_1(0)$ because $\kappa_2(0)\ne \kappa_1(0)$. 
Orient each circle $S_1(t)$ counterclockwise and fix
the orientations of $D_1(t)$ and $D_2(t)$ such that the contacts are oriented. Let $s_1(t)$ be the center of $S_1(t)$. \bluevar{Since $P\parallel s'(0), s''(0)$, %$s_1(t_0)=s(t_0)$, 
it follows that the distance between $s_1(t)$ and $s(t)$ is $O(t^3)$. Thus}
it suffices to prove that the distance from $s_1(t)$ to the bisector of $c_1(0)c_2(0)$ is $O(t^3)$. 
\begin{figure}[h]
\hfill
\begin{overpic}[width=0.3\textwidth,  ]{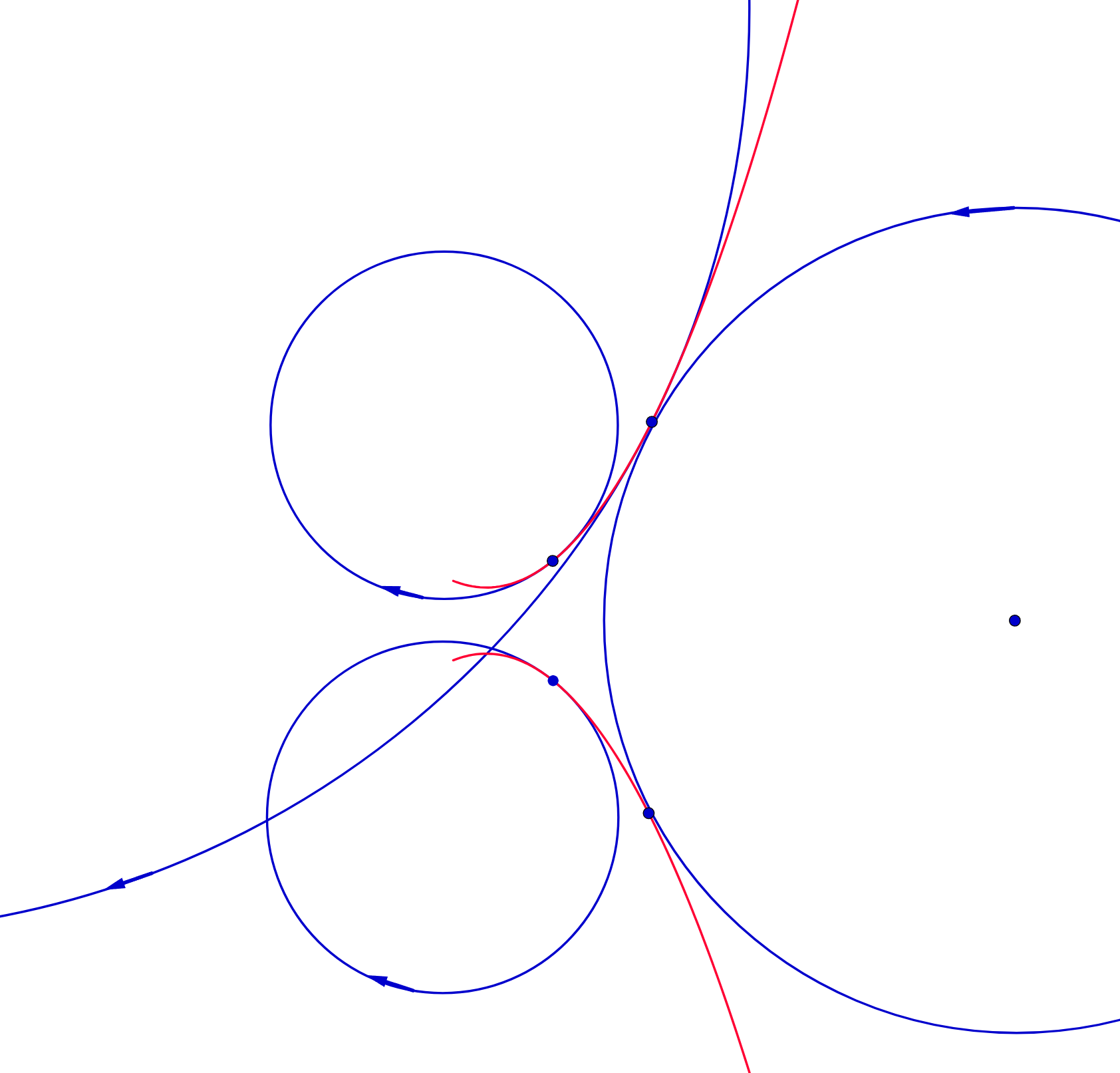}
   \small \put(-6,22){\contour{white}{$D_1(t)$}}
     \put(27,-1){\contour{white}{$D_2(0)$}}
     \put(27,77){\contour{white}{$D_1(0)$}}
     \put(33,50){\contour{white}{$c_1(0)$}}
     \put(34,25){\contour{white}{$c_2(0)$}}
     \put(92,40){\contour{white}{$s_1(t)$}}
     \put(78,81){\contour{white}{$S_1(t)$}}
     \put(68,1){\contour{white}{$c_2$}}
     \put(72,91){\contour{white}{$c_1$}}
     \put(59, 55){\contour{white}{$c_1(t)$}}
     \put(60,24){\contour{white}{$c_2(t)$}}
\end{overpic}
  \hspace{0.1cm}
  \hfill{}
  \begin{overpic}[width=0.3\textwidth, ]{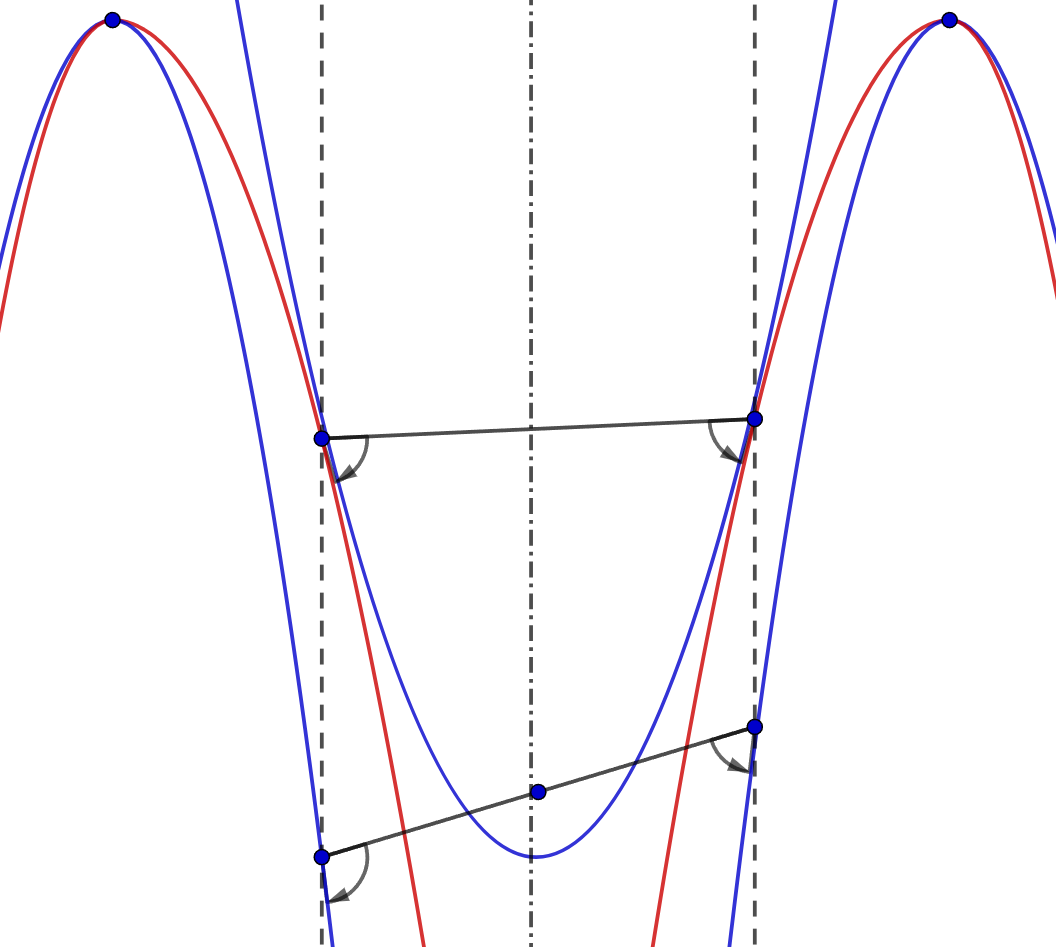}
   \small \put(5,91){\contour{white}{$c_2(0)$}}
   \put(83,91.5){\contour{white}{$c_1(0)$}}
   \put(51,75){\contour{white}{$L$}}
  \put(9,60){\contour{white}{$D_2$}}
  \put(79,60){\contour{white}{$D_1$}}
  \put(31,52){\contour{white}{$c_2(t)$}}
  \put(52,53){\contour{white}{$c_1(t)$}}
  \put(41,2){\contour{white}{$c_2$}}
  \put(53,2){\contour{white}{$c_1$}}
  \put(10,6){\contour{white}{$p_2(t)$}}
  \put(73,20){\contour{white}{$p_1(t)$}}
  \put(50.25,88){\color{blue}\circle{1}}%
  %\put(50.25,88){\contour{blue}{\circle{1}}}
  \put(48,92){\contour{white}{$s(0)$}}
\linethickness{0.05mm}	
	\put(34,6){\color{black}\line(1,-3){4}}
 \put(35,-11){\contour{white}{$\alpha_1(t)$}}
 \linethickness{0.05mm}	
	\put(69,17){\color{black}\line(-1,-3){8}}
 \put(58,-11){\contour{white}{$\alpha_2(t)$}}
\end{overpic}
\hfill{}
  \begin{overpic}[width=0.3\textwidth, ]{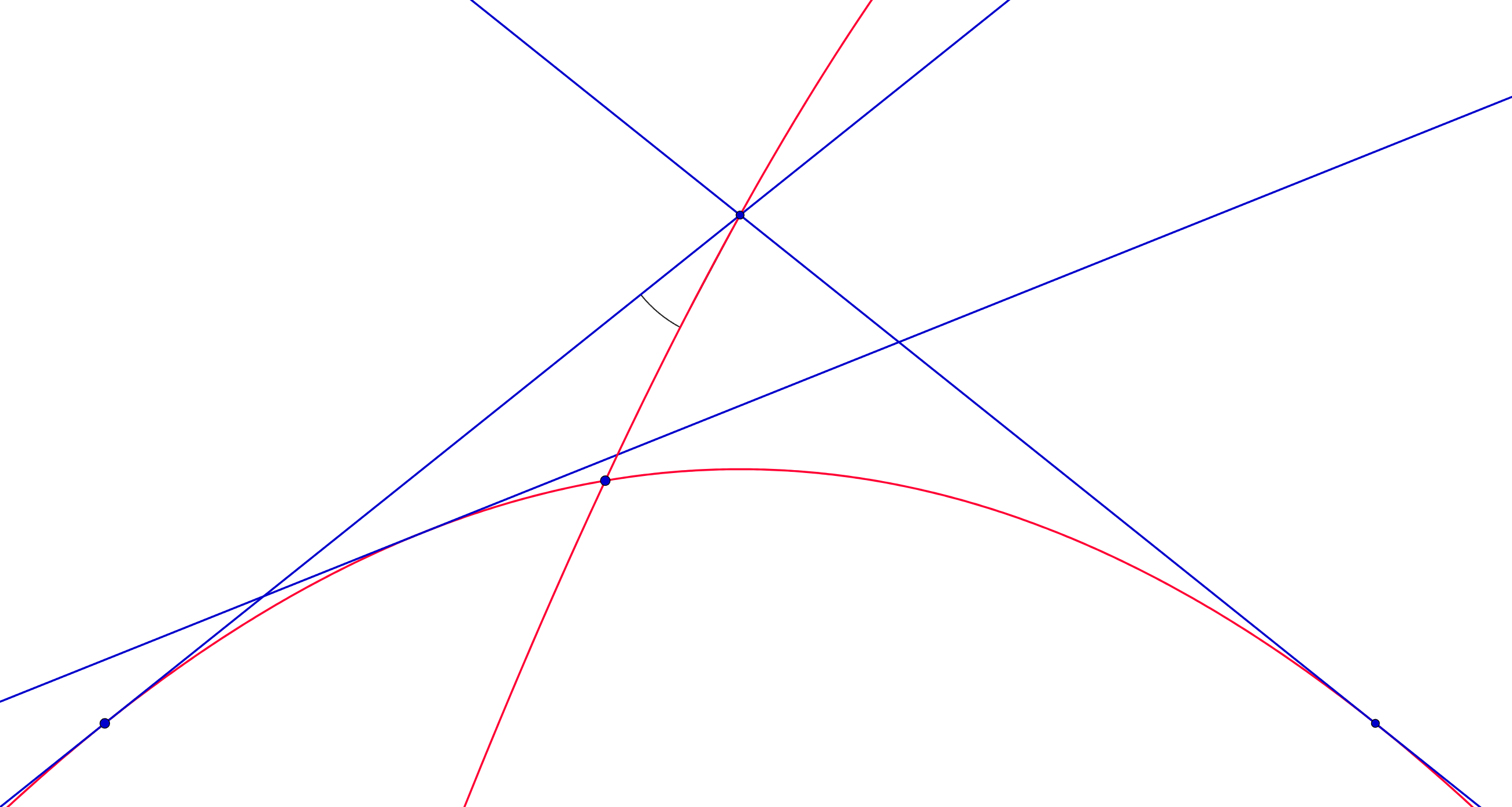}
   \small 
     \put(33,54){\contour{white}{$R_{b}'$}}
     \put(51,53){\contour{white}{$\alpha'$}}
     \put(64,45){\contour{white}{$R_{a}'$}}
     \put(90,49){\contour{white}{$R_{t}'$}}
     \put(34,34.5){\contour{white}{$C'$}}
     \put(6,1){\contour{white}{$A'$}}
     \put(90,7){\contour{white}{$B'$}}
     \put(39,13){\contour{white}{$O'$}}
     \put(65,13){\contour{white}{$e$}}
     \put(37.5,26){\contour{white}{$\gamma$}}
     %\put(41.2,31){\contour{white}{$\angle$}}
\end{overpic}
  \hfill{}
  \vspace{0.4cm}
    \caption{The notation in the proofs of Lemmas~\ref{l-osculating-Dupin-cyclide}, \ref{l-isotropic-circle-center}, and~\ref{l-conoid} (from the left to the right).}
    \label{fig:2}
\end{figure}

Next, let us prove that $S_1(t)$ and $D_1(0)$ are disjoint.
For this purpose, we compute the derivative of the signed curvature $k_1(t)$ of the curve $c_1(t)$ at $t=0$. By the Euler and Meusnier theorems, we get 
$$k_1(t)=
\frac{\kappa_1(t)\cos^2 \angle(c_1,c(t)) +\kappa_2(t)\sin^2\angle(c_1,c(t))}
{\cos\angle (s(t)c_1(t),P)}.
$$
Differentiating, we get $k_1'(0)=\kappa_2'(0)\ne 0$ because $\angle(c_1,c(0))=\pi/2$ and $\angle (s(0)c_1(0),P)=0$. %Then by the Tait--Kneser theorem, the osculating circles $D_1(t)$ are nested. 
Since the signed curvatures of $S_1(0)$ and $D_1(0)$ are 
$\kappa_1(0)$ and $\kappa_2(0)$ respectively, by the assumption $\kappa_2'(0)(\kappa_1(0)-\kappa_2(0))>0$ it follows that the signed curvature of $D_1(t)$ is between the signed curvatures of $D_1(0)$ and $S_1(t)$ for $t>0$ small enough. By the Tait--Kneser theorem, $D_1(t)$ is disjoint from $D_1(0)$, and by construction, $D_1(t)$ has an oriented contact with $S_1(t)$. Thus $D_1(0)$ and $S_1(t)$ are separated by $D_1(t)$, hence are disjoint.  

Finally, we estimate the distance from $s_1(t)$ to the bisector of $c_1(0)c_2(0)$. Since $c_1$ and $D_1(0)$ have second-order contact, it follows that 
$c_1(t)$ is within distance $O(t^3)$ from $D_1(0)$. Then the distance between the disjoint circles $S_1(t)$ and $D_1(0)$ is $O(t^3)$. Analogously, the distance between $S_1(t)$ and $D_2(0)$ is $O(t^3)$. Then the difference in the distances from the center $s_1(t)$ to the centers of $D_1(0)$ and $D_2(0)$ is $O(t^3)$. Since $D_1(0)$ and $D_2(0)$ are symmetric with respect to the bisector of $c_1(0)c_2(0)$, it follows that 
the distance from $s_1(t)$ to the bisector is $O(t^3)$, which proves the lemma.
\end{proof}

\subsection{Surfaces with both isotropic principal curvatures constant}

As a motivation for studying isotropic channel surfaces,
and also as one step in their classification, let us find all surfaces with both isotropic principal curvatures constant.

\begin{theorem}\label{l-both-curvatures-constant}
An admissible surface has constant nonzero isotropic principal curvatures $\kappa_1$ and $\kappa_2$ if and only if it is isotropic congruent to a subset of the paraboloid
%\begin{align*}%\label{eq-rotational}
$
    2z = \kappa_1 x^2+\kappa_2 y^2.
$    
%\end{align*}
%\mscomm{Unproved so far.}
\end{theorem}

Let us see how the notion of an isotropic channel surface naturally arises in the proof of this theorem.

\begin{lemma}\label{constant_isphere}
Assume that the isotropic principal curvature $\kappa_1$ % is constant and nonzero 
along an isotropic principal curvature line $c$ of an admissible surface $C$ \bluevar{is constant and nonzero}. At each point of $c$, take the parabolic isotropic sphere with the radius $1/\kappa_1$ that \bluevar{is tangent to} $C$ at this point. Then all these isotropic spheres coincide.
   %Consider an admissible surface $C$ with one of the isotropic principle curvatures $\kappa_1 = \mathrm{const}\ne 0$.  Consider an isotropic principle curvature line $c$ of $C$ such that isotropic principle curvature along $c$ is $\kappa_1$. At each point of $c$,  take the parabolic isotropic sphere with the radius $1/\kappa_1$ that \bluevar{is tangent to} $C$ at this point. Then all these isotropic spheres coincide. 
\end{lemma}
%add the fact of parabaloid envelope
\begin{proof}
%Two points $p = (a, b, c)$ and $p_1 = (a,b,d)$ with the same top view are called \emph{parallel points}. The $\lvert c - d\rvert$ is called \emph{replacing distance} between them.
%By the \emph{center} of a parabolic isotropic sphere of radius $r$ we mean the point inside the isotropic sphere that lies on the vertical line through the vertex of the paraboloid at the Euclidean distance $1/r$ from the vertex.
By the \emph{center} of a parabolic isotropic sphere of radius $r$ we mean the point obtained from the vertex of the paraboloid by the translation by the vector $(0,0,r)$.
The unique parabolic isotropic sphere of radius $r$ with the center $m=(m_1,m_2,m_3)$ is given by the equation
\begin{equation}\label{eq-center}
   2(z-m_3+r) = \frac{1}{r}\left((x-m_1)^2+(y- m_2)^2\right).
\end{equation}
Beware that the notion of a center is not invariant under isotropic congruences, but similar notions are common in isotropic geometry \cite[Chapter 9]{sachs}.

The \emph{isotropic curvature center} of $C$ at a point $p = (x,y,z) \in c$ is the center of the parabolic isotropic sphere of radius $r=1/\kappa_1$ that \bluevar{is tangent to} $C$ at $p$. It is given by %the formula
$$
m=
%\left[
%\begin{matrix}m_1\\ m_2\\ m_3\end{matrix}
%   \right]=
\left[
\begin{matrix}
      x - \frac{1}{\kappa_1} f_x \\
      y - \frac{1}{\kappa_1} f_y  \\
      z - \frac{1}{2\kappa_1}\left(f_x^2+f_y^2-2\right)
   \end{matrix}
   \right],
$$
where we locally represent $C$ as the graph of a function $z=f(x,y)$. This formula is obtained from the three equations~\eqref{eq-center}, $f_x= \kappa_1(x-m_1)$, $f_y= \kappa_1(y -m_2)$.

Now let $p(t) =(x(t), y(t), f(x(t),y(t)))$ run through the isotropic principal curvature line 
$c$ and let $m(t)$ be the corresponding isotropic curvature center. It suffices to prove that %$m(s)$ does not depend on $s$, i.e. 
$m'(t)=0$. We omit the arguments of the functions $x,y,f$ in what follows.

%The top view of the direction of the curve $c$  is the vector $(p_1(s)^\prime, p_2(s)^\prime)^T$. 
Since $c$ is the isotropic principal curvature line, it follows that 
%$$\nabla^2(f)(p_1(s), p_2(s)) (p_1(s)^\prime, p_2(s)^\prime)^T = \kappa_1 (p_1(s)^\prime, p_2(s)^\prime)^T.$$ 
%Thus 
\begin{align*}
    \begin{cases}
      f_{xx}x' + f_{xy}y' = \kappa_1 x'\\
      f_{xy}x' + f_{yy}y' = \kappa_1 y',
    \end{cases}\,
\end{align*}
hence 
$$
m'=\left[
\begin{matrix}
x' - \frac{1}{\kappa_1}f_{xx}x'-\frac{1}{\kappa_1}f_{xy}y'\\
y' - \frac{1}{\kappa_1}f_{xy}x'-\frac{1}{\kappa_1}f_{yy}y'\\
f_{x}x' + f_{y}y' - f_{x}x' - f_{y}y'
\end{matrix}
   \right]
=0.
$$
\end{proof}

The meaning of this lemma is that $\kappa_1 = \mathrm{const}$ implies that the surface is essentially an isotropic pipe surface, which we are going to define now.

%%%% TOO TECHNICAL DEFINITIONS %%%%
%An \emph{isotropic pipe surface} is the envelope of a continuous one-parameter family of congruent parabolic isotropic spheres, i.e. the inclusion-maximal surface $C$ \bluevar{is tangent to} each sphere $S(t)$ along a single continuous curve $c(t)$ so that the curves $c(t)$ cover $C$.
%
%This `global' definition makes sense when the underlying inclusion-maximal surface $C$ exists, or, speaking geometrically, when the envelope does not have singularities. In the opposite case, we need the following `local' definition.
%
%By a \emph{smooth part of the envelope} of the family $S(t)$ we mean any surface $C_1$ \bluevar{is tangent to} each sphere $S(t)$ along a single continuous curve $c(t)$ so that the curves $c(t)$ cover $C_1$ (even if the above $C$ does not exist). A surface is \emph{locally an isotropic pipe surface} if each point has a neighborhood that is a smooth part of the envelope of a family of congruent parabolic isotropic spheres. (The family may depend on the point.) The curves $c(t)$ are still called \emph{characteristics}. 
%
An \emph{isotropic channel surface} is the envelope of a smooth one-parameter family of parabolic isotropic spheres $S(t)$, i.e.~the surface $C$ \bluevar{tangent to} each isotropic sphere $S(t)$ along a single curve $c(t)$ 
without endpoints so that the curves $c(t)$ cover $C$. The curves $c(t)$ are called \emph{characteristics}.
%. The \emph{envelope}, \emph{characteristics} $c(t)$, etc are defined analogously to the Euclidean case. 
If the radii of the isotropic spheres are constant, then %the envelope 
$C$ is called an \emph{isotropic pipe surface}. 
%
%An immediate consequence of %these definitions and Lemma~\ref{constant_isphere} is the following result.
%
%\begin{lemma}\label{constant_isphere_envelope} 
%An admissible surface with one of the isotropic principle curvatures $\kappa_1 = \mathrm{const}\ne 0$ and without umbilic points is locally an isotropic pipe surface.
%%%%is the envelope of the isotropic spheres constructed in Lemma~\ref{constant_isphere} for all isotropic principle curvature lines, unless $C$ is contained in an isotropic sphere. In particular, $C$ 
%%%%is an isotropic pipe surface or a subset of an isotropic sphere. 
%\end{lemma}
%
%\begin{proof}
%    
%\end{proof}

To proceed, we need an isotropic analog of Lemma~\ref{l-characteristic} above.

\begin{lemma}\label{ip-curvature}
    Consider a characteristic $c(t)$ of an isotropic channel surface $C$. Let $r(t)$ be the radius of the isotropic sphere $S(t)$ that \bluevar{is tangent to} $C$ along~$c(t)$. Then $c(t)$ is an isotropic principal curvature line on $C$ and the isotropic principal curvature along $c(t)$ is $1/r(t)$.
    If $r'(t) \ne 0$ then $c(t)$ is %an arc of 
    an elliptic isotropic circle.
    If $r(t)=\mathrm{const}$ then for some $t$ the curve $c(t)$ is %an arc of 
    a parabolic isotropic circle.
    %and its radius.
    %Each characteristic $c(t)$ of an isotropic channel surface $C$ is an isotropic circle, which is an isotropic principal curvature line. The isotropic principal curvature along $c(t)$ equals 
    %is the inverse of % 
    %$1/r(t)$, where $r(t)$ is 
    %the radius of the isotropic sphere $S(t)$ that \bluevar{is tangent to} $C$ along $c(t)$. If $r'(t)\ne 0$, then the isotropic circle $c(t)$ has elliptic type; otherwise, it has parabolic type.  \mscomm{If $r'(t)\ne 0$, then the isotropic circle $c(t)$ has elliptic type. --- Khusrav, will you add the proof of the latter?}
\end{lemma}

\begin{proof}
  Since $c(t)$ is an isotropic principal curvature line of $S(t)$, then by the isotropic Joachimsthal theorem \cite[Section~36]{strubecker:1942a}, %Lemma \ref{comm:principalcurvatureline}, 
  it is an isotropic principal curvature line of $C$. 

  Let $p$ be a point on $c(t)$ and $T$ be the tangent line to $C$ through $p$. Then by Meusnier's theorem (Theorem~\ref{l-Meusnier}), the osculating isotropic circle of $c(t)$ lies on the isotropic sphere of radius $1/\kappa_n$ that \bluevar{is tangent to} $C$ at $p$. Here $\kappa_n$ is the isotropic principal curvature along $c(t)$ because $c(t)$ is an isotropic principal curvature line.

  On the other hand, $c(t)$, hence its osculating isotropic circle, lies on the isotropic sphere $S(t)$ of radius $r(t)$ that \bluevar{is also tangent to} $C$ at $p$. This implies that $\kappa_n=1/r(t)$. 

    %Let us prove that this isotropic sphere is tangent to the surface. Assume the converse. Thus, the tangent plane of $S$ intersects $C$ along a curve. Therefore, the isotropic normal curvature of $C$ at $p$ with the direction $T$ is not zero contradiction. 
    
    %It is well known that each plane that contains $T$ in intersection with this sphere gives us an osculating isotropic circle. In particular, the tangent plane of $C$ at $p$ gives us an osculating isotropic circle with zero radii (point). 
    %Thus, this tangent plane is the tangent plane to this sphere at $p$ as well; hence, this sphere \bluevar{is tangent to} $C$ at $p$.
    %Since this sphere \bluevar{is tangent to} $C$, it coincides with  $S(t)$.
    %Then, by Meusnier's theorem (Theorem~\ref{l-Meusnier}), isotropic principal curvature along $c(t)$ is equal to $1/r(t)$. 
    %By applying  Lemma \ref{l-Meusnier}, to the surface $C$, point $P$ and dir  isotropic principal curvature along it is equal to the inverse of the radius of $S(t)$ \mscomm{---why?}. 

  %Let $C$ be the envelope of the family of parabolic isotropic spheres 
  Now let $S(t)$ have the equation $$2z =  A(t)\left(x^2+y^2\right)+B(t) x+C(t) y+D(t). $$ Then %the characteristic 
  $c(t)$ is contained in the set defined by the system (cf.~\cite[Section~5.13]{bruce-giblin:1984}) %\mscomm{--- it does not seem to be always true with our definition of the envelope!}
\begin{equation}
  \left\{\begin{array}{@{}l@{}}\label{envelope}
  A(t)\left(x^2+y^2\right)+B(t) x+C(t) y+D(t) -2z = 0,\\
  A^\prime(t)\left(x^2+y^2\right)+B^\prime(t)x+C^\prime(t) y+D^\prime(t) = 0.
  \end{array}\right.\,
\end{equation}
If $r'(t)\ne 0$ then $A'(t) =r^\prime(t)/r(t)^2 \ne 0$ because $r(t) = 1/A(t)$. Remove the quadratic terms in the second equation by subtracting the first equation with the coefficient $A^\prime(t)/A(t) \ne 0$. This introduces a term linear in $z$ into the second equation. Thus $c(t)$ is the intersection of $S(t)$ with a non-isotropic plane, hence an elliptic isotropic circle. 

Finally, assume that $r(t)=\mathrm{const}$ so that $A'(t) =0$. There exists $t$ such that at least one of the derivatives $B'(t),C'(t),D'(t)\ne 0$, otherwise all $S(t)$ coincide and there is no envelope. Then the second equation of~\eqref{envelope} defines an isotropic plane. Thus 
$c(t)$ is the intersection of $S(t)$ with the plane, hence a parabolic isotropic circle.  
    %\mscomm{What if the two equations in the system are proportional? Assume $r'(t)\ne 0$ in the lemma!}
    %Moreover, 
\end{proof}

This lemma and its proof remain true, if we allow the characteristics $c(t)$ to have endpoints (in the definition of an isotropic channel surface); then 
%for an arbitrary surface $C$ \bluevar{ tangent to} each sphere $S(t)$ along a curve $c(t)$, not necessarily the inclusion-maximal surface (envelope, which may now have singularities). The only difference is that 
$c(t)$ is going to be an arc of an isotropic circle instead of a full one.

\begin{proof}[Proof of Theorem~\ref{l-both-curvatures-constant}] 
By Lemma~\ref{constant_isphere},
%Lemma~\ref{l-both-curvatures-constant-1},
%the surface is locally an envelope of 
locally there are two families of congruent parabolic isotropic spheres of radii $1/\kappa_1$ and $1/\kappa_2$ respectively \bluevar{that are tangent to} the surface along isotropic principal curvature lines. By the last assertion of Lemma~\ref{ip-curvature}, one of the characteristics $c(t)$ of one family 
is an arc of a parabolic isotropic circle. Performing an appropriate isotropic congruence, one can make the parabolic isotropic sphere $S(t)$ \bluevar{that is tangent to} the surface along $c(t)$ symmetric with respect to the plane of $c(t)$. The isotropic spheres of the other family are congruent and touch $S(t)$ at the points of~$c(t)$. Hence they are obtained by parabolic rotations of one isotropic sphere, and thus the surface is locally parabolic rotational. By Proposition~\ref{thm-parabolic-rotational}, the theorem follows. (The envelope of the other family can also be computed directly.)
\end{proof}

\subsection{Isotropic channel CRPC surfaces}

%\HP{There should be an analogous result in isotropic geometry.} 
The classification of isotropic channel CRPC surfaces is similar to the Euclidean ones, 
but in addition to rotational surfaces, we get parabolic rotational ones. 

An \emph{analytic isotropic channel surface} is an analytic surface that is the envelope of a one-parameter analytic family of parabolic isotropic spheres.

% An \emph{isotropic Dupin cyclide generated by three parabolic isotropic spheres} is the envelope of all parabolic isotropic spheres and planes that are tangent to these three. The same cyclide is the envelope of another family of parabolic isotropic spheres and planes including the initial three isotropic spheres. Both families are called \emph{generating families} for the cyclide. The principal curvature lines on an isotropic Dupin cyclide are isotropic circles according to Lemma~\ref{comm:principalcurvatureline}. 

\begin{theorem} \label{th-isotropic-weingarten}
An analytic isotropic channel Weingarten surface is an isotropic rotational or isotropic pipe surface.
%unless one of the isotropic principal curvatures is constant.
%%\mscomm{MS: it is not necessarily THE Weingarten relation because both isotropic principle curvatures can be constants simultaneously.}
In particular, an analytic isotropic channel surface with a constant nonzero ratio of isotropic principal curvatures is a subset of an isotropic rotational or parabolic rotational one. 
%(in the latter case, both isotropic principal curvatures are constant). 
%%%% OLD VERSION: %%%%
%An isotropic channel Weingarten surface to any regular relation $F(\kappa_1,\kappa_2)=0$ is an isotropic rotational or isotropic pipe surface. In the latter case, $\kappa_1=\mathrm{const}$ or $\kappa_2=\mathrm{const}$. In particular, if a channel surface has a constant ratio of isotropic principle curvatures, then it is either isotropic rotational or parabolic rotational (in the latter case, both $\kappa_1=\mathrm{const}$ and $\kappa_2=\mathrm{const}$). 
%%%%%%%%%%%%
\end{theorem}

The proof is analogous to the Euclidean one, but we consider the centers curve instead of the spine curve.
If the characteristics $c(t)$ are elliptic isotropic circles, then the locus of their centers $s(t)$ %of their top views 
is called the \emph{centers curve}.

 \begin{lemma}\label{cone:vertices}
     If the centers curve of an \bluevar{analytic} isotropic channel surface is contained in a vertical line, then the surface is isotropic rotational.
 \end{lemma}

 \begin{proof}
     Bring the vertical line to the $z$-axis by an appropriate isotropic congruence.
     Use the notation from the proof of Lemma \ref{ip-curvature}. If $A'(t)\ne 0$, then the second equation of~\eqref{envelope} defines a circle, which represents the top view of the characteristic $c(t)$. Thus the top view of all the characteristics $c(t)$ are circles with the center at the origin. 
     %\mscomm{This needs a one-sentence explanation; a priori unclear why cone vertices have anything to do with the circle centers.}
     %Since the center of all characteristics lies on that vertical line. 
     Therefore $B^\prime(t) = C^\prime(t) =0$ for all $t$, hence $B(t)$ and $C(t)$ are constants.  Then the second equation of  \eqref{envelope} takes form $x^2+y^2 = -D'(t)/ A'(t)$. Substituting it into the first equation of~\eqref{envelope}, we obtain that all the isotropic circles $c(t)$ lie in the planes parallel to one plane  $B(0)x+C(0)y-2z = 0$. By continuity, this remains true for the roots of the equation $A'(t)=0$. Therefore the surface is isotropic rotational. 
 \end{proof}

To ensure that the centers curve $s(t)$ is a vertical line, we need the following two technical lemmas. In the first one, for a line segment joining two symmetric parabolic isotropic circles, we express the distance from the midpoint of the segment to the symmetry axis in terms of the replacing angles between the segment and the isotropic circles. The \emph{(oriented) replacing angle} between two non-isotropic lines lying in one isotropic plane is the difference in their slopes. 
%It is well-known and easy to prove that the oriented angles 

\begin{lemma}\label{l-isotropic-circle-center} (See Fig.~\ref{fig:2} to the middle.)
Assume that two parabolic isotropic circles $D_1$ and $D_2$ of isotropic curvature $A$ lie in an isotropic plane and are symmetric with respect to an isotropic line $L$. Let the distance between their axes be $2B\ne 0$. Let two points $p_1$ and $p_2$ lie on $D_1$ and $D_2$ respectively. Let the segment $p_1p_2$ have isotropic length $d$ and form replacing angles $\alpha_1$ and $\alpha_2$ with $D_1$ and $D_2$ respectively. Then the isotropic distance from the midpoint of $p_1p_2$ to the line $L$ equals $\lvert\alpha_1+\alpha_2\rvert d/\lvert 8AB\rvert $.
\end{lemma}

\begin{proof}
%This is a direct computation in the Cartesian coordinates. 
Without loss of generality, $D_1$ and $D_2$ lie in the $xz$-plane and have the equations $z=A(x\pm B)^2$. If the $x$-coordinates of $p_1$ and $p_2$ are $x_1$ and $x_2$, then 
we compute directly $\lvert\alpha_1+\alpha_2\rvert=4\lvert AB\rvert \cdot\lvert x_1+x_2\rvert/\lvert x_1-x_2\rvert$. Since 
$\lvert x_1-x_2\rvert=d$ and $\lvert x_1+x_2\rvert/2$ is 
the desired isotropic distance, the lemma follows.
\end{proof}

\begin{lemma}\label{l-isotropic-osculating-Dupin-cyclide}
    Under the notation of Lemma~\ref{ip-curvature}, assume that the second principal curvature is constant along $c(t)$ and different from the first one. If $r'(t)\ne 0$, then $s'(t)$ is vertical (or zero).
\end{lemma}

\begin{proof} Fix a particular value of $t$, say, $t=0$, and assume that $t$ is sufficiently close to this value. 
It suffices to prove that the isotropic distance from $s(t)$ to $s(0)$ is $O(t^2)$. 

We do it by reduction to a planar problem. \bluevar{See Fig.~\ref{fig:2} to the middle.} %\mscomm{Add a figure?}
Since $r'(0)\ne 0$, by Lemma~\ref{ip-curvature} it follows that $c(0)$ is an elliptic isotropic circle. Performing an isotropic congruence of the form $z\mapsto z+px+qy$, we can take $c(0)$ to a horizontal circle. \bluevar{Take an isotropic plane $P$ passing through $s(0)$ and parallel to $s'(0)$.} %Fix some $t_0>0$ and the isotropic plane $P$ passing through $s(0)$ and $s(t_0)$. 
The section of $c(t)$ by the plane $P$ consists of two points $c_1(t)$ and $c_2(t)$. The section of $C$ coincides with the two curves $c_1(t)$ and $c_2(t)$ because the characteristics cover the surface. The section of $S(t)$ is a parabolic isotropic circle \bluevar{that is tangent to} the two curves at the points $c_1(t)$ and $c_2(t)$. Then the chord $c_1(t)c_2(t)$ forms equal replacing angles (of opposite signs) with the tangents at $c_1(t)$ and $c_2(t)$. %For $t=t_0$, 
\bluevar{It suffices to prove that the isotropic distance from}
the midpoint %of the top view 
of the chord \bluevar{to $s(0)$ is $O(t^2)$}.%is $s(t_0)$.

Let us estimate how those replacing angles and the midpoint change if we replace the two curves with their osculating isotropic circles  $D_1$ and $D_2$ (possibly degenerating into lines) at the points $c_1(0)$ and $c_2(0)$. Since the second principal curvature is constant along $c(0)$ and the plane of $c(0)$ is horizontal, it follows that $D_1$ and $D_2$ are symmetric with respect to the vertical line $L$ through $s(0)$. 
Since the first principal curvature is different, it follows that $D_1\ne D_2$. Let $p_1(t)$ and $p_2(t)$ be the points of $D_1$ and $D_2$ lying on the vertical lines through the points $c_1(t)$ and $c_2(t)$ respectively. The replacing angle between the tangents to the respective curves at the points $p_1(t)$ and $c_1(t)$ is $O(t^2)$ because $D_1$ is osculating. The same is true for the tangents at $p_2(t)$ and $c_2(t)$. The replacing angle between $p_1(t)p_2(t)$ and $c_1(t)c_2(t)$ is %$O(\bluevar{t^2})$. 
$O(t^3)$.
Thus the replacing angles $\alpha_1(t)$ and $\alpha_2(t)$ which $p_1(t)p_2(t)$ forms with $D_1$ and $D_2$ satisfy $\alpha_1(t)+\alpha_2(t)=O(t^2)$.  Now the result follows from Lemma~\ref{l-isotropic-circle-center} and its analog for lines $D_1$ and $D_2$.  
\end{proof}

\begin{proof}[Proof of Theorem~\ref{th-isotropic-weingarten}]
Consider an analytic isotropic channel Weingarten surface $C$ and the isotropic parabolic sphere $S(t)$ of radius $r(t)$ \bluevar{that is tangent to} $C$ along the characteristic $c(t)$. 
%Consider the following 2 cases. \khcomm{Mikhail will discuss the case $r^\prime(t) = 0$ for some $t$.  }
%Case 1: $r^\prime(t) =0$ for all $t$. Thus $r(t)$ is constant and $C$ is an isotropic pipe surface. 
%Moreover, since  $A^\prime(t) =  r^\prime(t) A(t)^2 = 0$, thus the characteristic $c(t)$ is the intersection of $S(t)$ with an isotropic plane, hence a parabolic isotropic circle. 
%Case 2: $r^\prime(t) \ne 0$ for some $t$. 
Then according to Lemma~\ref{ip-curvature}, the %characteristic $c(t)$ is an elliptic isotropic circle and an isotropic principal curvature line of $C$.  The 
isotropic principal curvature $\kappa_1$ along $c(t)$ is $1/r(t)$.  If the Weingarten relation is not $\kappa_1=\mathrm{const}$, then %$C$ is a pipe surface because $r(t)$ is a constant, and we are done. Otherwise, 
the other isotropic principal curvature $\kappa_2$ is constant along $c(t)$. 

If $r(t)=\mathrm{const}$ or $\kappa_1\equiv\kappa_2$, then we get an isotropic pipe surface or an isotropic sphere.
%and one of the isotropic principle curvatures is constant.
Otherwise, restrict the range of $t$ to an interval where $r'(t) \ne 0$ and $\kappa_1\ne\kappa_2$. By Lemma~\ref{l-isotropic-osculating-Dupin-cyclide}, the centers curve $s(t)$ has vertical tangent vector $s'(t)$ for all $t$.
Then $s(t)$ is contained in a vertical line, and by Lemma~\ref{cone:vertices} the surface $C$ is isotropic rotational.
%%%
%Consider a cone that \bluevar{is tangent to} $C$ along characteristic $c(t)$. %!!! 
%Denote by $v(t)$ the vertex of this cone. The curve, which consists of the cone vertices for each characteristic of $C$, is called \emph{cone vertices line}. %!!! 
%Let us prove that the cone vertices curve is tangent to the axis of $c(t)$ (vertical line through the center of the ellipse) at the points where $v'(t)\ne 0$. \mscomm{Mikhail will say something why $v'(t)$ at all exists.} By an appropriate isotropic congruence, bring $c(t)$ to a horizontal circle.  Consider the family of parabolic isotropic spheres of radius $1/\kappa_2$ \bluevar{is tangent to} $C$ along $c(t)$ (they degenerate into planes for $\kappa_2=0$). They all \bluevar{are tangent to} the isotropic sphere $S(t)$, hence are obtained by rotations of one parabolic isotropic sphere (or plane) about the axis of $c(t)$. Let $D$ be the envelope of the family; it is obtained by rotations of one parabolic isotropic circle (or line) about the same axis. By construction, $C$ and $D$ have a second-order contact along $c(t)$. Then the cone vertices curve $v(t)$ \bluevar{is tangent to} the axis of $c(t)$. \mscomm{Mikhail should prove this in the appendix.}  Since $v(t)$ has a vertical tangent everywhere, then it is a vertical line. Therefore by Lemma \ref{cone:vertices} the surface $C$ is isotropic rotational. 

Finally, suppose that $C$ has a constant ratio of the isotropic principal curvatures. Then it is Weingarten. As we have proved, it is an isotropic rotational or pipe surface. In the latter case, $\kappa_1 =\mathrm{const}$ by Lemma~\ref{ip-curvature}, hence $\kappa_2 = \mathrm{const}$, and $C$ is a subset of a parabolic rotational surface by Theorem~\ref{l-both-curvatures-constant}. 
%Since $s(t)$ has second-order contact with a straight line for all $t$, it is a straight line and $C$ is a rotational surface. 
   %Then our channel surface could be found by solving the following system of ODE's:
    %$$\begin{cases}\label{envelope}
     %   A(t)\left(x^2+y^2\right)+B(t) x+C(t) y+D(t) -2z = 0\\
    %    \dot{A}(t)\left(x^2+y^2\right)+\dot{B}(t) x+\dot{C}(t) y+\dot{D}(t) = 0.
   % \end{cases}$$
%We have two cases:
%\begin{itemize}
    %\item Case 1: $A(t)$ is a constant. Our spheres have the same radii. Therefore it is an isotropic pipe surface. Moreover, the second equation of \ref{envelope} is an isotropic plane. Thus our channel surface consists  of congruent isotropic parabolic circles.
    %\item Case 2: 
    %$A(t)$ is not a constant, then the second equation of \ref{envelope} is an isotropic sphere of the \bluevar{cylindrical} type.
 %Thus characteristics of our channel surface $C$ are isotropic elliptical circles. By Lemma \ref{second-order}, we know that  for each characteristic $c(t)$  there is an osculating isotropic Dupin cyclide $D(t)$ which has second-order contact with $C$ along $c(t)$. Let $S(t)$ be the isotropic sphere that \bluevar{is tangent to} $C$ along a characteristic $c(t)$. Consider the generating family for $D(t)$ that does not contain $S(t)$. According to the Lemma \ref{sec-family}, all the isotropic spheres or planes of this family are congruent.  Then after applying Lemma \ref{rot:trans} to $D(t)$ we get that $D(t)$ is an isotropic rotational or an isotropic translational surface. 
%\end{itemize}
\end{proof}

\section{Ruled surfaces} \label{sec-ruled}

%------------------------------------------------------------------
%\subsection{Ruled surfaces}
%-------------------------------------------------------------------

Let us now turn to ruled surfaces. In Euclidean geometry we will not encounter a new surface, but in isotropic geometry there is a non-trivial CRPC ruled surface. 

Our arguments are based on line geometry. For the concepts used in the following, we refer to \cite{pottwall:2001}. The methods for ruled surfaces and channel surfaces are actually related via Lie's line-sphere correspondence. We again restrict ourselves to analytic surfaces (with nonvanishing Gaussian curvature); then the rulings form an analytic family because the direction of a ruling is asymptotic. %For simplicity, we restrict ourselves to analytic surfaces. This is not really a restriction; all sufficiently smooth CRPC surfaces are analytic by the Petrowsky theorem \cite[p.3--4]{Petrowsky-39} and the rulings form an analytic family because the direction of a ruling is an asymptotic direction.

An \emph{analytic ruled surface} is an analytic surface covered by an analytic family of line segments. The lines containing the segments are the \emph{rulings}. 

\subsection{Euclidean ruled CRPC surfaces.} We start with the Euclidean
case and show the following result.

\begin{prop} \label{p-Euclidean-ruled}
The only ruled surfaces with a constant nonzero ratio of principal curvatures %CRPC surfaces in Euclidean 3-space 
are the ruled minimal
surfaces, i.e., helicoids. 
\end{prop}

\begin{proof} %Since the definition of CRPC surfaces excludes vanishing principal curvatures, 
A ruled CRPC surface must be skew (without torsal rulings) and its asymptotic curves should intersect under a constant angle $\gamma$ \cite[Section~2.1]{helicalcrpc}.
One family of asymptotic curves is the rulings. Let us fix a ruling $R$ and consider the other asymptotic tangents $A(p)$ (different from $R$) at all points $p \in R$. They form a quadric $L(R)$ %a regulus of 
(the so-called \emph{Lie quadric} of $R$; see e.g. \cite[Corollary~5.1.10]{pottwall:2001}). If the angle between $R$ and $A(p)$ is constant, it can only be a right angle (so that $L(R)$ is a right hyperbolic paraboloid). Indeed, if the angle $\gamma$ is not a right one,
then the ideal points of the lines $A(p)$ form a conic $c_\omega$ in the ideal plane $\omega$ (ideal conic of a rotational cone with axis
$R$) which does not contain the ideal point $R_\omega$ of $R$. 
However, $R_\omega$ and $c_\omega$ should lie in the same curve $L(R)\cap \omega$ %at infinity 
(conic or pair of lines)%of the Lie quadric $L(R)$
, which is not possible. 
So, $\gamma=\pi/2$ and our surface is a skew ruled minimal surface,
i.e. a helicoid by the Catalan theorem. 
\end{proof}
\begin{remark}
Here we used the famous Catalan theorem stating that the only ruled minimal surfaces are helicoids and planes. Remarkably, in Lemmas~\ref{l-Catalan}--\ref{l-CRPC-conoids} below we actually obtain a line-geometric proof of this classical result. 
Indeed, we have just shown %(in the proof of Proposition~\ref{p-Euclidean-ruled}) 
that the Lie quadric of each ruling must be a right hyperbolic paraboloid. Then Lemma~\ref{l-Catalan} and its proof remain true in Euclidean geometry. Then without loss of generality, all the rulings are parallel to the plane $z=0$. Since the asymptotic directions and the rulings are orthogonal, their top views are also orthogonal, and our Euclidean minimal surface is an \emph{isotropic} minimal surface as well. The Catalan theorem now reduces to Lemmas~\ref{l-conoid}--\ref{l-CRPC-conoids}, where the case of a hyperbolic paraboloid is easily excluded.
\end{remark}
%It remains to apply Lemmas~\ref{l-conoid}--\ref{l-CRPC-conoids} and rule out the hyperbolic paraboloid by a direct computation. 

The proof of Proposition~\ref{p-Euclidean-ruled} already indicates that there is hope to get a ruled
CRPC surface to a constant $a \ne -1$ in isotropic geometry. This is what we will now pursue.

\subsection{Isotropic ruled CRPC surfaces.}

\begin{theorem}\label{thm-ruled} (See Fig.~\ref{fig:ruled})
An admissible analytic ruled surface has a constant ratio $a< 0$ of isotropic principal curvatures if and only if it is isotropic similar to a subset of 
either the hyperbolic paraboloid
$
    z = x^2+ay^2,
$    
or the helicoid %$x=y\tan z$ (if $a=-1$) \mscomm{(fix equation!)}, 
\begin{align}
%\label{eq-helicoid}\label{paraboloid}
%r(u,v)&=\begin{bmatrix}
%           u \\
%           v \\
%           u^2+av^2
%         \end{bmatrix},\\
%\intertext{or the helicoid}
\label{eq-helicoid}
r(u,v)&=\begin{bmatrix}
           u\cos v \\
           u\sin v \\
           v
         \end{bmatrix},
&\text{if } a= -1,\\
%\left(u\cos v,\quad u\sin v,\quad v \right), & 
\intertext{or the surface}
\label{eq-spiral}
r_a(u,v)&=\begin{bmatrix}
           u\cos v \\
           u\sin v \\
           \exp\left(\frac{a+1}{\sqrt{\lvert a \rvert}}v\right)
         \end{bmatrix},
&\text{if } a \ne -1.
%\left(u\cos v,\quad u\sin v,\quad \exp\left(\frac{a+1}{\sqrt{\lvert a \rvert}}v\right) \right), & &\text{if } a\ne -1.   
\end{align}
%\end{align*}
\end{theorem}
%\khcomm{It seems that in (5) $\exp\left(\frac{a+1}{\sqrt{|a|}}u\right) $ }

\begin{proof}
This follows directly from Lemmas~\ref{l-Catalan}--\ref{l-CRPC-conoids} below (which themselves rely on standard Lemmas~\ref{l-envelope}--\ref{l-tangents} from Appendix~\ref{sec-appendix}).
\end{proof}

\begin{figure}[htbp]
\hfill
\begin{overpic}[width=0.28\textwidth]{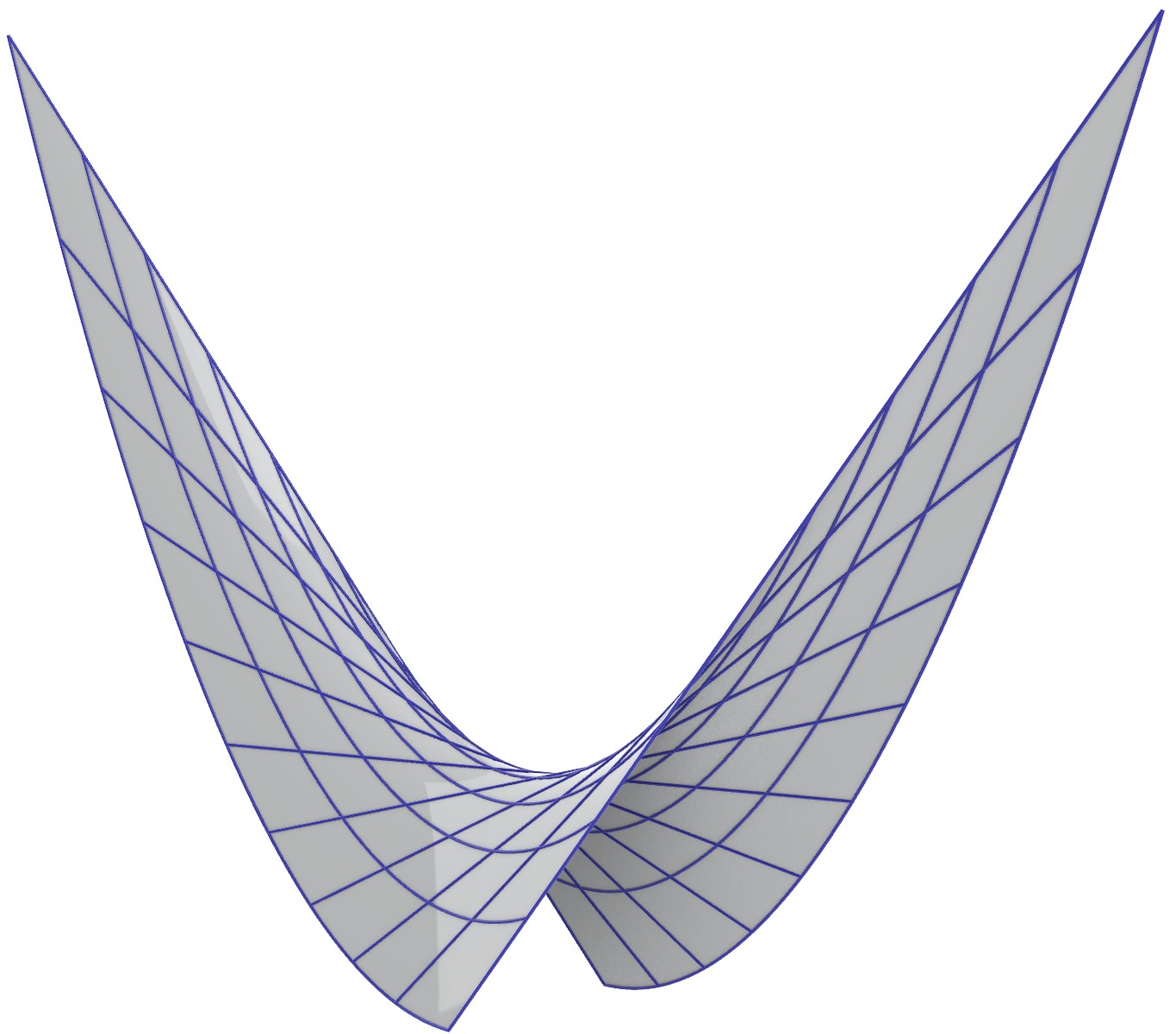}
    %\put(44,80){\contour{white}{$\widetilde{z}$}}
\end{overpic}
\hfill
\begin{overpic}[width=0.15\textwidth]{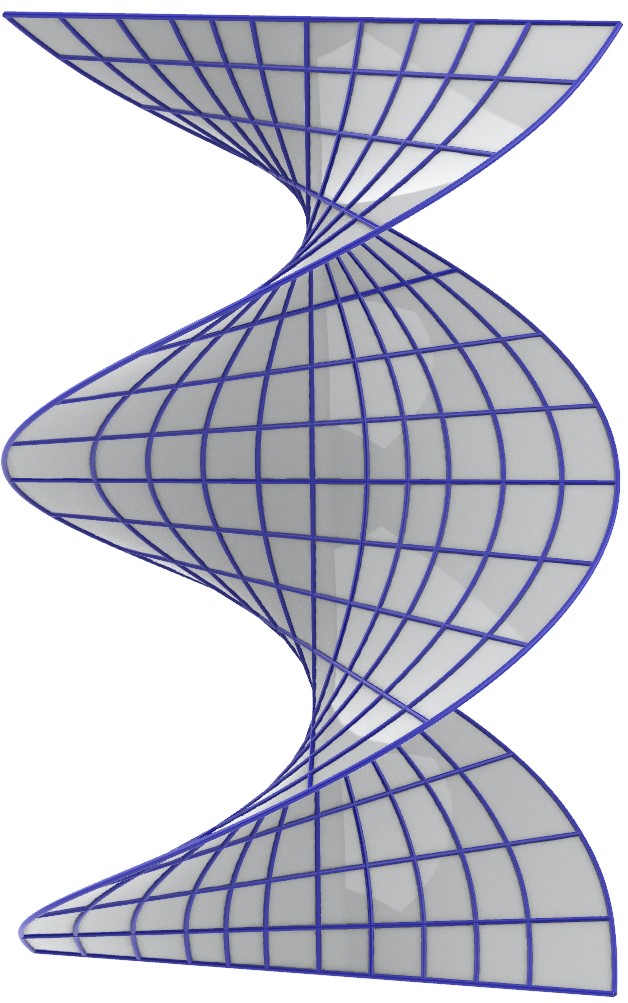}
    %\put(44,80){\contour{white}{$\widetilde{z}$}}
\end{overpic}
\hfill
\begin{overpic}[width=0.22\textwidth]{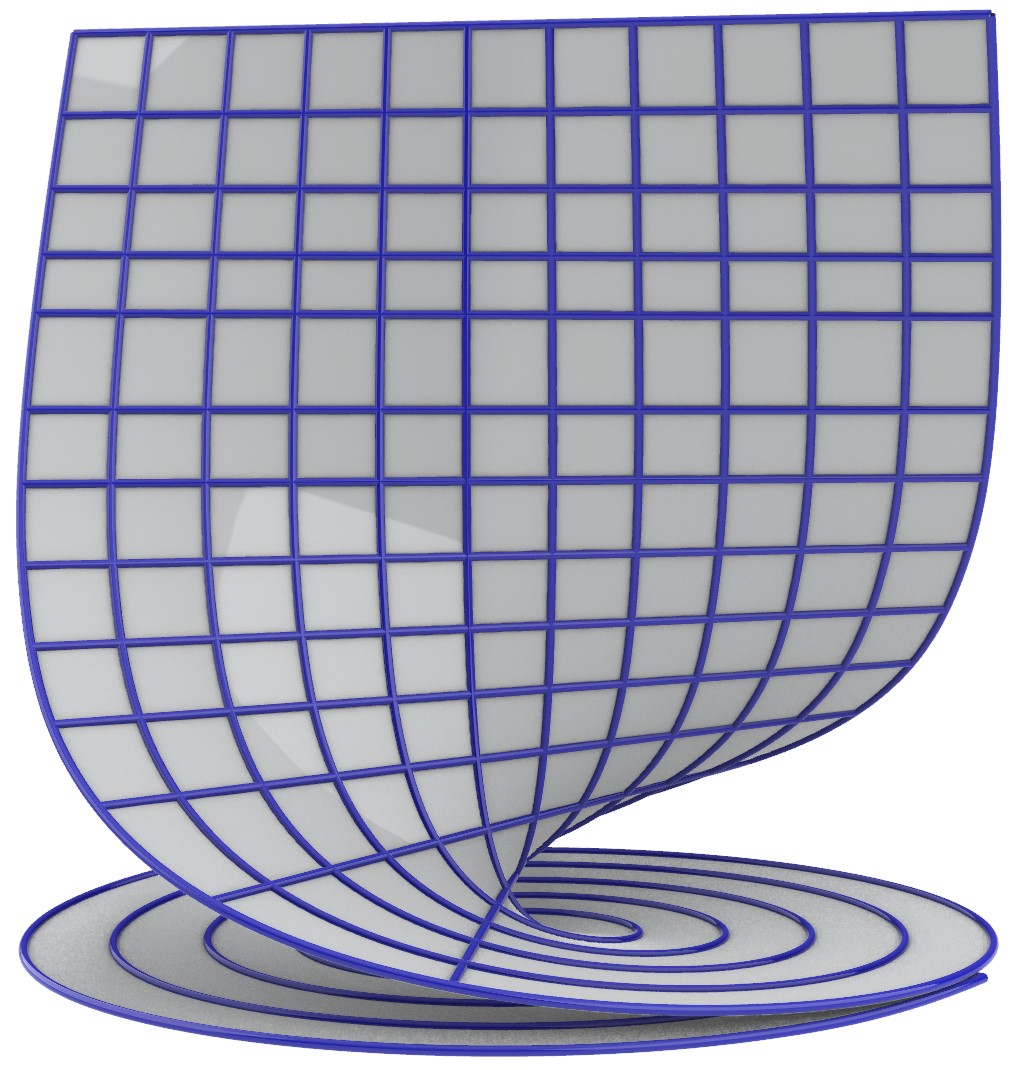}
    %\put(44,80){\contour{white}{$\widetilde{z}$}}
\end{overpic}
  \hfill{}
    \caption{Ruled isotropic CRPC surfaces (from the left to the right): 
    a hyperbolic paraboloid, %\ref{paraboloid}
    helicoid~\eqref{eq-helicoid}, and spiral ruled surface~\eqref{eq-spiral}.}
    \label{fig:ruled}
\end{figure}

\begin{lemma} \label{l-Catalan} An admissible analytic ruled surface with a constant nonzero ratio of isotropic %or Euclidean 
principal curvatures is a \emph{conoidal} (or \emph{Catalan}) surface, i.e. all the rulings are parallel to one plane.
\end{lemma}

\begin{proof}
Let us show that the Lie quadric of each ruling $R$ (see \cite[Corollary~5.1.10]{pottwall:2001})  is a hyperbolic paraboloid. Since the surface is admissible, the top view $R'$ of $R$ is not a point.
Since isotropic angles are seen as Euclidean angles in the top view, the top views $A(p)'$ of the asymptotic tangents $A(p)$ at points $p$ of  $R$ must form the same angle $\gamma$ with  $R'$. Hence the lines $A'(p)$ are parallel to each other, implying that the quadric formed by the lines $A(p)$  is %a regulus $R_A$ on 
a hyperbolic paraboloid. 
%, because which has to be the Lie quadric $L(R)$. 

%Then one line of the (extended) family $A(p)$ is a line at infinity $A_\omega$. 
Then the second family of rulings (distinct from $A(p)$) of the quadric intersects the ideal plane $\omega$ by a line $A_\omega$. Since the Lie quadric has a second-order contact with our surface \cite[Theorem~5.1.9]{pottwall:2001}, it follows that  
$A_\omega$ has a second-order contact with the ideal curve $s_\omega$ %(possibly degenerating to a point) 
formed by the ideal points of the rulings of our surface \cite[Proposition~5.1.11]{pottwall:2001}. 
As a curve that has an osculating straight line at each point, the curve $s_\omega$ is itself a straight line %or a point 
and 
therefore our surface must be a conoidal surface \bluevar{($s_\omega$ cannot degenerate to a point as the isotropic Gaussian curvature $K\ne 0$).} 
\end{proof}

\begin{lemma} \label{l-conoid}
An admissible analytic conoidal surface with a constant nonzero ratio of isotropic %or Euclidean \usepackage{}
principal curvatures is a \emph{conoid}, i.e.,~all the rulings are parallel to a fixed plane and intersect a fixed line. The fixed line is either vertical or belongs to another family of rulings. %of the surface.
In the latter case, the surface is a hyperbolic paraboloid with a vertical axis.
\end{lemma}

%The proof relies on two well-known Lemmas~\ref{l-envelope}--\ref{l-tangents}, which we prove below to keep the work self-contained.

\begin{proof} By the analyticity, it suffices to prove the lemma for an arbitrarily small part of our surface. Thus in what follows we freely restrict and extend our surface.

Let $R_t$ be the analytic family of the rulings of the surface. Since \bluevar{$K\ne 0$},
%the ratio of the principal curvatures is nonzero, 
it follows that %the Gaussian curvature is negative and 
there are no torsal rulings; in particular, the surface is not a plane. 
%By Lemma~\ref{l-Catalan} all $R_t$ are parallel to one plane. 

Let $R'_t$ be the top view of $R_t$. By Lemma~\ref{l-envelope}, one of the following cases~(i)--(iii) holds, after we restrict $t$ to a smaller interval. %Consider cases (i)--(iii) separately. 

Case (i):  %the top views of 
all $R'_t$ have a common point. Then all $R_t$ intersect one 
vertical line and the lemma is proved.

Case (ii): %the top views of 
all $R'_t$ are parallel. Then due to the fixed angle between the asymptotic directions in the top view, the second family of
asymptotic curves also appears as parallel lines in the top view. 
Hence those curves lie in the isotropic planes. However, at non-inflection points of the asymptotic curves the osculating planes are the tangent planes of the surface \cite[ Page 28]{strubecker1969differentialgeometrie}. Thus these tangent planes needed to be isotropic, which is not possible for an admissible surface. Hence, there are no non-inflection points, both families of asymptotic curves are straight lines, and our surface is a hyperbolic paraboloid with a vertical axis.

Case (iii): %the top views of 
all $R'_t$ touch one curve $e$ (envelope). 
Let us show that this case is actually impossible.

For this purpose, we are going to extend our surface to reach the envelope. 
Let $h(t)$ be the Euclidean distance from the ruling $R_t$ to the \bluevar{fixed} plane parallel to \bluevar{all the rulings}. We have $h(t)\ne\mathrm{const}$ because our surface is not a plane. By continuity, there is an interval~$I$ where $h'(t)$ has a constant sign. Then the union $\bigcup_{t\in I}R_t$ is an analytic surface containing a part of the initial surface. The resulting surface is not admissible: By Lemma~\ref{l-envelope}  the envelope forms a part of the boundary of the top view of the surface, hence the tangent planes are isotropic at the points with the top views lying on the envelope.

Switch to the new surface $\bigcup_{t\in I}R_t$. By analyticity, it still has a constant ratio of isotropic principal curvatures (at the admissible points). % where the tangent plane is non-isotropic). 
The Gaussian curvature still vanishes nowhere because there are no torsal rulings. Thus the whole surface, including non-admissible points, is covered by two analytic families of asymptotic curves (recall that the asymptotic curves are the same in Euclidean and isotropic geometry, hence they acquire no singularities at the non-admissible points). One of the families consists of the rulings, and at admissible points, the other one crosses them under constant angle $\gamma$ in the top view. 

Now we \bluevar{prove that there is }%find 
 an asymptotic curve $\alpha$ containing a non-admissible point $O$ but not entirely consisting of non-admissible points. See Fig.~\ref{fig:2} to the right.
The top view $\alpha'$ needs to have a common point with the envelope~$e$. %(but should not be contained in the latter).
%such that  %Start with an arbitrary ruling $R_a$ with $a\in I$. 
Take $a,b\in I$ close enough so that the angle between $R'_a$ and $R'_t$ is an increasing function in $t$ on $[a,b]$ not exceeding $\pi-\gamma$. Let $A'$ and $B'$ be the tangency points of $R'_a$ and $R'_b$ with the envelope, and $C\in R_a$ be the point with the top view $C':=R'_a\cap R'_b$. Then the second asymptotic curve $\alpha$ through $C$ is the required one. Indeed, the angle between its top view $\alpha'$ and $R'_a$ equals $\gamma$, hence $\alpha'$ enters the curvelinear triangle $A'B'C'$ formed by the arc $AB$ of the envelope and two straight line segments $B'C'$ and $C'A'$. Since
the angle between $\alpha'$ and $R'_t$ is constant and
the angle between $R'_a$ and $R'_t$ is increasing, the curve $\alpha'$ cannot reach the sides $B'C'$ and $C'A'$ as long  as it remains smooth. Since the asymptotic curves extend till the surface boundary, it follows that $\alpha'$ has a common point with the envelope
and $\alpha$ has a non-admissible point $O$, as required.
 
%Let $O\in\alpha$ be a point with the top view on the envelope. 
%Then the tangent plane to the surface at $O$ is isotropic. 
Since $\alpha$ does not entirely consist of non-admissible points, it follows that at the other close enough points $P\ne O$ of $\alpha$,
the tangents cross the rulings under constant angle $\gamma$ in the top view. By the continuity, the limit $L'$ of the top views of the tangents crosses the top view of the ruling through $O$ under angle $\gamma$. 

Now let us prove that $L'$ must be the top view of the ruling through $O$, and thus get a contradiction. %Consider the tangent of $\alpha$ at $O$.

If the tangent of $\alpha$ at $O$ is not vertical, then $L'$ coincides with the top view of the tangent, hence with the top view of the tangent plane at $O$, hence with the top view of the ruling through $O$. 

If the tangent of $\alpha$ at $O$ is vertical, then by Lemma~\ref{l-tangents} the limit $L'$  coincides with the top view of the limit of the osculating planes at the points of $\alpha$. But the osculating plane of an asymptotic curve at non-inflection points is the tangent plane, and the points close enough to $O$ are non-inflection. Hence we again obtain the top view of the tangent plane at $O$, equal to the top view of the ruling through $O$. 

This contradiction shows that case (iii) is impossible, completing the proof. % of the lemma.
\end{proof}

\begin{lemma} \label{l-CRPC-conoids} An admissible conoid has a constant ratio $a\ne 0$ of isotropic principal curvatures if and only if it is isotropic similar to a subset of one of the surfaces $z=x^2+ay^2$, \eqref{eq-helicoid}, or \eqref{eq-spiral} \bluevar{from Theorem~\ref{thm-ruled}}.
%listed in Theorem~\ref{thm-ruled}.
\end{lemma}

\begin{proof} Let all the rulings of the conoid be parallel to a fixed plane $\alpha$ and intersect a fixed line $l$. We have two possibilities indicated in Lemma~\ref{l-conoid}.

If $l\not\parallel Oz$ then by Lemma~\ref{l-conoid} the surface is a hyperbolic paraboloid with a vertical axis. By Example~\ref{ex-paraboloid} it is isotropic similar to a subset of the paraboloid $z = x^2+ay^2$.

If $l\parallel Oz$ then performing an isotropic similarity, we can take $l$ to the $z$-axis and $\alpha$ to the plane $z=0$. The resulting conoid can be parameterized as 
$$ r(u,v) = (u \cos v, u \sin v, h(v))$$
for some smooth function $h(v)$. The asymptotic curves (distinct from the rulings) are characterized by the differential equation 
\cite[p.~137]{kruppa:1957}
$$ u(v)^2 = b h'(v)$$
for some constant $b$. The top views of the asymptotic curves
must intersect the lines through the origin under the constant angle $\gamma$, where $\cot^2(\gamma/2)=\lvert a\rvert$. We now have to distinguish whether this angle is  right one or not.
%

%\noindent \emph{Case A2.} 
If $\gamma$ is a right angle, then the top views of asymptotic curves must be concentric circles, leading to $u(v)=\mathrm{const}$ and $h(v)=v$ up to a translation and a scaling along the $z$-axis. We get helicoid~\eqref{eq-helicoid}. 
%The generating helical motion deserves this name both in isotropic and Euclidean geometry (as a subgroup of both motion groups).

%\noindent \emph{Case A1.} 
If $\gamma$ is not a right angle, then the top views %of the asymptotic curves 
must be logarithmic spirals
$ u(v)= c e^{\pm v\cot{\gamma}}$ 
for some constant $c$. This yields 
$ h(v)= e^{\pm 2v\cot{\gamma}} 
%, \quad C=\frac{C_1^2}{2qC_0}.
$
up to a translation and a scaling along the $z$-axis.
Changing the signs of $v$ and $y$, if necessary, we arrive at~\eqref{eq-spiral}.
\end{proof}

\subsection{Geometry of the surfaces and their characteristic curves}

Ruled CRPC surface~\eqref{eq-spiral} is a  \emph{spiral surface} (see \cite{wunderlich:DG2}), generated by a one-parameter group of (Euclidean and isotropic) similarities, composed of rotations about the $z$-axis and central similarities with center at the origin. The paths of that motion are cylindro-conical spirals which appear in the
top view as logarithmic spirals with polar equation \bluevar{$r(v)= c\cdot e^{2v \cot{\gamma}}$ for some constant $c$}. However, the
asymptotic curves (different from rulings) are not such paths. They are
expressed as 
$$ c(v)= (c e^{v \cot{\gamma}} \cos v,  c e^{v\cot{\gamma}} \sin v, e^{2v \cot{\gamma}}), $$
and are also obtained by intersecting the ruled surface with isotropic spheres (of variable isotropic radius $c^2\bluevar{/2}$),
\begin{equation}\label{eq-iso-sphere}
\bluevar{z}= \frac{1}{c^2}(x^2+y^2). 
\end{equation}
On these, the curves $c(v)$ are isotropic loxodromes. Their tangents are contained in a linear line complex with the $z$-axis as the axis. 
This is related to another non-Euclidean interpretation of isotropic CRPC ruled surface~\eqref{eq-spiral} and its \bluevar{asymptotic} %characteristic
curves $c(v)$: One can \bluevar{view} %define 
one of the paraboloids~\eqref{eq-iso-sphere} as absolute quadric of the projective
model of hyperbolic 3-space. There, \eqref{eq-spiral} is a helicoid and the asymptotic curves $c(v)$ are paths of a hyperbolic helical motion (one- parameter \bluevar{subgroup} of the group of hyperbolic congruence transformations). It is also well known and easy to see that the hyperbolic helices are projectively equivalent to Euclidean spherical loxodromes \cite{Strubecker1931}. \mscomm{MS: It is better to make the reference more precise, maybe cite a particular Satz. The word 'loxodrome' appears in Strubecker's work only on page 65, and cannot find anything on the projective equivalence around.}
%%%% 
% the idea is very simple. Let us take 2 coaxial quadrics of revolution, say a sphere S and an ellipsoid \Omega enclosing the sphere. Let \Omega be the absolute quadric of a projective model of hyperbolic space. Then S is the distance surface of a line  A  (the axis of revolution). During a helical motion about A, S is mapped onto itself, and planes through A map to planes through A. Rulings of S (isotropic lines) are mapped to rulings. Hence, the cross ratio of the following four lines is unchanged: tangent to helical path, tangent to intersetion with a plane through A, and the 2 rulings of S. By Laguerre's angle formula, this shows that the Euclidean angle between the first two lines is constant, hence the helical path is a Euclidean loxodrome.
%%%%

In summary, we have proved the following result.

\begin{prop}
%The only ruled CRPC surfaces in isotropic 3-space are helicoids with an isotropic axis, hyperbolic  paraboloids with an isotropic axis, and 
The asymptotic curves of spiral ruled surfaces~\eqref{eq-spiral}, distinct from the rulings, lie on isotropic spheres. Viewing one of these isotropic spheres as the absolute quadric
of the projective model of hyperbolic geometry, these surfaces are helicoids and the asymptotic curves are helical paths. The latter are projectively equivalent to Euclidean spherical loxodromes.
\end{prop}

\section{Helical surfaces} \label{sec-helical}

A \emph{helical motion  through the angle $\phi$ about the $z$-axis with pitch $h$} is the composition of the rotation through the angle $\phi$ about the $z$-axis and the translation by $h\phi$ in the $z$-direction. The helical motion is also an isotropic congruence. A surface invariant under the helical motions for fixed $h$ and all $\phi$ is called \emph{helical with pitch $h$}. In particular, for $h=0$ we get a rotational surface.

%Take care of $a\mapsto 1/a$ !!!}

\begin{theorem}\label{helical} (See Fig.~\ref{fig:helical})
An admissible helical surface with nonzero pitch has a constant ratio $a\ne 0$ of isotropic principal curvatures, if and only if it is isotropic similar to a subset of one of the surfaces
\begin{align}\label{helicalane-1}
    r_{a}(u,v) &=\begin{bmatrix}
           \cos v\left(\cos u\sin^a u\right)^{-\frac{1}{a+1}} \\
           \sin v\left(\cos u\sin^a u\right)^{-\frac{1}{a+1}} \\
           \bluevar{v+u+ \frac{1}{a^2-1}{\tan u}+\frac{a^2}{a^2-1}\cot u}
         \end{bmatrix},
& &\text{if } a \neq \pm 1,\\
    %\left(\cos v\left(\cos u\sin^a u\right)^{-\frac{1}{a+1}}, \quad
%\sin v\left(\cos u\sin^a u\right)^{-\frac{1}{a+1}},  \quad  
%v+u+\cot 2u+\frac{a^2+1}{a^2-1}\csc 2u\right), 
%\left(\frac{\cos v}{(\cos \frac{u}{2}\sin^a\frac{u}{2})^{\frac{1}{a+1}}}, \quad
%\frac{\sin v}{(\cos \frac{u}{2}\sin^a\frac{u}{2})^{\frac{1}{a+1}}},  \quad  
%v+\frac{a^2+1}{a^2-1}\csc u+\cot u+\frac{1}{2}u\right), 
%& &\text{if } a \neq \pm 1,\\
\label{helical-1}
r_{c}(u,v)&=\begin{bmatrix}
           u\cos v \\
           u\sin v \\
           c\log u + v
         \end{bmatrix},
         %\left(u\cos v,\quad u\sin v, \quad c\log u + v\right),
& &\text{if } a = -1,
\end{align}
where  $c$ is an arbitrary constant
%. 
\bluevar{%In~\eqref{helicalane-1}--\eqref{helical-1}, the variable 
and $v$ runs through $\mathbb{R}$.
In~\eqref{helical-1}, $u$ runs through $(0,+\infty)$.
In~\eqref{helicalane-1}, $u$ runs through a subinterval of $(0,\pi/2)$, where $\tan^2 u\ne a$.}
\end{theorem}

%\begin{remark}
%In~\eqref{helicalane-1}--\eqref{helical-1}, the variable $v$ runs through $\mathbb{R}$.In~\eqref{helical-1}, $u$ runs through $(0,+\infty)$.In~\eqref{helicalane-1}, $u$ runs through a subinterval of $(0,\pi/2)$, where $\tan^2 u\ne a$. 
%\end{remark}
%\mscomm{Check if the surface is admissible for those values!}
%\mscomm{(This was obtained by the change of variables 
%$\tanh (u/2)=\tan (u_{\mathrm{new}}-\pi/4)$ from the old expression
%with $q:=(a-1)/(2(a+1))$:
%$$
%    r_{a}(u,v) &=
%\left(e^{-q u}\sqrt{\cosh u} \cdot \cos v, e^{-q u}\sqrt{\cosh u} \cdot\sin v,    (4q^2+1) \cosh u+4q (\arctan\tanh u/2-\sinh u)+ 4qv\right), \text{if } a \neq \pm 1.
%$$
%}

\begin{figure}[htbp]
\hfill
\begin{overpic}[width=0.17\textwidth]{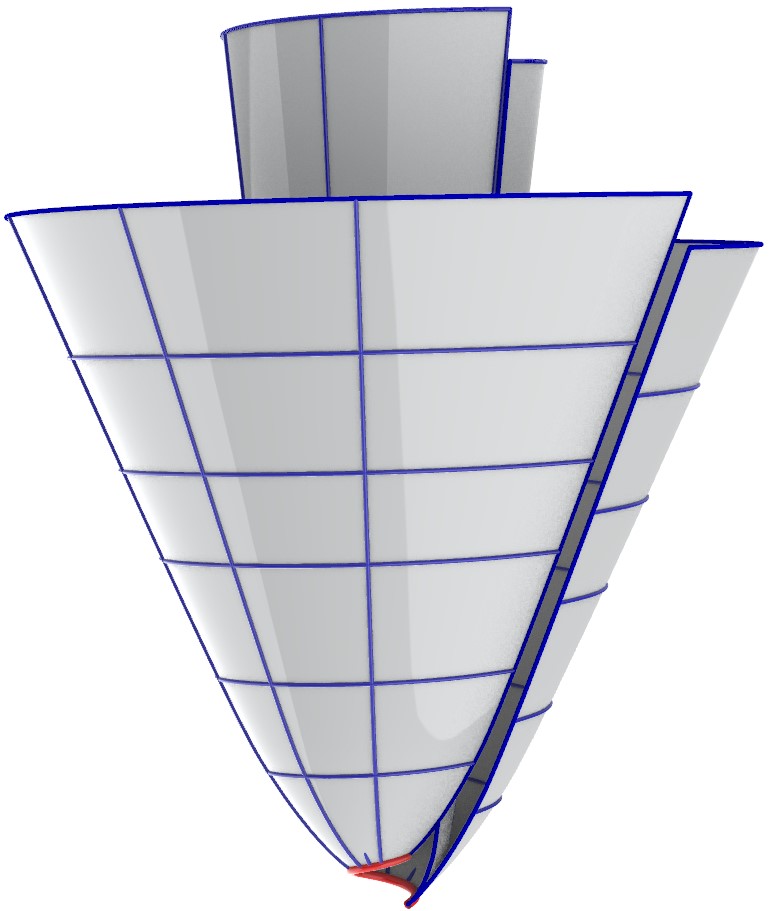}
    %\put(44,80){\contour{white}{$\widetilde{z}$}}
\end{overpic}
\hfill
\begin{overpic}[width=0.17\textwidth]{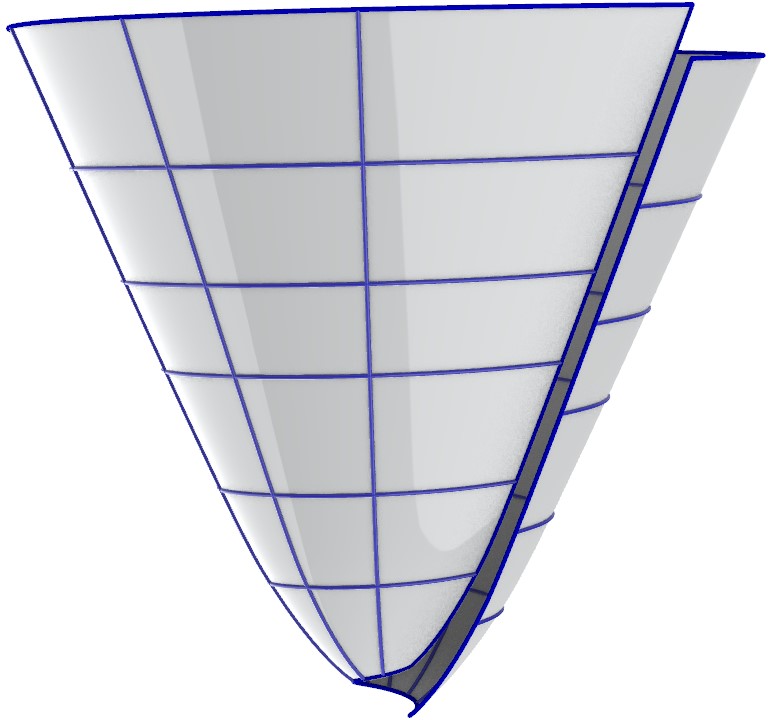}
    %\put(44,80){\contour{white}{$\widetilde{z}$}}
\end{overpic}
\hfill
\begin{overpic}[width=0.075\textwidth]{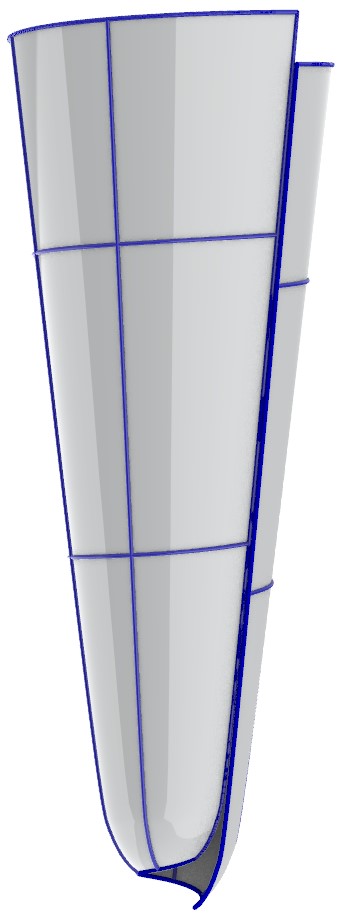}
    %\put(44,80){\contour{white}{$\widetilde{z}$}}
\end{overpic}
\hfill
\begin{overpic}[width=0.25\textwidth]{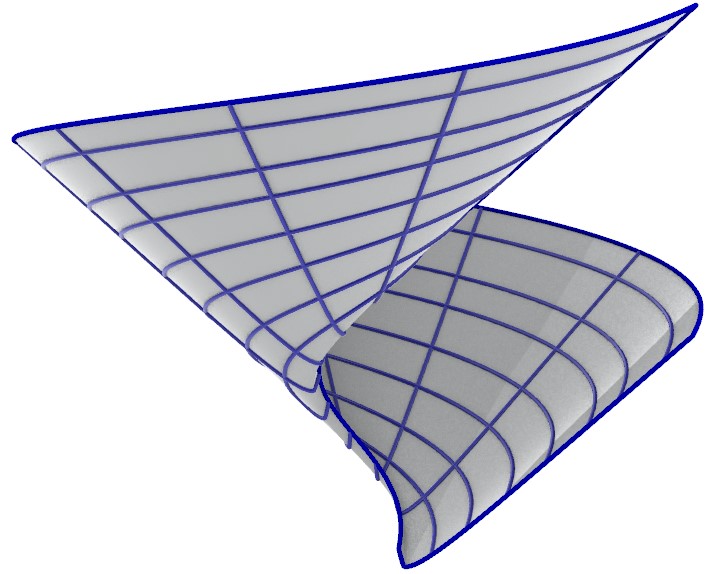}
    %\put(44,80){\contour{white}{$\widetilde{z}$}}
\end{overpic}
\begin{overpic}[width=0.25\textwidth]{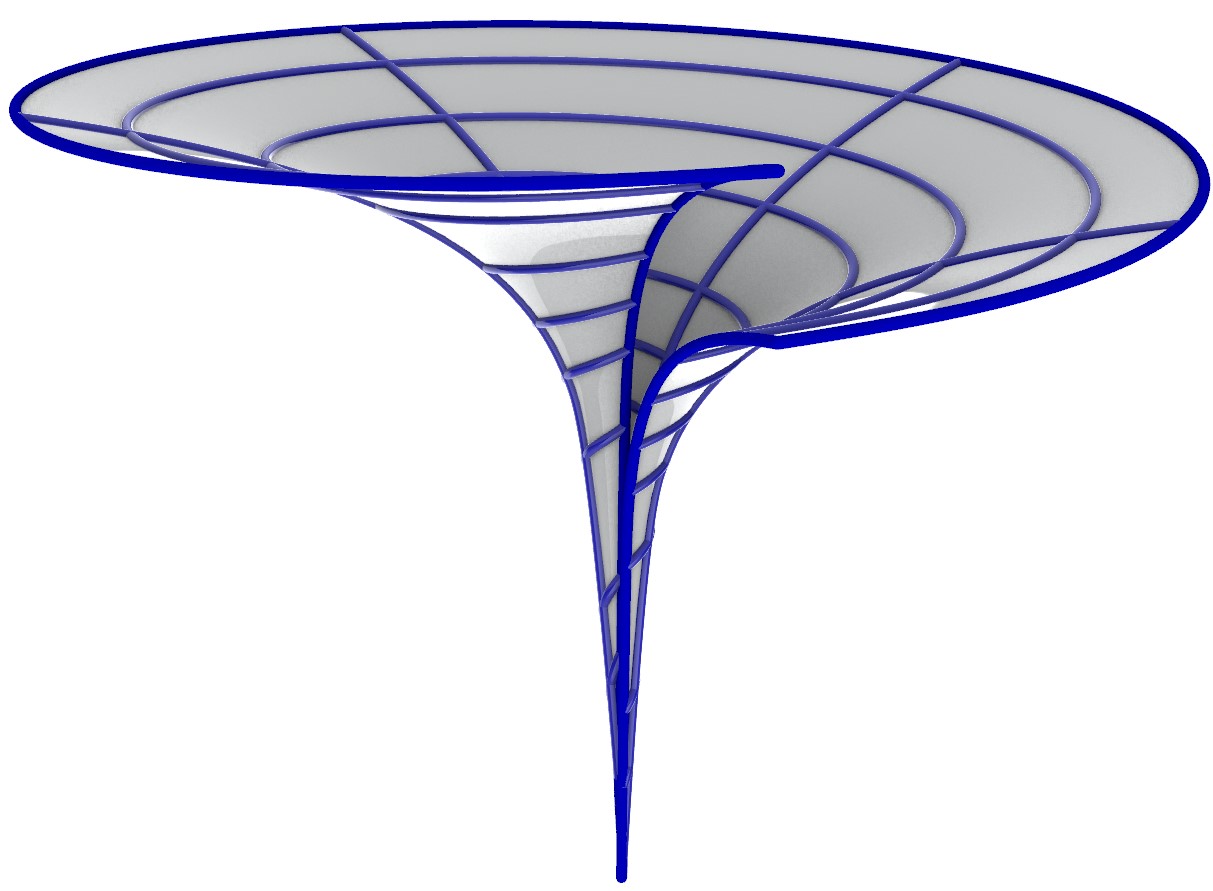}
    %\put(44,80){\contour{white}{$\widetilde{z}$}}
\end{overpic}
\hfill{}
\caption{Helical isotropic CRPC surfaces (from the left to the right):
surface~\eqref{helicalane-1} for $a>0$; its outer part; its inner part; surface~\eqref{helicalane-1} for $a<0$; surface~\eqref{helical-1}. The singular curve $\tan^2(u) = a$ of the leftmost surface is depicted in red; it splits the surface into the parts $\tan^2(u) > a$ and $\tan^2(u) < a$ shown separately.
}
%\mscomm{Restrict to a smaller part of the first surface so that the blue part becomes (relatively) larger and more visible?}}
\label{fig:helical}
\end{figure}

\begin{proof}
    Since the pitch is nonzero, it can be set to $1$ by appropriate scaling along the $z$-axis. Take the section of the surface by a half-plane bounded by the $z$-axis. Since the surface is admissible, the section is a disjoint union of smooth curves without vertical tangents. Then one of those curves can be parameterized as $z=f(\sqrt{x^2+y^2})$ for some smooth function $f(u)$ defined in an interval inside the ray $u>0$. Hence, up to rotation about the $z$-axis, our surface can be parameterized as 
    \begin{equation}\label{eq-helical-parametrization}
        r(u,v) = (u\cos v, u\sin v, f(u)+v).
    \end{equation}
      %where $h \ne 0$. 
    Then the isotropic Gaussian and mean curvatures are (see \cite[Eq.~(4.4)]{karacan})
    \begin{equation} \label{eq-K-helical}
        K = \frac{u^3f''(u)f'(u)-1}{u^4}, \quad H = \frac{f'(u)+uf''(u)}{2u}.
    \end{equation}
%Here  $f''(u)f'(u)-1\neq 0$ because $a\ne 0$. % and the surface is admissible. 
Then the equation ${H^2}/{K} = (a+1)^2/(4a)$ is equivalent to
    \begin{equation}\label{eq-helical}
        au^2(f'(u)+uf''(u))^2 = (a+1)^2(u^3f''(u)f'(u)-1).
    \end{equation}
    
First let us solve the equation for $a= - 1$. In this case, $f''(u)u+f'(u) =0$. Hence $f(u) = c\log u + c_{1}$ for some constants $c$ and $c_{1}$. By performing the isotropic similarity $z\mapsto z-c_{1}$ we bring  our surface to form~\eqref{helical-1}.
%(up to renaming the parameter $C_{1}/h$ to $c$). 

Assume further that $a\ne - 1$. %Denote by $k(u)$ the $uf'(u)$, then $f''(u) =(uk'(u)-k(u))/u^2 $. 
Then~\eqref{eq-helical} is equivalent to (see \cite[Section 1.1]{check}) 
\begin{equation}\label{eq-helical1}
    \left(\frac{(a-1)  \left(u^2 f''(u)+uf'(u)\right)}{2 (a+1) }\right)^2- \left(\frac{uf'(u)-u^2 f''(u)}{2}\right)^2=1.
\end{equation}
Thus the first fraction here vanishes nowhere (in particular, $a\ne 1$). We may assume that it is positive, otherwise change the sign of $f$ and $v$ in~\eqref{eq-helical-parametrization}, leading to just a rotation of the surface through the angle $\pi$ about the $x$-axis. Then 
the first fraction in~\eqref{eq-helical1} can be set to $\csc (2s(u))$ and the second one can be set to $\cot(2s(u))$ for some smooth function $s(u)$ with the values in $(0,\pi/2)$.  Therefore  by direct calculations (see  \cite[Section 1.2]{check}) %\mscomm{(add ref to an automatic checking!)}
\begin{align}\label{eq-helical2}
    f'(u) &= \frac{a  \cot s(u)+\tan s(u)}{(a-1) u},\\
    \label{eq-helical3}
    f''(u) &= \frac{ a\tan s(u) +\cot s(u)}{(a-1) u^2}.
\end{align}
Taking the derivative of~\eqref{eq-helical2} with respect to $u$ and combining it with \eqref{eq-helical3} we obtain 
\begin{equation}\label{eq-helical4}
    \frac{s'(u) (\tan s(u)-a \cot s(u))}{(a+1)}=\frac{1}{u}
\end{equation}
(see  \cite[Section 1.3]{check}).
In particular, $s'(u)\ne 0$ everywhere, hence $s(u)$ has an inverse function $u(s)$.
Integrating both sides of~\eqref{eq-helical4} and using that $s(u)$ assumes values in $(0,\pi/2)$, we get $u(s) = c_{2}\left(\cos s\sin^a s\right)^{-\frac{1}{a+1}}$ for some constant $c_{2}\ne 0$.

Denote $f(s):=f(u(s))$. By the chain rule, \eqref{eq-helical2}, and~\eqref{eq-helical4} we get 
\begin{equation}\label{eq-helical5}
     f'(s) = \left.\frac{f'(u)}{s'(u)}\right\rvert_{u=u(s)}= \frac{(\tan s+a \cot s) (\tan s-a \cot s)}{(a-1) (a+1)}.
\end{equation}
Integrating both sides of~\eqref{eq-helical5}, %with respect to $s$, 
we get $$f(s) =s+ \cot 2 s + \frac{a^2+1}{a^2-1}\csc 2 s+ c_{3}= \bluevar{s+ \frac{1}{a^2-1}{\tan s}+\frac{a^2}{a^2-1}\cot s+ c_{3}}$$ for some constant $c_{3}$ \cite[Section 1.4]{check}. The isotropic similarity $(x,y,z)\mapsto(x/c_{2},y/c_{2},z-c_{3})$ brings our surface to form~\eqref{helicalane-1} (up to renaming the parameter $s$ to $u$).
\end{proof}

%\subsection{Geometry of the surfaces}

Family~\eqref{helical-1}  is a family of helical isotropic minimal surfaces 
joining helicoid~\eqref{eq-helicoid} and logarithmoid~\eqref{eq-logarithmoid} (after appropriate scaling of the $z$-coordinate). It can be alternatively described as the family of the graphs of the harmonic functions $z=\mathrm{Re}(C\log (x+iy))$ with varying complex parameter $C$ (again, up to isotropic similarity).
\bluevar{
It is the associated family of the helicoid in isotropic geometry \cite[Sections~4.1 and~4.3(a)]{ds21b}; thus surfaces~\eqref{helical-1} can be called ``\emph{isotropic helicatenoids}''. Just like their Euclidean analogs, they are isometric to each other for different $c$, after scaling of the $z$-coordinate by $1/\sqrt{1+c^2}$. Notice that in isotropic geometry, the most natural notion of \emph{isometry} requires preservation of both the metric and the isotropic Gaussian curvature~\cite{isometric-isotropic}. This is indeed the case here by~\eqref{eq-K-helical}.
%for family~\eqref{helical-1} by~\cite[Theorem~2]{ds21b}.
%In addition, the family of surfaces $r_c$~\eqref{helical-1} represents a set of isometric simply isotropic minimal surfaces. Notably, $r_c$  is the associated family of the helicoid (see e.g., \cite[Section~4.1 and Example~4.3]{ds21b}).
}

\section{Translational surfaces}\label{sec-translational}

%\begin{definition} 
Now we present the main result \bluevar{of the paper}. If $\alpha(u)$ and $\beta(v)$ are two curves in $\mathbb{R}^3$,
then the surface $r(u,v)=\alpha(u)+\beta(v)$
is called the \emph{translational surface formed by $\alpha(u)$ and $\beta(v)$}.
%\end{definition}

%\mscomm{Take care of $a\mapsto 1/a$ !!!}
\begin{theorem}\label{thm-two-planar} (Fig.~\ref{fig:translation})
An admissible translational surface formed by a planar curve $\alpha$ and another curve $\beta$ has a constant ratio $a\ne 0$ of isotropic principal curvatures, if and only if it is isotropic similar to a subset of one of the surfaces 
\begin{align}\label{eq-two-isotropic}
    r_{a}(u,v) 
    &= \begin{bmatrix}
           u \\
            v \\
           v^2+a u^2
         \end{bmatrix},\\
    %(u,\quad v, \quad  v^2+a u^2),\\
    %r_{c,\theta}(u,v) 
    %&= (u,\quad v+ u\cot\theta, \quad  v^2+a u^2),\\
    \label{eq-isotropic+nonisotropic}
    r_b(u,v) 
    &= \begin{bmatrix}
          v+ b\cos v\\
            b\sin v+(b^2-1)\log\left\lvert b-\sin v\right\rvert+(1-b^2)u \\
           \exp u
         \end{bmatrix},
    %\left(v+ b\cos v,\quad  b\sin v+(b^2-1)\log\left\lvert b-\sin v\right\rvert+(1-b^2)u, \quad \exp u\right), 
    &\text{if }a &\neq 1,\\
    \label{eq-nonisotropic+nonisotropic}
    r(u,v) 
    &= \begin{bmatrix}
           u+v \\
           \log\left\lvert\cos u \right\rvert - \log\left\lvert\cos v\right\rvert \\
           u
         \end{bmatrix},
    %\left(u+v,\quad \log\left\lvert\cos u \right\rvert - \log\left\lvert\cos v\right\rvert, \quad u\right),
    &\text{if }a &=-1,
    %&= \left((c-1)v- b\cos v,\quad  - b\sin v+\log\left|sign(c)\sqrt{|c|}+\sqrt{|c-1|}\sin v\right|- u, \quad \exp u\right), \quad c \neq 1,
\end{align}
where
%$\theta$ is a constant such that $c(c-\csc^2\theta)\ge 0$ and $a:=2c-\csc^2\theta\pm 2\sqrt{c(c-\csc^2\theta)}$, $a:=2c-1\pm 2\sqrt{c(c-1)}$ and $b:=\sqrt{c/(c-1)}$ and $c:=4a/(a-1)^2$.
we denote
$b:=(a+1)/(a-1)$. In particular, $\beta$ must be a planar curve as well. \bluevar{In~\eqref{eq-two-isotropic}, the variables $u$, $v$ run through $\mathbb{R}$. 
In~\eqref{eq-isotropic+nonisotropic}, $u$ runs through $\mathbb{R}$ and $v$ runs through an interval where $\sin v\ne b$ (for $a<0$) and  $b\sin v\ne 1$ (for $a>0$). In \eqref{eq-nonisotropic+nonisotropic},
$(u, v)$ runs through a subdomain of $(-\pi/2,\pi/2)^2\setminus\{u+v=0\}$.}
\end{theorem}

\begin{figure}[htbp]
\begin{overpic}[width=0.2\textwidth]{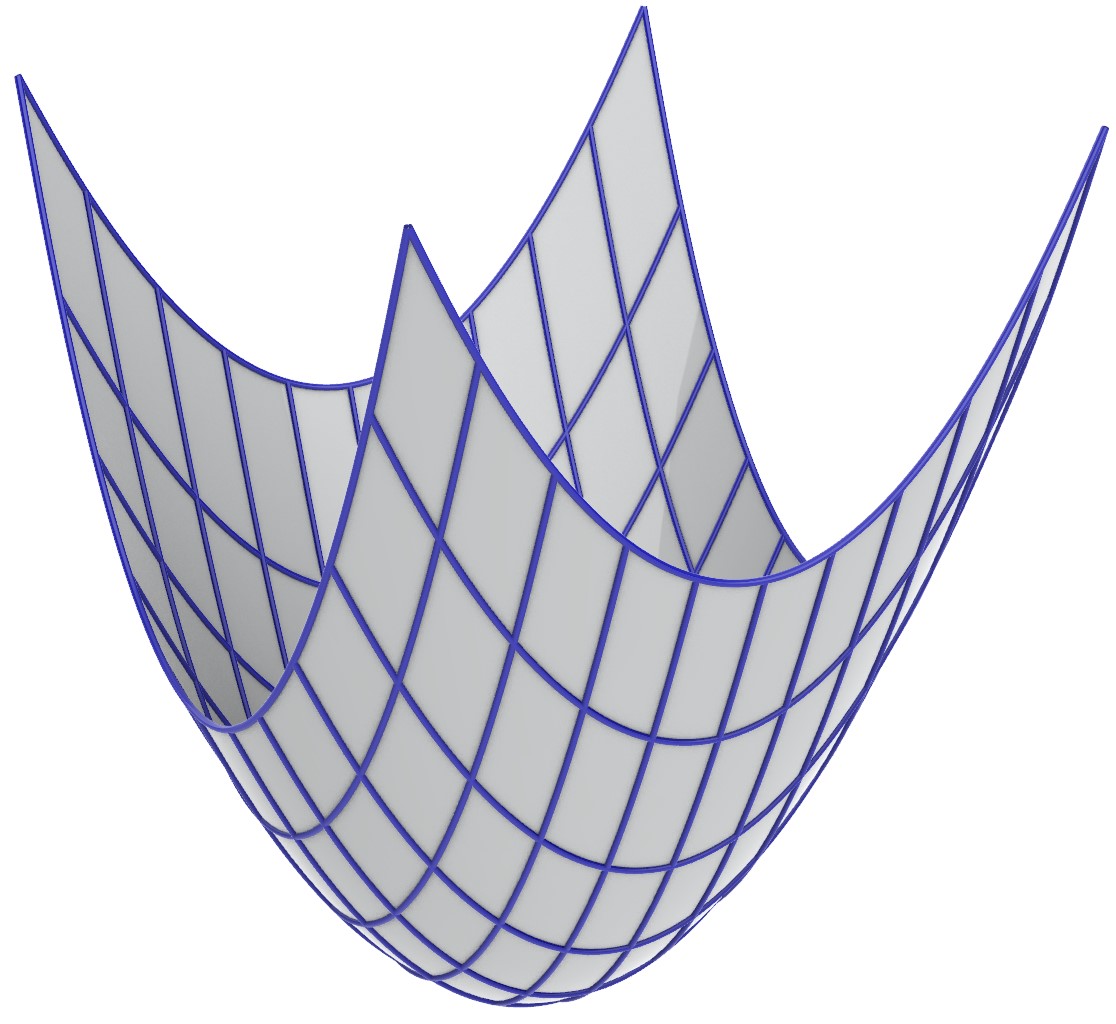}
    %\put(44,80){\contour{white}{$\widetilde{z}$}}
\end{overpic}
\hfill
\begin{overpic}[width=0.23\textwidth]{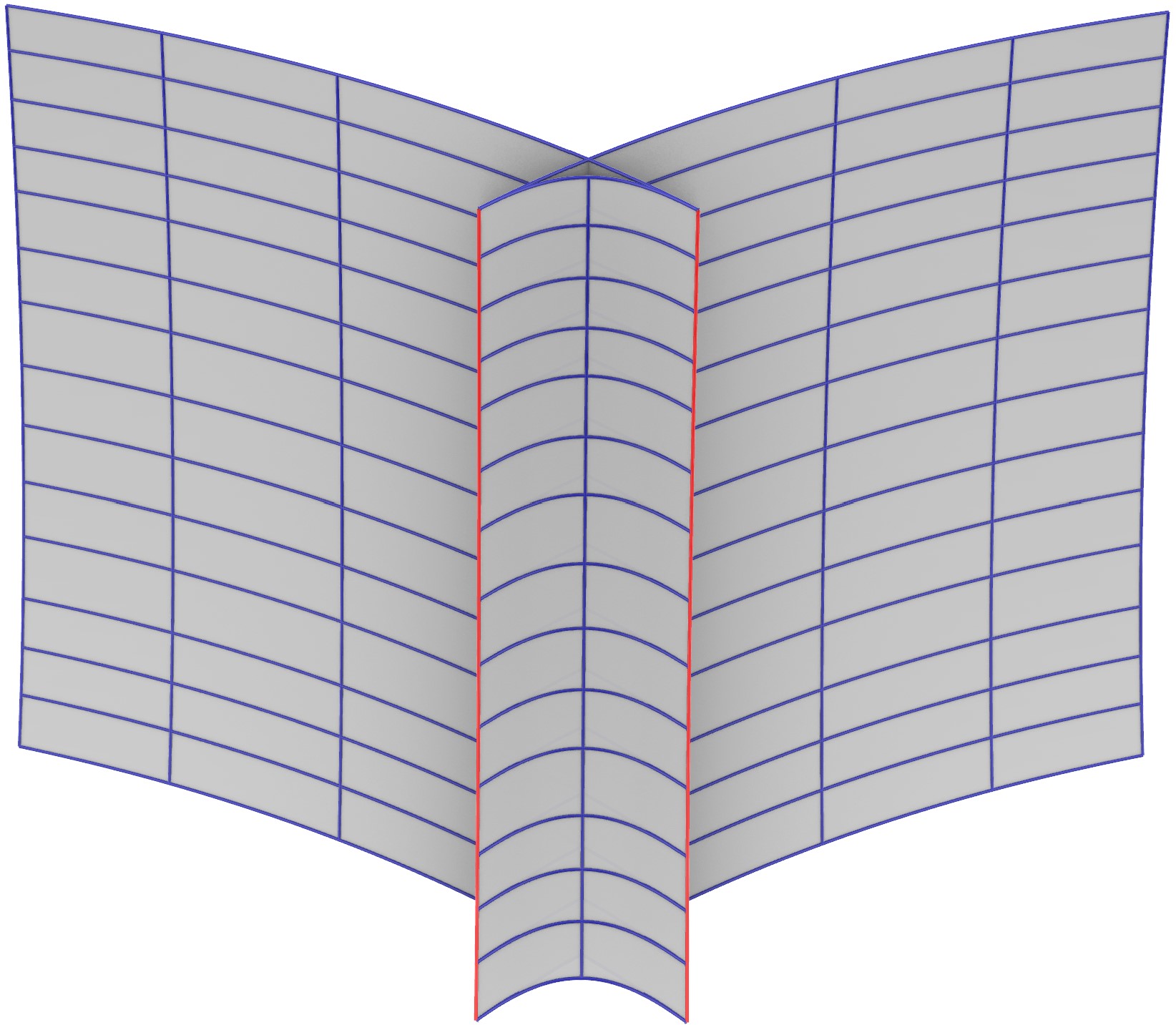}
    %\put(44,80){\contour{white}{$\widetilde{z}$}}
\end{overpic}
\hfill
\begin{overpic}[width=0.27\textwidth]{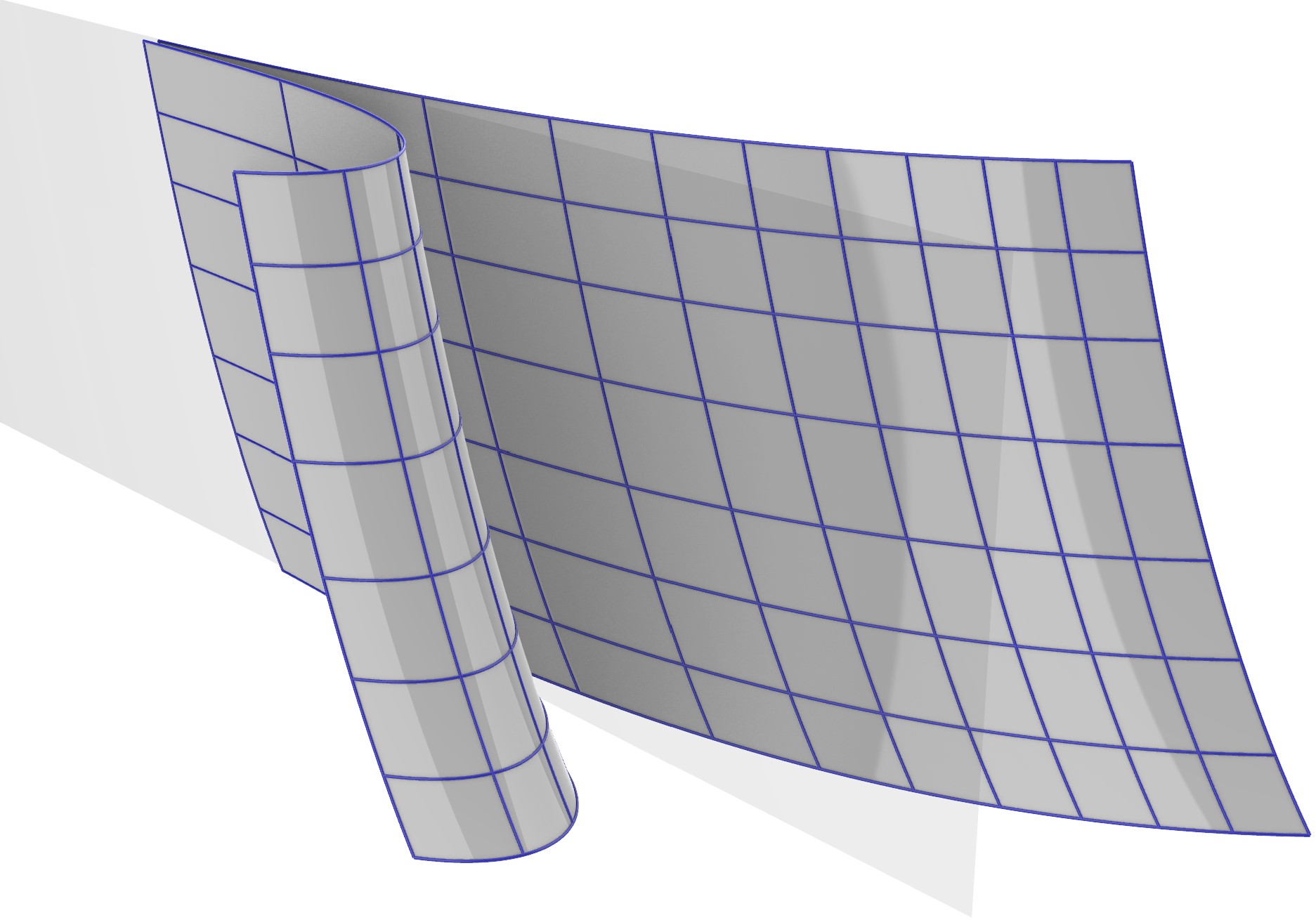}
    %\put(44,80){\contour{white}{$\widetilde{z}$}}
\end{overpic}
\hfill
\begin{overpic}[width=0.26\textwidth]{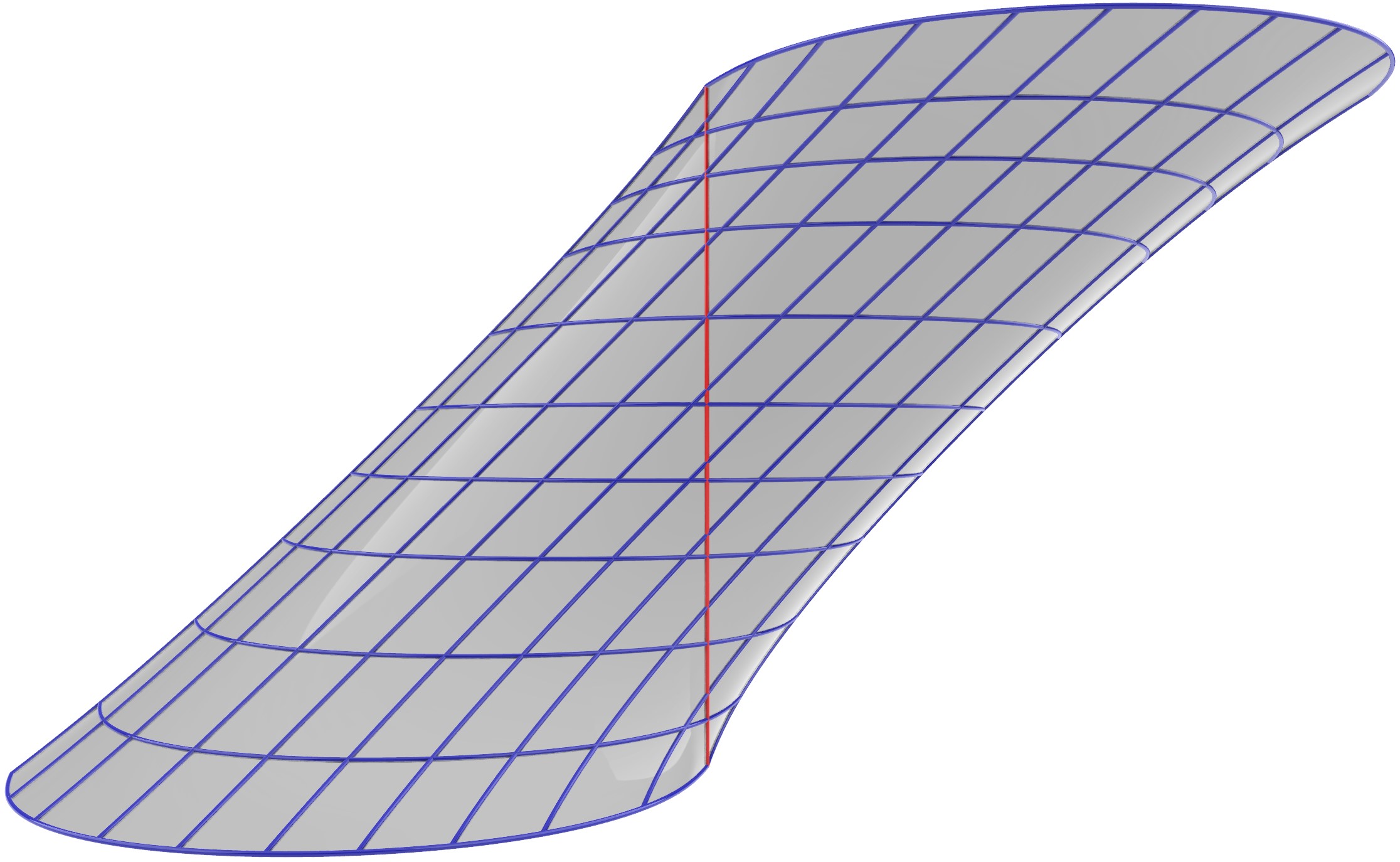}
    %\put(44,80){\contour{white}{$\widetilde{z}$}}
\end{overpic}
\hfill{}
\caption{Translational isotropic CRPC surfaces (from the left to the right):
surface~\eqref{eq-two-isotropic} for $a>0$; surface~\eqref{eq-isotropic+nonisotropic} for $a>0$ and $a<0$; surface~\eqref{eq-nonisotropic+nonisotropic}. The red curves  of surface~\eqref{eq-isotropic+nonisotropic} are the singular curves  $b\sin(v) = 1$. The red line of ~\eqref{eq-nonisotropic+nonisotropic} is the line $u + v = 0$ where the surface has isotropic tangent planes.
 The transparent plane is the asymptotic plane of surface~\eqref{eq-isotropic+nonisotropic} with $a < 0$ obtained in the limit $v\to\arcsin b$. 
% are depicted in red; it splits the surface into the parts $b\sin(v) > 1$ and $b\sin(v) < 1$. The singular line  $u + v = 0$ of the  surface  ~\eqref{eq-nonisotropic+nonisotropic} is depicted in red; it splits the surface into the parts $u +v < 0$ and $u +v > 0$.\\
%Bottom: the curves $u=\mathrm{const}$ and $v=\mathrm{const}$ on surfaces~\eqref{eq-isotropic+nonisotropic} and~\eqref{eq-nonisotropic+nonisotropic}. 
}
%\mscomm{Move the figure to the right place and rephrase the caption analogously to the previous figures, e.g., use 'singular curve' instead of 'the surface does not exist', use 'top', 'bottom', 'left', 'right', 'middle' instead of 'first figure', 'second figure' etc.}}
\label{fig:translation}
\end{figure}

%\begin{remark} Let $\alpha$ and $\beta$ be the two planar curves in the theorem.
%For surfaces~\eqref{eq-two-isotropic}, both curves $\alpha$ and $\beta$ are isotropic circles and their planes are isotropic. For surfaces~\eqref{eq-isotropic+nonisotropic}, the plane of exactly one of the curves $\alpha$ and $\beta$ is isotropic. 
%\end{remark}

%\textbf{Remark that the duality gives new types of surfaces --- conjugate geodesics on the dual surfaces. In case III, $K=-1/(\cosh y-\cos x)^2$ but it is not a Bianchi surface as the $x$- and $y$-directions are not asymptotic.}

\bluevar{Surfaces~\eqref{eq-two-isotropic}--\eqref{eq-nonisotropic+nonisotropic} for $a=-1$ can be viewed as the three \emph{isotropic Scherk minimal surfaces} \cite{Strubecker+1977+minimal,ds21b}.}

 %Denote $c:=(a+1)^2/4a$.
%Let $\alpha$ and $\beta$ be the two planar curves in the theorem. 
The theorem follows from Lemmas~\ref{l-case-i}--\ref{l-case-v}, where the following 5 cases are considered:
\begin{enumerate}
        \item $\alpha$ and $\beta$ are isotropic planar;
        %\mscomm{ADD A GEOMETRIC PROOF IN CASE I!}
        \item $\alpha$ is isotropic planar and $\beta$ is non-isotropic planar;
        \item $\alpha$ and $\beta$ are non-isotropic planar;
        \item $\alpha$ is isotropic planar and $\beta$ is non-planar;
        \item $\alpha$ is non-isotropic planar and $\beta$ is non-planar.
    \end{enumerate}
In our arguments, we use the expressions for $K$ and $H$ obtained by combining  \cite[Eq.~(4,8), (4,11), (4,18), (5,1)]{strubecker:1942}. (The convention for the sign of $H$ in  \cite{strubecker:1942} is different from ours %in  Lemma ~\ref{l-case-iii} and ~\ref{l-case-v} 
but this does not affect the equation $H^2/K = \frac{(a+1)^2}{4a}$ of CRPC surfaces.) 
\begin{lemma} \label{l-case-i}
    Under the assumptions of Theorem~\ref{thm-two-planar}, if $\alpha$ and $\beta$ are contained in isotropic planes, then the surface is isotropic similar to a subset of~\eqref{eq-two-isotropic}.
\end{lemma}

\begin{proof}
%Case I: $\alpha$ and $\beta$ are isotropic planar. Let us show that $\alpha$ and $\beta$ are isotropic circles and our surface is isotropic similar to a subset of the surface $r_{a}^\mathrm{I}(u,v)$.
Performing a rotation about a vertical axis, we can take the planes of \bluevar{$\alpha$ and}~$\beta$ to the planes \bluevar{$y=kx$ and $x=0$ for some $k\in\mathbb{R}$}. Since the surface is admissible, $\alpha$ and $\beta$ cannot have vertical tangents. Thus the curves can be parameterized as  $\alpha(u) = (u, ku, f(u))$ and $\beta(v) = (0, v, g(v))$ for % some $k\in\mathbb{R}$ and 
some functions $f(u)$ and $g(v)$ defined on some intervals. Then our surface is
    \begin{equation}
         r(u,v) = (u, ku + v, f(u)+g(v)).
    \end{equation}
Then the isotropic Gaussian and mean curvatures are (see \cite[Sections~4--5]{strubecker:1942} and \cite[Eq.~(3.2)]{aydin2017})
    \begin{equation}
        K = f''(u)g''(v), \quad H = \frac{(k^2+1)\,g''(v)+f''(u)}{2}.
    \end{equation}
    Since $\bluevar{K}\ne 0$, it follows that %$K\ne 0$, hence
    $f''(u)g''(v)\neq 0$ for all $(u,v)$ from the domain. Then the equation ${H^2}/{K} = (a+1)^2/(4a)$ is equivalent to the following differential equation:
    %\begin{equation}
    %    \left((1+\cot{\theta}^2)g''(v)+f''(u)\right)^2 = 4\cot{\theta}^2cf''(u)g''(v) \Longleftrightarrow 
    %\end{equation}
    \begin{equation}\label{eq4}
        \left(k^2+1+\frac{f''(u)}{g''(v)}\right)^2 = \frac{(a+1)^2}{a}\,\,\frac{f''(u)}{g''(v)}.
    \end{equation}
    By \eqref{eq4} we get ${f''(u)}/{g''(v)} = \mathrm{const}$. Then $f''(u) = \mathrm{const}$ and $g''(v) = \mathrm{const}$. 
    Thus $f(u)+g(v)$ is a polynomial of degree $2$. By Example~\ref{ex-paraboloid}, our surface is isotropic similar to a subset of~\eqref{eq-two-isotropic}.
    %, and our surface is the graph of a polynomial of degree $2$ in $x$ and $y$. %By an appropriate isotropic similarity, we can bring the graph to form~\eqref{eq-two-isotropic} for some $a\in\mathbb{R}$ with $|a|\ge 1$. An immediate computation shows that this $a$ actually equals to the ratio of the isotropic principal  curvatures of the resulting surface.
    %%%%
    %    Hence $g(v) = pv^2+qv+r$ for some $p,q,r$ with $p\ne 0$. Performing the isotropic similarity $(x,y,z)\mapsto (x,y,(z-qy-r)/p)$, we can take $\beta$ to the curve $\beta(v) = (0, v, v^2)$. Analogously, we  can take $\alpha$ to the curve $\alpha(u) =  (u, u\cot{\theta}, au^2)$ for some constant $a$.    Substituting  $g(v)=v^2$ and $f(u) = au^2$ into~\eqref{eq4}, we get $a = 2c-\csc^2\theta\pm 2\sqrt{c(c-\csc^2\theta)}$, which agrees with the notation in the theorem. In particular, $c(c-\csc^2\theta)\ge 0$ and the resulting surface is $r(u,v)=\alpha(u)+\beta(v)=r_{c,\theta}(u,v)$.
\end{proof}

One can see the same geometrically. The directions of the two \bluevar{parametric curves} $u=\mathrm{const}$ and $v=\mathrm{const}$ through each point of a translational surface are conjugate (by the \bluevar{property mentioned} %definition 
in Section~\ref{sec-preliminaries}). The top view of the \bluevar{parametric curves} consists of two families of parallel lines.  Therefore the top view of the two conjugate directions is the same everywhere. But a pair of conjugate directions and the ratio $a \neq 1$ of isotropic principal curvatures are enough to determine the two isotropic principal directions up to symmetry. 
By continuity, the top views of the isotropic principal directions are the same everywhere. Consider one \bluevar{parametric curve} $v=\mathrm{const}$ and a pair of \bluevar{parametric curves} $u=\mathrm{const}$. These three \bluevar{parametric curves} intersect at two points. The \bluevar{parametric curves} $u=\mathrm{const}$ are translations of each other along the \bluevar{parametric curve} $v=\mathrm{const}$. Hence the isotropic normal curvature of the former two \bluevar{parametric curves} is the same at these two points. Since the ratio of isotropic principle curvatures is constant, by the (isotropic) Euler formula (see \cite[Eq. ~(4,28)]{strubecker:1942}) it follows that the isotropic principle curvatures are also the same. Thus, again by the Euler formula, the isotropic normal curvature of the \bluevar{parametric curve} $v=\mathrm{const}$ is the same everywhere. Thus the latter \bluevar{parametric curve}, hence $\alpha(u)$, is a parabolic isotropic circle. Analogously, $\beta(v)$ is a parabolic isotropic circle, and our surface is a paraboloid.  

\begin{lemma} \label{l-case-ii}
    Under the assumptions of Theorem~\ref{thm-two-planar}, if $\alpha$ and $\beta$ are contained in an isotropic and a non-isotropic plane respectively, then the surface is isotropic similar to a subset of~\eqref{eq-isotropic+nonisotropic}.
\end{lemma}

\begin{proof}
%Case II: $\alpha$ is isotropic planar and $\beta$ is non-isotropic planar. Let us show that our surface is isotropic similar to a subset of the surface $r_{a}^\mathrm{II}(u,v)$.
Performing an isotropic similarity of the form $z\mapsto z+px+qy$ we can take the plane of $\beta$ to the plane $z=0$. After that, performing a rotation about a vertical axis, %or a translation along a horizontal vector,
we can take the plane of $\alpha$ to the plane $x=0$. Since the surface is admissible, $\alpha$ cannot have vertical tangents. Thus it can be parameterized as $\alpha(u) = (0, -u, f(u))$ for some function $f(u)$. Since the surface $r(u,v) = \alpha(u)+\beta(v)$ is admissible, $\beta$ cannot have tangents parallel to the $y$-axis. Thus it can be parameterized as $\beta(v) = (v, g(v), 0)$
for some function $g(v)$. Then the surface is
     \begin{equation}\label{eq-translation-our-surface}
             r(u,v) = (v, -u+g(v), f(u)).
     \end{equation}
Hence the isotropic Gaussian and mean curvatures are (see \cite[Sections~4--5]{strubecker:1942}) 
    \begin{equation}
        K = f'(u)f''(u)g''(v), \quad H = \frac{f'(u)g''(v)+\left(1+g'(v)^2\right)f''(u)}{2}.
    \end{equation}
Since $\bluevar{K}\ne 0$, it follows that $f'(u)f''(u)g''(v) \neq 0$. Therefore the equation ${H^2}/{K} = (a+1)^2/(4a)$ is equivalent to
    \begin{equation}\label{iso-noniso}
        \left((1+g'(v)^2)\frac{f''(u)}{f'(u)}+g''(v)\right)^2 = \frac{(a+1)^2}{a}\frac{f''(u)}{f'(u)}g''(v).
    \end{equation}
   By~\eqref{iso-noniso} we get ${f''(u)}/{f'(u)}=\mathrm{const}$. Hence $f(u) = pe^{\lambda u} + q$ for some constants $p,q,\lambda$, where $p,\lambda\neq 0$. %Performing the isotropic similarity $(x,y,z)\mapsto (x,y,(z-q)/p)$ we can take $\alpha$ to the curve $\alpha(u) = (0, -u, e^{\lambda u})$. 
   Substituting   $\lambda$ for $\frac{f''(u)}{f'(u)}$   in~\eqref{iso-noniso}, we get (see \cite[Section ~2.1]{check}) 
    \begin{equation}\label{iso-noniso2}
        g''(v) = \frac{\lambda}{4a}\left(a+1 \pm \sqrt{(a-1)^2-4ag'(v)^2}\right)^2.
    \end{equation}
    To solve the resulting ODE, introduce the new variable $p=g'(v)$. Since $g''(v)\ne 0$, it follows that the function $g'(v)$ has a smooth inverse $v(p)$ and there is a well-defined composition $g(p):=g(v(p))$. Substituting $g'(v) = p$ into \eqref{iso-noniso2} and using $g''(v) = \frac{dp}{dv} =1/\frac{dv}{dp}= p/\frac{dg}{dp}$ for $p\ne 0$, we obtain
    \begin{align}\label{iso-noniso3}
        \lambda \frac{dg}{dp} &= \frac{4ap}{\left(a+1 \pm \sqrt{(a-1)^2-4ap^2}\right)^2}, \\
        %\text{and}\\
        \lambda \frac{dv}{dp} &= \frac{4a}{\left(a+1 \pm \sqrt{(a-1)^2-4ap^2}\right)^2}.
        %pp'_{g} = \frac{\lambda}{4a}\left((a+1) \pm \sqrt{(a-1)^2-4ap^2}\right)^2.
    \end{align}
    Clearly, $a\ne 1$. Integrating, we obtain (see \cite[Section ~2.2]{check}) 
    %\textbf{(Replace $\pm\arccos(...)$ by $\arctan(...)\mp\frac{\pi}{2}\mathrm{sgn}(ap)$!)}
 %   \begin{equation}
 %        \lambda g(p) =-\log{\left\lvert a+1\pm \sqrt{(a-1)^2-4ap^2}\right\rvert} - \frac{a+1}{a+1\pm \sqrt{(a-1)^2-4ap^2}} + C_{1} \\ 
 %   \end{equation}
%$\lambda v(p) =$
  %  \begin{equation}
       %  \begin{cases}
        % \frac{(a-1)^2}{4a}\arctan\left(\frac{p^2(a+1)\pm \sqrt{(a-1)^2-4ap^2}}{p\left(a+1 \mp \sqrt{(a-1)^2-4ap^2}\right)}\mp\frac{\pi}{2}\sign (ap)\right) + \frac{p(a+1)}{a+1\pm \sqrt{(a-1)^2-4ap^2}} +C_{2}, \quad p\neq 0\\
        % C_{2}, \quad p = 0,
        % \end{cases} 
   % \end{equation}
    \begin{align}\label{iso-noniso4}
         \lambda g(p) &=-\log{\left\lvert a+1\pm \sqrt{(a-1)^2-4ap^2}\right\rvert} - \frac{a+1}{a+1\pm \sqrt{(a-1)^2-4ap^2}} + C_{1}, \\ 
         \notag
         \lambda v(p) &=
         \frac{(a-1)^2}{4a} \left[\arctan p \mp \arctan \left(\frac{(a+1) p}{\sqrt{(a-1)^2-4 a p^2}}\right)\right]+\\
         &\hspace{5.2cm}+\frac{(a+1) p}{a+1 \pm \sqrt{(a-1)^2-4 a p^2}}+C_{2}
         \label{iso-noniso4a}
         %%%%%% OLD FORMULA: %%%%%%%%
         %\begin{cases}
         %\frac{(a-1)^2}{4a}\left(\arctan\left(\frac{p^2(a+1)\pm \sqrt{(a-1)^2-4ap^2}}{p\left(a+1 \mp \sqrt{(a-1)^2-4ap^2}\right)}\right) + C(\sign p)%\mp\frac{\pi}{2}\sign (ap)
         %\right) +\\ \qquad \qquad \qquad \qquad \qquad \qquad  +\frac{p(a+1)}{a+1\pm \sqrt{(a-1)^2-4ap^2}} +C_{2}, &\text{if } p\neq 0\\
         %C_{2},  &\text{if } p = 0,
         %\end{cases} 
         %%%%% ANCIENT INCORRECT FORMULA %%%%
        %\lambda v(p) &= \pm\frac{(a-1)^2}{4a}\arccos\frac{4ap}{(a-1)\left(a+1 \pm \sqrt{(a-1)^2-4ap^2}\right)} + \frac{p(a+1)}{a+1\pm \sqrt{(a-1)^2-4ap^2}} +C_{2}
    \end{align}
       %$$v(p) =-\frac{(a-1)^2}{2a} \arctan \left(\frac{2 a p}{a-1\mp \sqrt{(a-1)^2-4 a p^2}}\right)+\frac{(a+1) p}{a+1 \pm \sqrt{(a-1)^2-4 a p^2}} $$
    for some constants $C_1$, $C_2$.
    %, and a piecewise-constant function $$C(\sign p):=\frac{\pi}{2}\sign (1-a\mp (1+a))\sign p,$$ which ensures smooth gluing of solutions. 
    Denote by $w$ the expression in square brackets in \eqref{iso-noniso4a} plus $\pm\sign{(a-1)}\cdot\pi/2$.
    %the first summand in \eqref{iso-noniso4a} without the factor $(a-1)^2/4a$ plus $\pm\pi/2$ \mscomm{Change!!!}.
    Passing to the new variable $w$ and using the notation $b := (a+1)/(a-1)$,
    we get (see  \cite[Section ~2.3]{check}) 
    \begin{align}\label{iso-noniso5}
         \lambda g(w) &= \frac{(b^2-1)\log\lvert b-\sin w\rvert + b\sin w+C'_{1}}{(b^2-1)}, \\
         \lambda v(w) &= \frac{w + b\cos{w} + C'_{2}}{(b^2-1)}
    \end{align}
    for some other constants $C_1'$ and $C_2'$. Performing the isotropic similarity $$(x,y,z)\mapsto \left(\lambda(b^2-1)x-C'_{1}, \lambda(b^2-1)y-C'_{2},\frac{z-q}{p}\right)$$ 
    %we bring the curves $\alpha$ and $\beta$ to the form $\left(-\frac{C'_{1}}{2},(1-b^2)\lambda u-\frac{C'_{2}}{2},e^{\lambda u}\right)$ and $$\left(w + b\cos{w}+\frac{C'_{1}}{2},(b^2-1)\log{\lvert b-\sin w\lvert}+ b\sin w+\frac{C'_{2}}{2},0\right)$$ respectively.
    and renaming the parameters $u$ and $w$ to $u/\lambda$ and $v$, we bring~\eqref{eq-translation-our-surface} to form~\eqref{eq-isotropic+nonisotropic}. 
    \end{proof}

\begin{lemma} \label{l-case-iii}
    Under the assumptions of Theorem~\ref{thm-two-planar}, if $\alpha$ and $\beta$ are contained in non-isotropic planes, then the surface is isotropic similar to a subset of~\eqref{eq-nonisotropic+nonisotropic}.
\end{lemma}

\begin{proof}
%Case III: $\alpha$ and $\beta$ are non-isotropic planar. Let us show that this is only possible for $a=-1$, i.e., for an \emph{isotropic minimal} surface, and moreover our surface is isotropic similar to a subset of the surface $r^\mathrm{III}(u,v)$.
Similarly to the previous lemma, performing an isotropic similarity of the form $z\mapsto px+qy+rz$ and a rotation about a vertical axis we take the planes of $\alpha$  and $\beta$  to the planes $z = x$ and  $z=0$ respectively. The tangent to $\alpha$ cannot be perpendicular to the $x$-axis at each point,
because otherwise $\alpha$ is a straight line and $a=0$, contradicting to the assumptions of the theorem. Thus $\alpha'(u)\not\perp Ox$ at some point $u$. By continuity, the same is true in an interval around $u$. In what follows switch to an inclusion-maximal interval $(u_1,u_2)$ with this property. Notice that then each endpoint $u_k$ is either an endpoint of the domain of $\alpha(u)$ or there exist finite $$\lim_{u\to u_k}\alpha(u) \quad \text{and} \quad \lim_{u\to u_k}\alpha'(u)/\lvert\alpha'(u)\rvert\perp Ox.$$ On the interval $(u_1,u_2)$, the curve $\alpha$ can be parameterized as $\alpha(u) = (u, f(u), u)$ for some smooth function $f(u)$. Analogously, on a suitable interval $(v_1,v_2)$ the curve $\beta$ can be parameterized as $\beta(v) = (v, g(v), 0)$ for some smooth $g(v)$. Therefore a part of our surface can be parameterized as
     \begin{equation}\label{eq-part-surface}
             r(u,v) = (u+v, f(u)+g(v), u).
     \end{equation}
Thus the isotropic Gaussian and mean curvatures are (see \cite[Sections~4--5]{strubecker:1942} and \cite[Eq.~(6.3), (6.8)]{Aydin2020})
    \begin{equation}
        K = \frac{f''(u)g''(v)}{(f'(u)-g'(v))^4}, 
  \quad H = \frac{(1+f'(u)^2)g''(v)+(1+g'(v)^2)f''(u)}{2\lvert f'(u)-g'(v)\rvert^3}.
    \end{equation}
Here  $f''(u)g''(v) \neq 0$ and $f'(u)-g'(v) \neq 0$ because $\bluevar{K}\ne 0$ and the surface is admissible. Hence the equation ${H^2}/{K} = (a+1)^2/(4a)$ is equivalent to
    \begin{equation}\label{noniso-noniso}
        \left((1+f'(u)^2)g''(v)+(1+g'(v)^2)f''(u)\right)^2 = \frac{(a+1)^2}{a}f''(u)g''(v)(f'(u)-g'(v))^2.
    \end{equation}
    
First let us solve the equation in the case when $a=-1$ (this was done in \cite{Milin,Strubecker+1977+minimal}). In this case, $$f''(u)/(1+f'(u)^2) = -g''(v)/(1+g'(v)^2) = c$$ for some constant $c\ne0$. Hence 
$$f(u) = \frac{1}{c}\log\lvert\cos(cu+c_{1})\rvert + c_{2} \quad \text{and} \quad g(v) = -\frac{1}{c}\log\lvert\cos(cv+c_{3})\rvert + c_{4}$$ 
for some constants $c_{1}, c_{2}, c_{3},$ and $c_{4}$. Changing the parameters $(u,v)$ to $(u-c_{1},v-c_3)/c$ and performing the isotropic similarity $(x,y,z)\mapsto (\bluevar{c}x+c_{1}+c_3,\bluevar{c(}y-c_{2}-c_4\bluevar{),c}z+c_{1})$ we bring~\eqref{eq-part-surface} to form \eqref{eq-nonisotropic+nonisotropic}. Notice that there are no points $u_1,u_2\in\mathbb{R}$ with finite $\lim_{u\to u_k}f(u)$ and $\lim_{u\to u_k}1/\bluevar{\sqrt{2+f'(u)^2}}=0$; hence the above maximal intervals $(u_1,u_2)$ and $(v_1,v_2)$ coincide with the domains of $\alpha(u)$ and $\beta(v)$, and~\eqref{eq-part-surface} actually coincides with the whole given surface.

Now let us prove that for $a\neq -1$ equation~\eqref{noniso-noniso} has no solutions with $f''(u)g''(v) \ne 0$.
%$a \neq 1$ otherwise, our surface is \textbf{isotropic sphere} and $a\neq -1$. 
The equation is equivalent to
\begin{equation}\label{noniso-noniso2}
        (1+f'(u)^2)g''(v)+(1+g'(v)^2)f''(u)  \pm(a+1)\sqrt{\frac{f''(u)g''(v)}{a}}(f'(u)-g'(v)) = 0.
\end{equation}
We may assume that here we have a plus sign and $g''(v)>0$, otherwise replace $(f(u),g(v))$ by $\sign(g''(v))(f(\pm u),g(\pm v))$.

If $a=1$ then~\eqref{noniso-noniso2} is equivalent to $$\left(\sqrt{f''(u)/g''(v)}+ f'(u)\right)^2+\left(g'(v)\sqrt{f''(u)/g''(v)}- 1\right)^2 = 0.$$ 
Hence $f'(u)=-1/g'(v)$ is constant. Therefore $f''(u) = 0$, a contradiction.

Assume further $a\ne \pm 1$. %Fix $v$ and denote $r:=g'(v)$, $s:=g''(v)$. %Now we can take $f'(u) = Y$ and $f''(u)s/a = X^2$, since $f''(u)s/a \ge 0$ \eqref{noniso-noniso}. 
For fixed $v$, a solution $f(u)$ of~\eqref{noniso-noniso2} gives a regular curve in the plane with the coordinates $$(X,Y) := \left(\sqrt{f''(u)/a}, f'(u)\right)$$ because $f''(u)\neq 0$. By~\eqref{noniso-noniso2}, the curve is contained in the conic 
       \begin{equation}\label{noniso-noniso1}
       \begin{split}
           a(1+g'(v)^2)X^2 + (a+1)\sqrt{g''(v)}XY+ &g''(v)Y^2-(a+1)\sqrt{g''(v)}g'(v)X+ g''(v)=0.
        \end{split}
       \end{equation}
The conic is irreducible because the determinant of its matrix is  $$-(a-1)^2(g'(v)^2+1)g''(v)^2/4\neq 0$$ for $a\neq 1$ (see~\cite[Section 2.4]{check}).  Since such irreducible conics~\eqref{noniso-noniso1} for distinct $v$ have a common curve, they actually do not depend on $v$. Thus the ratio of the coefficients at $X$ and $XY$ is constant.  Since $a\ne -1$, it follows that $g'(v) = \mathrm{const}$ and $g''(v) = 0$, a contradiction. Therefore there are no solutions for $a\ne -1$. % except~\eqref{eq-nonisotropic+nonisotropic}. 
\end{proof}

\begin{lemma} \label{l-case-iv}
    There are no admissible translational surfaces with constant nonzero ratio  of isotropic principal curvatures formed by an isotropic planar curve $\alpha$ and a nonplanar curve $\beta$.
\end{lemma}

\begin{proof}
%Case IV: $\alpha$ is isotropic planar and $\beta$ is non-planar. Let us prove that in this case there are no surfaces in question.
% Let $\alpha$ be an isotropic planar curve and $\beta$ be a nonplanar curve in the theorem. 
Assume the converse. Performing a rotation about a vertical axis, %or a translation along a horizontal vector,
we can take the plane of $\alpha$ to the plane $x=0$. Since the surface is admissible, $\alpha$ cannot have vertical tangents. Thus it can be parameterized as $\alpha(u) = (0, u, f(u))$ for some function $f(u)$. Since the surface $r(u,v) = \alpha(u)+\beta(v)$ is admissible, $\beta$ cannot have tangents perpendicular to the $x$-axis. Thus it can be parameterized as $\beta(v) = (v, g(v), h(v))$ for some functions $g(v)$ and $h(v)$. %Also $g'''(v)h''(v)-h'''(v)g''(v) \neq 0$, because $\beta$ is a space curve. 
Then the surface is
     \begin{equation}
             r(u,v) = (v, u+g(v), f(u)+h(v)).
     \end{equation}
Therefore the isotropic Gaussian and mean curvatures are (see \cite[Sections~4--5]{strubecker:1942})
    \begin{equation}
        K =-f''(u)(f'(u)g''(v)-h''(v)), 
  %\quad H = \frac{f'(u)g''(v)-h''(v)-(1+g'(v)^2)f''(u)}{2}.  %\left(f'(u)g''(v)-h''(v)-(1+g'(v)^2)f''(u)\right)/.
    \end{equation}
      \begin{equation}
        %K =-f''(u)(f'(u)g''(v)-h''(v)), 
  H = \frac{f'(u)g''(v)-h''(v)-(1+g'(v)^2)f''(u)}{2}.  %\left(f'(u)g''(v)-h''(v)-(1+g'(v)^2)f''(u)\right)/.
    \end{equation}
Here  $f''(u)\neq 0$ %and $f'(u)g''(v)-h''(v) \neq 0$ 
because $\bluevar{K}\neq 0$. Thus the equation ${H^2}/{K} = (a+1)^2/(4a)$ is equivalent to
    \begin{equation}\label{iso-space}
        \left(1+g'(v)^2+\frac{h''(v)-f'(u)g''(v)}{f''(u)}\right)^2 = \frac{(a+1)^2}{a}\frac{(h''(v)-f'(u)g''(v))}{f''(u)}.
    \end{equation}
%Denote by $t(u,v)$ the $(h''(v)-f'(u)g''(v))/f''(u)$. 
Here the right side (without the factor $(a+1)^2/a$) does not depend on $u$ because equation~\eqref{iso-space} is quadratic in it with the coefficients not depending on $u$. 
Differentiating the right side with respect to $u$, we get (see~\cite[Section 2.5]{check})
\begin{equation}\label{iso-space1}
    %t(u,v)_{u} =
    \frac{g''(v)(f'''(u)f'(u)-f''(u)^2)-h''(v)f'''(u)}{f''(u)^2} = 0.
\end{equation}
If there exists $u$ such that $f'''(u)\neq 0$, then  $h''(v) = \mathrm{const} \cdot g''(v)$, otherwise $g''(v)=0$ identically. In both cases $g'''(v)h''(v)-h'''(v)g''(v) = 0$ identically. Hence $\beta$ is a planar curve, a contradiction.
\end{proof}

\begin{lemma}\label{l-case-v}
There are no admissible translational surfaces with constant nonzero ratio  of isotropic principal curvatures formed by a non-isotropic planar curve $\alpha$ and a nonplanar curve $\beta$.
\end{lemma}

\begin{proof}
%Case V: $\alpha$ is non-isotropic planar and $\beta$ is non-planar. Let us prove that in this case there are no surfaces in question.
%Let  $\beta$ be a nonplanar curve and $\alpha$ be a non-isotropic planar curve  in the theorem. 
Assume the converse. Performing an isotropic similarity of the form $z\mapsto px+qy+z$ we can take the plane of $\alpha$ to the plane  $z=0$. Performing an appropriate rotation with a vertical axis and restricting to sufficiently small parts of our curves,
we may assume that the tangents to $\alpha$ and $\beta$ are not perpendicular to the $x$-axis at each point. Then the curves can be parameterized as $\alpha(u) = (u, f(u), 0)$ and $\beta(v) = (v, g(v), h(v))$ for some smooth functions $f(u)$, $g(v)$, and $h(v)$. 
%The tangent to $\alpha$ cannot be perpendicular to the $x$-axis at each point, because otherwise $\alpha$ is a straight line and $a=0$, contradicting to the assumptions of the theorem. Thus $\alpha'(u)\not\perp Ox$ at some point $u$. By the continuity, the same is true in an interval around $u$. In what follows restrict the curve $\alpha$ to this interval so that it can be parameterized as $\alpha(u) = (u, f(u), 0)$ for some smooth function $f(u)$. Analogously, restrict the curve $\beta$ to an appropriate interval so that it can be parameterized as $\beta(v) = (v, g(v), h(v))$ for some smooth $g(v)$ and $h(v)$ \mscomm{--- not always possible, see case II!}. 
Therefore a part of our surface can be parameterized as
     \begin{equation}
             r(u,v) = (u+v, f(u)+g(v), h(v)).
     \end{equation}
Thus the isotropic Gaussian and mean curvatures are (see \cite[Sections~4--5]{strubecker:1942} and \cite[Eq. ~(6.3), ~(6.8)]{Aydin2020})
    \begin{equation}
        K = \frac{f''(u)h'(v)\left(g''(v)h'(v)-h''(v)(g'(v)-f'(u))\right)}{(f'(u)-g'(v))^4}, 
        \end{equation}
    \begin{equation}
   H = -\frac{(1+g'(v)^2)f''(u)h'(v)+(1+f'(u)^2)\left(g''(v)h'(v)-h''(v)(g'(v)-f'(u))\right)}{2\lvert f'(u)-g'(v)\rvert^3}. \end{equation}
Here  $f''(u),h'(v),f'(u)-g'(v) \neq 0$ because $\bluevar{K}\ne 0$. % and the surface is admissible. Since $K \neq 0$, then $h'(v) \neq 0$ 
Hence the equation ${H^2}/{K} = (a+1)^2/(4a)$ is equivalent to
    %\begin{equation}\label{noniso-space}
    %\begin{split}
  $$ a\left((g'(v)^2+1)f''(u)+(f'(u)^2+1) \left(g''(v)-\frac{h''(v)}{h'(v)}(g'(v)-f'(u))\right)\right)^2-$$
   \begin{equation}\label{noniso-space}
    -(a+1)^2  (f'(u)-g'(v))^2 f''(u)\left(g''(v)-\frac{h''(v)}{h'(v)}(g'(v)-f'(u))\right)=0.
    %\end{split}
    \end{equation}
% where $s(v) := h''(v)/h'(v)$. 
 
  %First, let us consider the case when $a = 1$, then  $f''(u)(g''(v)-s(u) (g'(v)-f'(u))) > 0$, otherwise left side of ~\eqref{noniso-space} is positive  because $f'(u)\neq g'(v)$. Since ~\eqref{noniso-space1} is equivalent to $(f''(u)-(g''(v)-s(u) (g'(v)-f'(u)))f'(u)^2)^2+(f''(u)g'(v)^2-(g''(v)-s(u) (g'(v)-f'(u))))^2+2(g'(v)f''(u)+(g''(v)-s(u) (g'(v)-f'(u)))f'(u))^2+2f''(u)(g''(v)-s(u) (g'(v)-f'(u)))(g'(v)f'(u)+1)^2 =0$, hence $g'(v)f'(u)+1=0$. Therefore $f'(u) = \mathrm{const}$ and $f''(u) = 0$, a contradiction.
 %Now assume $a \ne  1$. 
 Fix a value of $v$. A solution $f(u)$ of~\eqref{noniso-space} gives a regular curve in the plane with the coordinates $(X,Y) := \left(f''(u), f'(u)\right)$ 
 because $f''(u)\neq 0$. The curve is disjoint with the line $X=0$. By~\eqref{noniso-space}, the curve is contained in the algebraic curve
\begin{equation}\label{noniso-space1}
a\left((g'(v)^2+1)X +(Y^2+1) L(Y)\right)^2-(a+1)^2 X (Y-g'(v))^2L(Y)=0,
\end{equation}
where $$L(Y) := g''(v)-h''(v) (g'(v)-Y)/h'(v).$$ The expression $L(Y)$ is not a zero polynomial in $Y$, because otherwise~\eqref{noniso-space1} reduces to $X=0$, whereas the above regular curve is disjoint with the line $X=0$.

%First, let us consider the case when $a = -1$. In this case~\eqref{noniso-space1} is equivalent to
%the curve is contained in the cubic plane curve 
%\begin{equation}\label{a=-1}
%(g'(v)^2+1)X +  Y^3 s(v)+Y^2 (g''(v)-g'(v)s(v))+ Y s(v) + g''(v)-g'(v)s(v)=0.
%\end{equation}
%The cubic curve~\eqref{a=-1} is irreducible because the left side is linear with respect to $X$. Since for distinct $v$ curves~\eqref{noniso-space1} contain the common curve $(f'(u),f''(u))$, they do not depend on $v$. Thus the ratio of the coefficients at $Y^2$ and $Y^3$ in~\eqref{a=-1} is constant. On each open interval where $s(v) \neq 0$, the ratio $(g''(v)-s(v) g'(v))/s(v)$ is constant. Therefore $(g'(v)/h'(v))' = (1/h'(v))'$, hence $g'(v) = c h'(v) +1$ for some constant $c$. At each point where $s(v) =0$, we have $h''(v) = 0$. In both cases $g'''(v)h''(v)-g''(v)h'''(v)=0$ identically. Thus the curve $\beta$ is planar, a contradiction.

Let us prove that the algebraic curve~\eqref{noniso-space1} is irreducible
(where an irreducible curve of multiplicity two is also viewed as irreducible). Indeed, otherwise the left side equals $$a(g'(v)^2+1)^2(X-P_{1}(Y))(X-P_{2}(Y))$$ for some complex polynomials $P_{1}(Y)$ and  $P_{2}(Y)$. Consider~\eqref{noniso-space1} as a quadratic equation in~$X$. Then its discriminant $D(Y)=a^2(g'(v)^2+1)^4(P_{1}(Y)-P_{2}(Y))^2$ is the square of a  polynomial in $Y$. %The discriminant equals
A direct computation gives (see ~\cite[Section ~2.6]{check})
\begin{multline}\label{discriminant}
D(Y)=(a+1)^2 (Y-g'(v))^2L(Y)^2 \cdot\\
 %\left((a-1)^2 g'(v)^2-4 a-2 (a+1)^2 g'(v)Y+((a-1)^2-4 a g'(v)^2) Y^2\right).
\cdot\left((a-1)^2 g'(v)^2-4 a-2 (a+1)^2 g'(v)Y+\left((a-1)^2-4 a g'(v)^2\right) Y^2\right).
 \end{multline}
 %\textbf{Rewriting of the next paragraph required!}
Assume $a\ne -1$, otherwise the left side of~\eqref{noniso-space1} is the square of a linear in $X$, hence irreducible, polynomial. All factors of $D(Y)$ except the last one are complete squares and not zero polynomials in $Y$. Hence the last factor, which is at most quadratic in $Y$, is a square of a polynomial in $Y$. Hence its discriminant (see~\cite[Section 2.7]{check}) $16 (a-1)^2 a (g'(v)^2+1)^2$  vanishes. Since  $a\ne 0$, we get $a= 1$. Then $D(Y) \le 0$ with the equality only for a finite number of real values of $Y$. Therefore~\eqref{noniso-space1} has only a finite number of real points $(X,Y)$ and cannot contain a regular curve. This contradiction proves that~\eqref{noniso-space1} is irreducible.

%Let us assume that $a \ne -1$, then $(g''(v)+s(v)(Y-g'(v))$ or non-square factor is identically zero. The non-square factor can not be identically zero since coefficients in front of $Y: -2(a+1)^2g'(v)$, and $Y^2:(a-1)^2-4ag'(v)^2$ could not be zero at the same time since $a\neq  1$. Therefore $(g''(v)+s(v)(Y-g'(v))=0$, hence $s(v) = 0$ and $g''(v) = 0$.

%Hence the curve ~\eqref{noniso-space1} is irreducible when $a\ne -1$, and it is the square of an irreducible curve if $a=-1$. 
Since irreducible curves~\eqref{noniso-space1} for distinct $v$ contain the same regular curve, they must all coincide. Thus the ratio of the free term and the coefficient at $Y$ in~\eqref{noniso-space1} is constant \bluevar{(including the case when one of the coefficients or both vanishes)}. Hence the ratio of the two coefficients of the \bluevar{linear} polynomial 
$L(Y)$ is constant. Thus $p(h'(v)g''(v)-h''(v) g'(v))-qh''(v)=0$ for some constants $p$ and $q$ not vanishing simultaneously. Therefore $((pg'(v)+q)/h'(v))' = 0$, hence $pg'(v)+q=r h'(v)$ and 
$pg(v)+qv+rh(v)+s=0$ for some constants $r$ and~$s$. Thus the curve $\beta(v)=(v,g(v),h(v))$ is planar, a contradiction.
\end{proof}

\section{Dual-translational surfaces}

\subsection{Isotropic metric duality} 

The principle of duality is a crucial concept in projective geometry. For example, in projective 3-space, points are dual to planes and vice versa, straight lines are dual to straight lines, and inclusions are reversed by the duality.

In contrast to Euclidean geometry, isotropic geometry possesses a \emph{metric duality}. It is defined as the polarity with respect to the unit isotropic sphere, which maps a point $P=\left(p_1, p_2, p_3\right)$ to the non-isotropic plane $P^*$ with the equation $z=p_1 x+p_2 y-p_3$, and vice versa. For two %non-parallel 
points $P$ and $Q$ %$=\left(q_1, q_2, q_3\right)$ 
at isotropic distance $d$, the dual planes $P^*$ and $Q^*$ intersect at the isotropic angle  $d$. %(between the planes). 
%$\phi=\sqrt{\left(p_1-q_1\right)^2+\left(p_2-q_2\right)^2}$. 
The latter is defined as the difference between the slopes of the two lines obtained in a section of $P^*$ and $Q^*$ by an isotropic plane orthogonal to the line $P^*\cap Q^*$.

%\mscomm{The isotropic angle between planes has not been defined in the paper! It seems better to discuss only notions required for the description of CRPC surfaces dual to translational ones. Do we really need duals of lines, contact elements etc?}

The following properties of the metric duality are straightforward. Parallel points, defined as points having the same top view, are dual to parallel planes. Two non-parallel lines in a non-isotropic plane are dual to two non-parallel lines in a non-isotropic plane. Two parallel lines in a non-isotropic plane are dual to two non-parallel lines in an isotropic plane.

The \emph{dual} $\Phi^*$ of an admissible surface $\Phi$ is the set of points dual to the tangent planes of $\Phi$. If $\Phi$ is the graph of a smooth function $f$, then the tangent plane at the point $(x_0, y_0, f(x_0, y_0))$ is $$z = xf_x(x_0, y_0) + yf_y(x_0, y_0) -(x_0f_x(x_0, y_0) + y_0f_y(x_0, y_0)-f(x_0, y_0)). $$ 
Hence $\Phi^*$ is parameterized by 
\begin{equation}
x^*(x,y)=f_x(x, y),\quad  y^*(x,y)=f_y(x, y), \quad z^*(x,y)=x f_x+y f_y-f .
\end{equation}
If $\Phi$ has parametric form $ (x(u,v), y(u,v), z(u,v))$, then  $\Phi^*$ is parameterized by
\begin{equation}\label{parametric:form}
x^*(u,v) =\frac{y_u z_v-y_v z_u}{x_v y_u-x_u y_v},\quad  y^*(u,v)=\frac{x_u z_v-x_v z_u}{x_u y_v-x_v y_u}, \quad z^*(u,v)=x x^*+y y^*-z.
\end{equation}
 
%The pair of a plane $\pi$ and a point $p$ on it is called \emph{contact element}. The contact elements of a surface are the points on it and their corresponding tangent planes. The integral curves of the field of principal directions are called \emph{isotropic principal curvature lines}. Contact elements along the isotropic principal curvature lines of $\Phi$ and $\Phi^*$ correspond in the duality. However, 
It is important to note that  $\Phi^*$ may have singularities that correspond to parabolic points of $\Phi$, where $K=0$, and doubly-tangent planes. This duality relationship is reflected in the following expressions that relate the isotropic curvatures of dual surfaces, as shown in \cite{Strubecker78}: $H^*=H/K$ and $K^*=1/K$.

Thus the dual of an isotropic CRPC surface is again an isotropic CRPC \bluevar{surface} because  $(H^*)^2/ K^* = H^2/K$. The classes of rotational, parabolic rotational, ruled, and helical CRPC surfaces are clearly invariant under the duality. Each surface~\eqref{eq-logarithmoid}, \eqref{eq-helicoid}, \eqref{eq-spiral}, \eqref{helicalane-1}, and~\eqref{helical-1} is isotropic similar to its dual. %The dual surface of ~\eqref{paraboloid} is $z = x^2+y^2/a$.
Two surfaces~\eqref{eq-rotational} are isotropic similar to the duals of each other. 

More properties of the metric duality can be found in \cite{sachs}. 

\subsection{Dual-translational isotropic CRPC surfaces}

For the translational surfaces, the duality leads to a new type of surfaces: the ones with a conjugate net of isotropic geodesics. A curve on a surface is an \emph{isotropic geodesic} if its top view is a straight line segment. Two families of curves form a \emph{conjugate net} if any two curves from distinct families intersect and their directions at the intersection point are conjugate (see the definition in Section~\ref{sec-preliminaries}).

\begin{prop}\label{thm-two-planar-dual} (See Fig.~\ref{fig:dualtranslation})
The dual surfaces of~\eqref{eq-isotropic+nonisotropic} and~\eqref{eq-nonisotropic+nonisotropic} are up to \bluevar{general} isotropic similarity 
%~\eqref{eq-isotropic+nonisotropic-dual} and ~\eqref{eq-nonisotropic+nonisotropic-dual} 
respectively
\begin{align}\label{eq-isotropic+nonisotropic-dual}
    r_b^{*}(u,v) 
    &= \exp u 
         \begin{bmatrix}
          \frac{\cos v}{ b-\sin v}\\
             1\\
           \frac{b-b^3+v \cos v }{\left(b^2-1\right) (b-\sin v)}-\log\left\lvert b-\sin v\right\rvert+u
         \end{bmatrix},\\
    \label{eq-nonisotropic+nonisotropic-dual}
    r^{*}(u,v) 
    &= \frac{1}{\tan u+\tan v} %\cos u \cos v\csc (u+v)
    \begin{bmatrix}
           %\frac{\cos u \sin v}{ \sin(u+v)} \\
           %-\frac{\cos u \cos v}{\sin (u+v)}\\
           %\frac{\cos u \cos v(\log \lvert\frac{\cos v}{\cos u}\rvert-u \tan u+v \tan v)}{\sin (u+v)}
          \tan v \\
           1\\
          \log \lvert\frac{\cos v}{\cos u}\rvert-u \tan u+v \tan v
         \end{bmatrix}.
\end{align}
%They are CRPC surfaces which 
They have a constant ratio (equal to $(b-1)/(b+1)$ and $-1$ respectively) of isotropic principal curvatures and
possess a conjugate net of isotropic geodesics. \bluevar{The domains of maps~\eqref{eq-isotropic+nonisotropic-dual} and~\eqref{eq-nonisotropic+nonisotropic-dual} are subsets of the domains of~\eqref{eq-isotropic+nonisotropic} and~\eqref{eq-nonisotropic+nonisotropic}, where the maps are injective; see Theorem~\ref{thm-two-planar}.}
\end{prop}

% \begin{remark} The domain of maps~\eqref{eq-isotropic+nonisotropic-dual} and~\eqref{eq-nonisotropic+nonisotropic-dual} is a subset of the domain of~\eqref{eq-isotropic+nonisotropic} and~\eqref{eq-nonisotropic+nonisotropic}, where the maps are injective; see Remark~\ref{rem-domain}.
% %In~\eqref{eq-isotropic+nonisotropic-dual}, $u$ runs through $\mathbb{R}$ and $v$ runs through an interval where $\sin v\ne b$ (for $a<0$) and  $b\sin v\ne 1$ (for $a>0$). In \eqref{eq-nonisotropic+nonisotropic-dual}, $(u, v)$ runs through $(-\pi/2,\pi/2)^2\setminus\{u+v=0\}$. 
% \end{remark}

\begin{figure}[htbp]
\hfill
\begin{overpic}[width=0.34\textwidth]{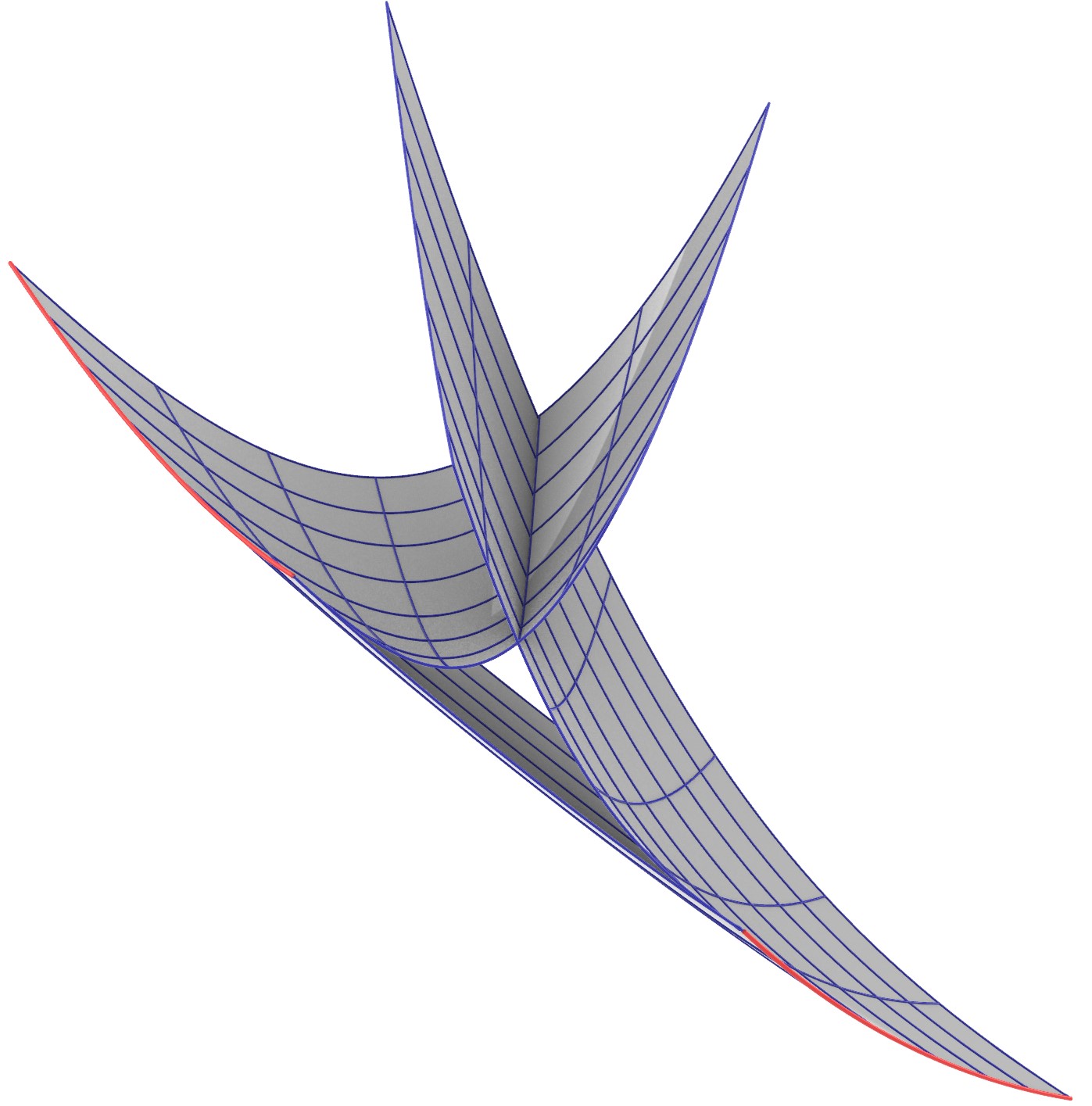}
    %\put(44,80){\contour{white}{$\widetilde{z}$}}
\end{overpic}
\hfill
\begin{overpic}[width=0.27\textwidth]{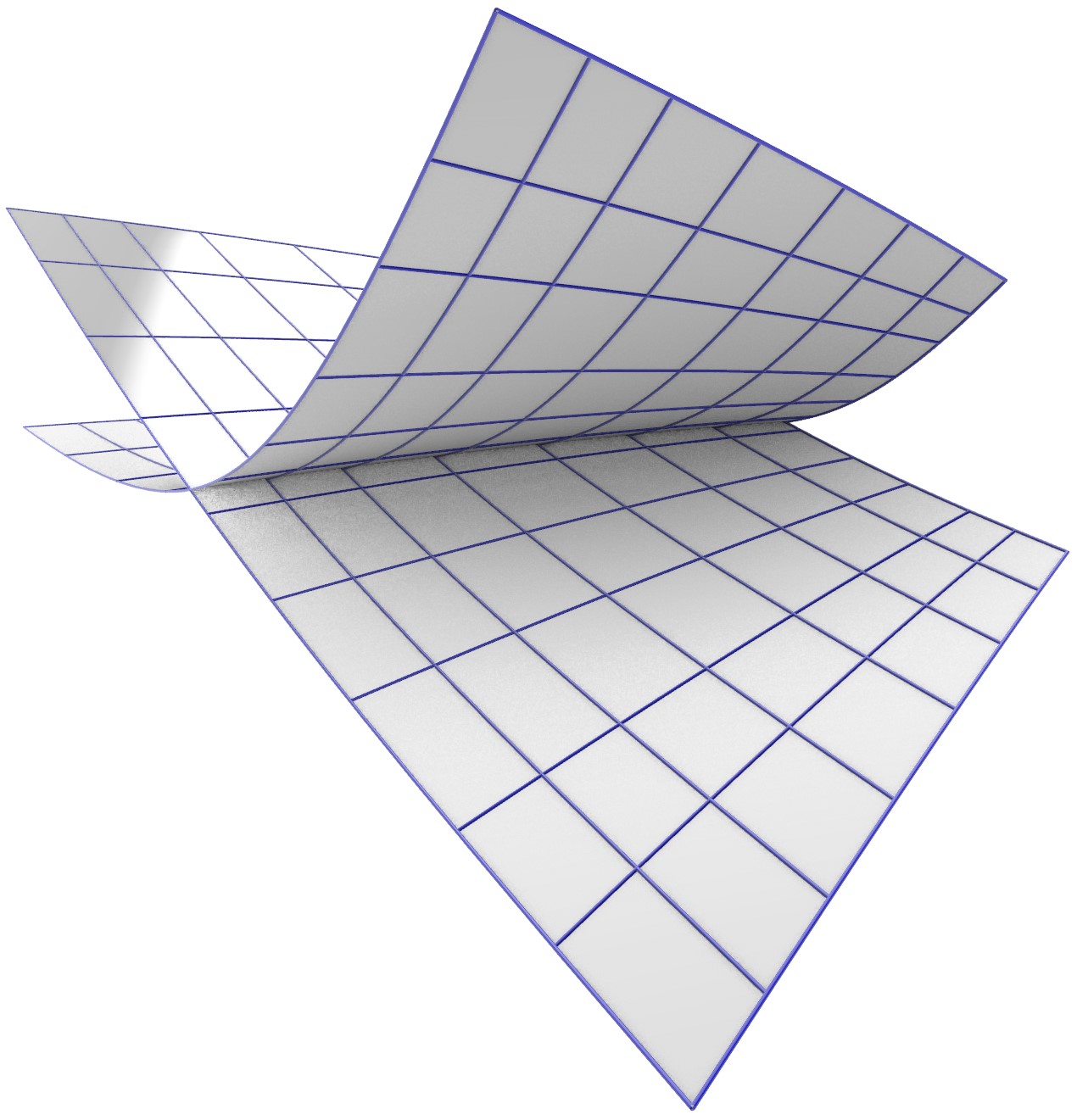}
    %\put(44,80){\contour{white}{$\widetilde{z}$}}
\end{overpic}
\hfill
\begin{overpic}[width=0.37\textwidth]{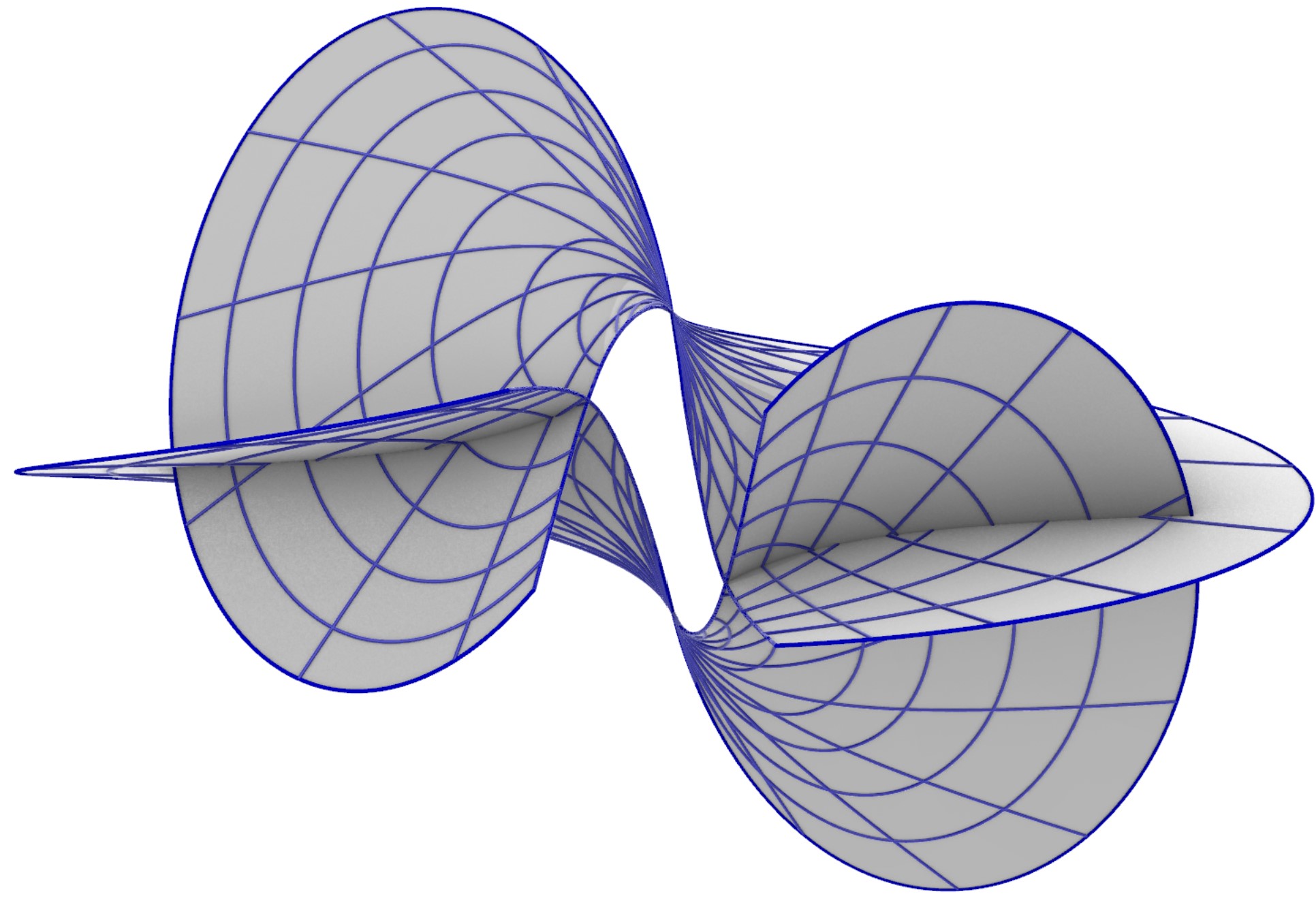}
    %\put(44,80){\contour{white}{$\widetilde{z}$}}
\end{overpic}
\hfill{}
\caption{Duals of the translational isotropic CRPC surfaces (from the left to the right):
surfaces~\eqref{eq-isotropic+nonisotropic-dual} for $a>0$ and $a<0$; surface~\eqref{eq-nonisotropic+nonisotropic-dual}. The red curves are the singular curves  $b\sin v = 1$  of surface~\eqref{eq-isotropic+nonisotropic-dual} with $a >0$. 
% are depicted in red; it splits the surface into the parts $b\sin(v) > 1$ and $b\sin(v) < 1$. The singular line  $u + v = 0$ of the  surface  ~\eqref{eq-nonisotropic+nonisotropic} is depicted in red; it splits the surface into the parts $u +v < 0$ and $u +v > 0$.\\
%Bottom: the curves $u=\mathrm{const}$ and $v=\mathrm{const}$ on surfaces~\eqref{eq-isotropic+nonisotropic} and~\eqref{eq-nonisotropic+nonisotropic}. 
}
%\mscomm{Move the figure to the right place and rephrase the caption analogously to the previous figures, e.g., use 'singular curve' instead of 'the surface does not exist', use 'top', 'bottom', 'left', 'right', 'middle' instead of 'first figure', 'second figure' etc.}}
\label{fig:dualtranslation}
\end{figure}

%Geometric properties of dual-translational surfaces - conjugate geodesics?}

The proposition is proved by direct calculation with the help of~\eqref{parametric:form} (see  \cite[Section~4]{check}). 
We still have to show that duals to translational surfaces possess a conjugate net of isotropic geodesics.  
At all points of a curve $u=\mathrm{const}$ on a translational surface, the tangents to the curves $v=\mathrm{const}$ are parallel
and form a general cylinder.
Hence all tangent planes along the curve $u=\mathrm{const}$ are parallel to one line. Then the duals of those planes form a section of the dual surface by an isotropic plane, which is an isotropic
geodesic. The same is true for the tangent planes along each curve $v=\mathrm{const}$. Since curves $u=\mathrm{const}$ and $v=\mathrm{const}$ form a conjugate net on the translational
surface and any projective duality maps conjugate
tangents to conjugate tangents, the metric dual to a translational
surface possesses a conjugate net of geodesics. 

\begin{remark}
 Surfaces with a conjugate net of Euclidean geodesics have been determined by A.~Voss \cite{voss1888}. The conjugate
net of geodesics is remarkably preserved by a one-parameter
family of isometric deformations. The conjugate nets
of geodesics are reciprocal parallel to the asymptotic nets
of surfaces with constant negative Gaussian curvature. Discrete
versions of Voss nets are quad meshes with planar faces that are flexible when the faces are rigid and the
edges act as hinges. We refer the reader to R.~Sauer \cite{sauer-1970}.
Analogous properties hold for the isotropic counterparts of Voss surfaces if one defines
an isometric deformation in isotropic space as one which
preserves the top view and isotropic Gauss curvature~\cite{isometric-isotropic}. We
will report on this and related topics in a separate
publication.   
\end{remark}

\section{Open problems}

Following the general philosophy discussed in Section~\ref{sec-intro}, one can try to apply the methods developed for the classification of helical and translational isotropic CRPC surfaces in Sections~\ref{sec-helical}--\ref{sec-translational} to their analogs in Euclidean geometry; cf.~\cite{helicalcrpc}. 
The case of translational surfaces generated by two spatial curves remains open in both geometries. 

It is natural to extend the search for CRPC surfaces to other Cayley--Klein geometries \bluevar{such as Galilean or Minkowski geometry, and Cayley--Klein vector spaces \cite{giering,S05}}. 
The transition from Euclidean to pseudo-Euclidean (Minkowski) geometry is not expected to lead to 
significant differences, but more case distinctions. It is relative differential geometry with respect 
to a hyperboloid. We may return to this topic in future research if it appears to be rewarding.

\appendix

\section{\bluevar{The } singularities in the top view}\label{sec-appendix}

For the classification of ruled CRPC surfaces (see Section~\ref{sec-ruled}), we need the following two well-known lemmas from singularity theory, which we could not find a good reference for.

%We are going to consider the top view of the rulings, hence we need the following well-known lemma, which we could not find a good reference for.

\begin{lemma} \label{l-envelope} If $R_t$ is an analytic family of lines in the plane, then for all $t$ in some interval $I$ the lines $R_t$ satisfy one of the following conditions
\begin{itemize}
    \item[(i)] they have a common point;
    \item[(ii)] they are parallel;
    \item[(iii)] they \bluevar{are tangent to} one regular curve (\emph{envelope}) %at pairwise distinct points, and $E$ is 
    forming a part of the boundary of the union $\bigcup_{t\in I}R_t$.
\end{itemize}
%If lines $L_t\subset\mathbb{R}^2$, where $t\in [0,1]$, form an analytic family, then either they all have a common point, or all are parallel, or there is an interval $I\subset [0,1]$ and a regular planar curve $E$ (\emph{envelope}) such that for all $t\in I$ the lines $L_t$ are tangent to $E$ at pairwise distinct points, and $E$ is contained in the boundary of the union $\bigcup_{t\in I}L_t$.
\end{lemma}

\begin{proof} %\mscomm{Adjust the following argument to our setup!}
Assume without loss of generality that all $R_t$ are not parallel to the $y$ axis
for $t$ in some interval $I_1$.
Let $L(x,y,t):=y+a(t)x+b(t)=0$ be the equation of the line $R_t$. % (the argument does not actually rely on the particular form of the equation). 
%Let the family parameter $t$ run through an interval $J$. 
For each $n=1,2,\dots$ consider the subset $\Sigma_n$ of $\mathbb{R}^2\times I_1$ given by
$L(x,y,t)=\tfrac{\partial}{\partial t}L(x,y,t)=\dots=\tfrac{\partial^n}{\partial t^n}L(x,y,t)=0$ but $\tfrac{\partial^{n+1}}{\partial t^{n+1}}L(x,y,t)\ne 0$, and also the subset $\Sigma_\infty$ given by $\tfrac{\partial^n}{\partial t^n}L(x,y,t)=0$ for all $n=0,1,\dots $ . %The projection of the union $\Sigma$ of all these subsets onto the $xy$-plane is the envelope $E$. 
By the analyticity, we may restrict to a subsegment $I_2\subset I_1$ such that %each of these subsets is either empty or a collection of the whole connected components of $\Sigma$. Moreover, we may assume that the 
no connected component of $\Sigma_n$ is contained in a plane of the form $\mathbb{R}^2\times \{t\}$. %, and their projections are disjoint. 
Consider the following 3 cases.

Case (i): $\Sigma_\infty\ne\emptyset$. Then take a point $(x,y,t)\in\Sigma_\infty$. By the analyticity $L(x,y,t)=0$ for all $t$. Hence $(x,y)$ is a point common to all $R_t$. 

Case (ii): $\Sigma_\infty=\Sigma_n=\emptyset$ for all $n$. Then the system $L(x,y,t)=\tfrac{\partial}{\partial t}L(x,y,t)=0$ has no solutions, i.e., $a'(t)=0$ and $b'(t)\ne 0$ everywhere. Hence $a(t)=\mathrm{const}$, $b(t)$ is monotone, and all $R_t$ are parallel.

Case (iii): $\Sigma_n\ne\emptyset$ for some $n$. Then take a point $(x,y,t)\in\Sigma_n$ and consider the map $G(x,y,t):= \left(L(x,y,t),\tfrac{\partial^n}{\partial t^n}L(x,y,t)\right)$; this generalizes the argument from~\cite[Section~5.21]{bruce-giblin:1984}, where $n=1$. 
Let us show that the top view (projection to the $xy$-plane along the $t$-axis) of $G^{-1}(0,0)$ is the required curve (envelope). Since $\tfrac{\partial L}{\partial t}=0$, $\tfrac{\partial^{n+1} L}{\partial t^{n+1}}\ne 0$, and $\tfrac{\partial L}{\partial y}\ne 0$, it follows that the differential $dG$ is surjective. Then by the Implicit Function Theorem, the intersection of $G^{-1}(0,0)$ with a neighborhood of $(x,y,t)$ is a regular analytic curve with the tangential direction $(dx,dy,dt)$  given by $$\tfrac{\partial L}{\partial t}dt+\tfrac{\partial L}{\partial x}dx+\tfrac{\partial L}{\partial y}dy=\tfrac{\partial^{n+1} L}{\partial t^{n+1}}dt+\tfrac{\partial^{n+1} L}{\partial t^{n}\partial x}dx+\tfrac{\partial^{n+1} L}{\partial t^{n}\partial y}dy=0;$$ cf.~\cite[Proof of Proposition~5.25]{bruce-giblin:1984}. Since $\tfrac{\partial L}{\partial t}=0$ and $\tfrac{\partial^{n+1}L}{\partial t^{n+1}}\ne 0$, it follows that the top view  of the curve is a regular analytic curve tangent to $L_t$. Since no component of $\Sigma_n$ is contained in the plane $\mathbb{R}^2\times \{t\}$, it follows that the top view is tangent to $L_t$ for each  $t$ %and has no other common points with $L_t$, for 
in a sufficiently small interval $I_3\subset I_2$.

The resulting envelope cannot be a straight line, otherwise, we have case (i). Then it has a non-inflection point. Then for a sufficiently small $I\subset I_3$, a part of the curve is contained in the boundary of the union $\bigcup_{t\in I}R_t$.
\end{proof}

%We are going to extend our surface so that it acquires non-admissible points. Therefore we need to consider top views of curves with vertical tangents. These top views are not smooth in general, but the following standard result allows to define their tangents.

\begin{lemma} \label{l-tangents} Assume that an analytic curve in $\mathbb{R}^3$
has a vertical tangent at a point $O$ and does not coincide with the tangent.
Then the limit of the top view of the tangent at a point $P$ tending to $O$ coincides with the top view of the limit of the osculating plane at a point $P$  tending to $O$. In particular, both limits exist.
\end{lemma}

\begin{proof} Let $r(t)$ be the arclength parametrization of the curve with $r(0)=O$. Since the curve is not a vertical line, $r'(t)\ne\mathrm{const}$. Hence by the analyticity $r'(t)=r'(0)+t^na(t)$ for some integer $n\ge 1$ and a real analytic vector-function $a(t)$ such that $a(0)\ne 0$. %In what follows restrict to an interval where $a(t)\ne 0$. 
Since $\lvert r'(t)\rvert=\mathrm{const}$, it follows that $a(0)\perp r'(0)$, i.e. $a(0)$ is horizontal.

The top view of the tangent at $P=r(t)$, where $t\ne 0$ is small enough, is parallel to the top view of $a(t)$ because $r'(0)$ is vertical. Hence the limit of the former is parallel to $a(0)$. 

The osculating plane at $P=r(t)$, where $t\ne 0$ is small enough, is parallel to the linear span of $r'(t)=r'(0)+t^na(t)$ and $r''(t)/t^{n-1}=na(t)+ta'(t)$. As $t\to 0$, the
latter tends to the span of $r'(0)$ and $na(0)$, with the top view again parallel to $a(0)$. %\cite{Allen2011}
\end{proof}

\bluevar{%\begin{remark}
    An exciting study of simply isotropic minimal surfaces containing isotropic lines can be found in the recent work by Akamine and Fujino \cite{AF22}, which may provide additional insights related to the concepts discussed above.
%\end{remark}
}

%{
%\footnotesize

\subsection{Acknowledgements}

This research has been supported by KAUST baseline funding (grant BAS/1/1679-01-01). \bluevar{For useful discussions and remarks, our sincere thanks go to Rafael López and Udo Hertrich-Jeromin.}   

\subsection{Statements and Declarations}

\textbf{Conflict of interests.} The authors have no relevant financial or non-financial interests to disclose.
%}

%\mscomm{Incorporate the references into the Bibfile!}
%\newpage
\bibliographystyle{unsrt}
\bibliography{sn-bibliography}

%\begin{thebibliography}{AA99}

%\end{thebibliography}
\end{document}